\documentclass[12pt,final]{article}

\usepackage{amsmath}
\usepackage{amsthm}
\usepackage{amssymb}
\usepackage{amsfonts}
\usepackage{latexsym}               
\usepackage{amsgen}
\usepackage[mathscr]{eucal}
\usepackage{mathrsfs}
\usepackage{enumerate}

\usepackage{mathptmx} 
\usepackage{fancybox}
\usepackage{color}
\usepackage{graphicx}

\usepackage{showkeys}
\usepackage{bm}

\newtheorem{theorem}{Theorem}
\newtheorem{proposition}[theorem]{Proposition}
\newtheorem{lemma}[theorem]{Lemma}
\newtheorem{corollary}[theorem]{Corollary}

\theoremstyle{definition}

\newtheorem{remark}[theorem]{Remark}

\numberwithin{equation}{section}
\numberwithin{theorem}{section}

\def\R{\mathbb{R}}
\def\ep{\varepsilon}
\def\vp{\varphi}

\def\cale{\mathcal{E}}
\def\calp{\mathcal{P}}
\def\calh{\mathcal{H}}
\def\dels{\delta}
\def\ub{u_{\text{\rm b}}}

\def\rt{\tilde{\rho}}
\def\ut{\tilde{u}}

\def\tt{\tilde{\theta}}

\def\div{\mathop{\rm div}}
\def\pd{\partial}

\def\dt{\partial_t}

\def\lpeasp#1#2{L^{#1}_{\text{\rm e},#2}}

\def\lteasp#1{\lpeasp{2}{#1}}

\def\lp#1#2{\| #2 \|_{L^{#1}}}
\def\lpo#1#2#3{\| #2 \|_{L^{#1}(#3)}}
\def\lt#1{\|#1\|}
\def\li#1{\lp{\infty}{#1}}

\def\hs#1#2{\| #2 \|_{H^{#1}}}

\def\ho#1{\hs{1}{#1}}

\begin{document}

\abovedisplayskip=8pt plus 2pt minus 4pt
\belowdisplayskip=\abovedisplayskip


\thispagestyle{plain}

\title{Stationary flows for viscous
heat-conductive fluid \\
in a perturbed half-space} 

\author{%
{\large\sc Mingjie Li${}^1$,}
{\large\sc Masahiro Suzuki${}^2$}
{\normalsize and}
{\large\sc Katherine Zhiyuan Zhang${}^3$}
}

\date{%
\normalsize
${}^1$%
College of Science, Minzu University of China,
\\
 Beijing 100081, P. R. China
\\
${}^2$%
Department of Computer Science and Engineering, 
Nagoya Institute of Technology,
\\
Gokiso-cho, Showa-ku, Nagoya, 466-8555, Japan
\\ [7pt]
${}^3$%
Department of Mathematics, Northeastern University, \\
Boston, MA 02115, USA
}

\maketitle

\begin{abstract}
We consider the non-isentropic compressible Navier--Stokes equation in a perturbed half-space with an outflow boundary condition as well as the supersonic condition. This equation models a compressible viscous, heat-conductive, and Newtonian polytropic fluid. We show the unique existence of stationary solutions for the perturbed half-space. The stationary solution constructed in this paper depends on all directions and has multidirectional flow. We also prove the asymptotic stability of this stationary solution.
\end{abstract}


\begin{description}
\item[{\it Keywords:}]
compressible Navier--Stokes equation, stationary solution, unique existence, asymptotic stability, multidirectional flow

\item[{\it 2020 Mathematics Subject Classification:}]
35B35; 
35B40; 
76N06; 
76N10; 
\end{description}


\newpage

\section{Introduction}

The motion of a compressible viscous, heat-conductive, and Newtonian polytropic fluid is governed by the following non-isentropic compressible Navier--Stokes system:
\begin{equation}\label{ns1}
\left\{\begin{array}{lll}
\rho_t+\textrm{div}(\rho {u})=0, \\ \vspace{3mm}
\displaystyle (\rho{u})_t+\textrm{div}(\rho{u}\otimes{u})+\nabla
P=\mu \Delta u+(\mu+\lambda)\nabla(\textrm{div}{u}),\\ \vspace{3mm}
\displaystyle (\rho E)_t+\textrm{div}\big({u}(\rho{E}+P)\big)
\displaystyle= \Delta(\kappa\theta+\frac{1}{2}\mu| u|^2)+\textrm{div} \big(\mu
 u\cdot\nabla{u}+\lambda u\textrm{div} u\big).
\end{array}\right.
\end{equation}
 Here $t \geq 0$ is time, $ x =(x_{1},x')=(x_{1},x_{2},x_{3}) \in\mathbb{ R}^3$ is the spatial coordinate, and the unknowns $\rho, {u}=(u_1,u_2,u_3)^{tr}\in \mathbb{ R}^3$, and $\theta$ denote the fluid density, velocity, and  absolute temperature, respectively. 
 The pressure $P=P(\rho, \theta)$ is a function of $\rho$ and $\theta$, and the specific total energy ${E}$ is given by 
 \begin{equation*}
 {E}:=(e+|{u}|^2/2)
 \end{equation*}
  with $e$ being the internal energy and $|{u}|^2/2$ the kinetic energy. The constant viscosity coefficients of the flow $\mu$ and $\lambda$  satisfy the physical restrictions 
  \begin{equation*}
   \mu > 0,\quad   2\mu + 3\lambda \geq 0.
  \end{equation*}
The parameter $\kappa >0$ is the ratio of the heat conductivity coefficient over the heat capacity.
The first equation is the conservation of mass, the second one
is the equation of momentum in which the pressure gradient
and viscosity as well as the convection effect are taken into account,
and the third equation is the balance of energy.
We study the ideal polytropic fluids
so that $P$ and $e$ are given by the state equations
\begin{equation*}
P(\rho, \theta) = R\rho\theta,\quad e = c_v\theta,
\end{equation*}
where $R$ is a positive constant and $c_v:=\frac{R}{\gamma-1}$ is the specific heat at constant volume with $\gamma>1$ being the adiabatic exponent. We study an initial--boundary value problem of \eqref{ns1}
in a perturbed half-space
\[
\Omega:=\{x=(x_1,x')\in \mathbb R^3 \, | 
\, x_1>M(x')\}
\quad \text{for $M \in \cap_{k=1}^\infty H^k(\mathbb R^2)$}.
\]
Note that for classical solutions, the initial--boundary value problem of  \eqref{ns1} can be rewritten
as
\begin{subequations}
\label{ns}
\begin{gather}
\left\{\begin{array}{lll}
\rho_t+\textrm{div}(\rho {u})=0, \\ \vspace{3mm}
\displaystyle \rho{u}_t+\rho{u}\cdot\nabla{u}+\nabla
P=\mu \Delta u+(\mu+\lambda)\nabla(\textrm{div}{u}),\\ \vspace{3mm}
\displaystyle c_v\rho\big(\theta_t+{u}\cdot\nabla\theta\big)+P\textrm{div}{u}
\displaystyle= \kappa\Delta\theta+2\mu|\mathcal{D}( u)|^2+\lambda(\textrm{div}u)^2,
\end{array}\right.
\label{cnse}
\\
  (\rho,{u},\theta)(0,x)=(\rho_0,{u}_0,\theta_{0})(x), 
  \label{ini1}
  \\
\lim_{x_1\to\infty}(\rho,u_1,u_2,u_3,\theta)(t,x_1,x')=(\rho_+,u_+,0,0,\theta_{+}),
 \label{bc1}
 \\
   u(t,M(x'),x')={u}_b (x') \quad \text{for} \
  x' \in \mathbb R^2,
  \label{bc2}
  \\
  \theta(t,M(x'),x')=\theta_b (x')\quad \text{for} \
  x' \in \mathbb R^2,
  \label{bc3}
 \end{gather}
 \end{subequations}
 where $\mathcal{D}(u) = (\nabla  u + (\nabla  u)^{tr})/2$ is the deformation tensor,
and $\rho_{+}>0$, $u_+<0$, $\theta_{+}>0$ are constants.
It is assumed that
\begin{equation}\label{outf}
{u}_b\cdot n\geq c_1>0,\quad \theta_b\geq c_2>0,
\end{equation}
with $c_1, c_2$ are constants.
The unit outer normal vector of the boundary 
$\partial\Omega=\{x\in \mathbb R^3 \, | \, x_1=M(x')\}$ is represented by 
\begin{equation} \label{nvector}
{n}(x')=(n_1,n_2,n_3)(x')
:=\left(\frac{-1}{\sqrt{1+|\nabla M|^2}},
\frac{\partial_{x_2}M}{\sqrt{1+|\nabla M|^2}},
\frac{\partial_{x_3}M}{\sqrt{1+|\nabla M|^2}} \right)(x').
\end{equation}

We assume that  the initial density and initial   absolute temperature are uniformly positive:
\begin{equation}
\quad
\inf_{x \in \Omega} \rho_0 (x) > 0,
\quad
\inf_{x \in \Omega} \theta_0 (x) > 0.
\quad
\nonumber
\end{equation}
We will construct solutions whose density is positive everywhere.
The outflow boundary condition $ u_b\cdot  n>0$ guarantees that no boundary condition is suitable for the density $\rho$.
The compatibility conditions are also necessary for the initial data $u_0$ and $\theta_{0}$. 
We will state clearly the conditions in Section \ref{sec2}.


Furthermore, we assume that the Mach number for the end states 
satisfies the supersonic condition:
\begin{equation}\label{super1}
\frac{|u_+|}{\sqrt{\gamma R\theta_+}}>1.
\end{equation}

The initial--boundary value problems of the compressible Navier--Stokes equation have been studied extensively. We are interested in the long-time behavior of the solutions. In this direction, Matsumura and Nishida \cite{m-n83} made a pioneering progress on the non-isentropic compressible Navier--Stokes equation \eqref{ns1} by considering the initial--boundary value problems over an exterior domain and a half-space. In \cite{m-n83}, the existence of the time-global solution is established, and the solution is proved to converge to the stationary solution as time goes to infinity, with the assumption that the initial perturbation from the stationary solution belongs to $H^m$ and its $H^m$-norm is small enough. Many researchers followed \cite{m-n83}, but they focused mainly on the isentropic compressible Navier--Stokes equation which does not include  the equation of entropy, while a pressure term $P=K \rho^{\gamma}$ with constants $K>0$ and $\gamma \geq 1$ is involved in the equation of motion to make up for the missing equation of entropy, since detailed studies on the non-isentropic equation were still hard in that time. Kagei and Kobayashi \cite{kagei05} further analyzed the isentropic compressible Navier--Stokes equation over the half-space in the case that the stationary solution is a constant state, and obtained an accurate convergence rate of the time global solutions toward the constant state by assuming the initial perturbation belongs to $H^m\cap L^1$. Valli \cite{Va83} considered the case when the equation of motion has an small external forcing term, and then proved the unique existence and stability of stationary solution over an bounded domain with the non-slip boundary condition. The velocity of the stationary solution in \cite{Va83} is nonzero but small due to the perturbative nature of the argument and the smallness of the external forcing term. Matsumura \cite{matu-01} classified all possible time-asymptotic states for a one-dimensional half-space problem of the isentropic compressible Navier--Stokes equation with no external force. It is also conjectured in \cite{matu-01} that one of time-asymptotic states for an {\it outflow problem} is a stationary solution of which end state satisfying the {\it supersonic condition} \eqref{super1}. 

An outflow problem in our context is an initial--boundary value problem with an {\it outflow boundary condition} in \eqref{outf}. The asymptotic stability of the stationary solution (Matsumura's conjecture) was proved in Kawashima, Nishibata and Zhu \cite{knz03}. After that, Nakamura, Nishibata and Yuge \cite{nny07} verified an exponential convergence rate toward the stationary solution by assuming that the initial perturbation belongs to some weighted Sobolev space. For a three-dimensional half-space $\mathbb R^3_+$ (that is, the case $M\equiv 0$), Kagei and Kawashima \cite{kg06} showed that a {\it planar stationary solution}, that is, a special solution independent of tangential direction $x'$, with its tangential velocities $(u_2,u_3)$ being zero, is time asymptotically stable. It has been also known from \cite{nn09} that the convergence rate is exponential if the initial perturbation decays exponentially fast at an infinite distance. Suzuki and Zhang \cite{SZ1} extended the results in \cite{kg06,nn09} by considering the problem on a {\it perturbed half-space} with a curved boundary, and established the unique existence and asymptotic stability of stationary solution of the isentropic compressible Navier--Stokes equation. Remarkably, the stationary solution constructed in \cite{SZ1} depends on all directions $x=(x_1,x')$ and has multidirectional flow, in contrast to the planar stationary solution studied in \cite{kg06,nn09}, which is independent of tangential $x'$ and has unidirectional flow along the $x_{1}$-axis.

In this paper, we investigate viscous heat-conductive fluid of which motion is governed by the non-isentropic compressible Navier--Stokes equation \eqref{ns1}. As mentioned above, the outflow problem for the isentropic compressible Navier--Stokes equation has been well studied. It is of greater interest to study the outflow problem for the {\it nonisentropic} compressible Navier--Stokes equation.
Kawashima, Nakamura, Nishibata and Zhu \cite{knnz-10} showed the unique existence and asymptotic stability of {\it the planer stationary solution} over a half line $\mathbb R_{+}$ assuming the outflow condition in \eqref{outf} and the supersonic condition \eqref{super1}. 
They also obtained the convergence rate of the time-global solution towards the planer stationary solution.
Wang \cite{w22} studied the planer stationary solution over a three-dimensional half-space $\mathbb R_{+}^{3}$, and then established the asymptotic stability as well as obtained the convergence rate.
It would be worth extending the results \cite{knnz-10,w22} to that over the perturbed half space $\Omega$. In this paper, we show the unique existence and asymptotic stability of the stationary solutions of \eqref{ns1} over the perturbed half space $\Omega$ with a curved boundary.


\begin{flushleft}
\textbf{Acknowledgements. } M. S. was supported by JSPS KAKENHI Grant Numbers 21K03308. 
\end{flushleft}

\subsection{Notations} \label{ss-Notation}
We introduce some frequently used notations in this paper.
Let $\pd_i := \frac{\pd}{\pd x_i}$ and $\dt := \frac{\pd}{\pd t}$.
The operators $\nabla := (\pd_1 ,\pd_2,\pd_3)$
and $\Delta := \sum_{i=1}^n \pd_i^2$ denote the standard gradient and Laplacian with 
respect to $x = (x_1,x_2,x_3)$.
We also define a standard divergence by
 $\div u := \nabla \cdot u := \sum_{i=1}^3 \pd_i u_i$.
The operator $\nabla_{x'} := (\pd_2,\pd_3)$ denotes the tangential gradient with respect to $x'=(x_2,x_3)$.

For a non-negative integer $k$,
we denote by $\nabla^k$ and $\nabla_{x'}^k$ the totality of 
all $k$-th order derivatives with respect to $x$ and $x'$, respectively.
For a domain $\Sigma \subset \R^n$ and $1 \le p \le \infty$,
the space $L^p(\Sigma)$ denotes the standard Lebesgue space
equipped with the norm $\lpo{p}{\cdot}{\Sigma}$.
For a non-negative integer $m$, $H^m = H^m(\Sigma)$ denotes
the $m$-th order Sobolev space over $\Sigma$ in the $L^2$ sense
with the norm $\|\cdot\|_{H^m(\Sigma)}$.
For any $\beta \ge 0$, the space $\lteasp{\beta}(\Sigma)$ denotes
the weighted $L^2$ space that is exponentially weighted in the normal direction. It is defined as $\lteasp{\beta}(\Sigma) := \{ u \in L^2(\Sigma)
  \, ; \, \|u\|_{L^{2}_{{\rm e},\beta}(\Sigma)} < \infty \}$ equipped with the norm %
\[
\|u\|_{L^{2}_{{\rm e},\beta}(\Sigma)} :=
\Bigl(
\int_{\Sigma} e^{\beta x_1} |u(x)|^2 \, dx
\Bigr)^{1/2}.
\]
In the case $\Sigma = \Omega$, 
the spaces $L^p(\Omega)$, $H^m(\Omega)$, and $\lteasp{\beta}(\Omega)$ 
are sometimes abbreviated by $L^p$, $H^m$, and $\lteasp{\beta}$, respectively. 
Note that $L^2 = H^0 = \lteasp{0}$, and we denote $\lt{\cdot} := \lp{2}{\cdot}$.

We also define the following solution spaces
\begin{align*}
X_m(0,T)
& :=
\{
(\vp,\psi,\zeta) \in C([0,T] ; H^m)
\; ; \,
\nabla \vp \in L^2([0,T] ; H^{m-1}), \
\nabla \psi,\nabla \zeta \in L^2([0,T] ; H^m)
\},
\\
X^{\text{\rm e}}_{m,\beta} (0,T)
& :=
\{
(\vp,\psi,\zeta) \in X_m(0,T)
\; ; \,
(\vp,\psi,\zeta) \in C([0,T] ; \lteasp{\beta}), \
\nabla \psi,\nabla \zeta \in L^2([0,T] ; \lteasp{\beta})
\},
\end{align*}
where $T>0$ and $\beta \ge 0$ are constants. Moreover, we define
\begin{equation*}
\mathcal{X}^{\text{\rm e}}_{m,\beta} (0,T)
=X^{\text{\rm e}}_{m-1,\beta} (0,T)
\cap L^\infty(0,T ; H^m).
\end{equation*}

We use $c$ and $C$ to denote generic positive constants 
depending on $\mu$, $\lambda$, $\kappa$, $c_{v}$, $\rho_+$, $u_+$, $\|M\|_{H^9(\mathbb R^2)}$ and $\alpha$
but independent of $t$, $\beta$, $\delta$, $\sigma$ and any further information of $\Omega$.
We note that the positive constants $\alpha$, $\beta$, $\delta$ and $\sigma$ 
will be given in the next subsection.
Let us also denote a generic positive constant depending additionally
on other parameters $a$, $b$, $\ldots$ by $C(a,\, b,\, \ldots)$, and a generic positive constant depending additionally
on some further information of $\Omega$ (other than $\|M\|_{H^9(\mathbb R^2)}$) by $C(\Omega)$.  
Furthermore, $A \lesssim B$ means $A \leq C B$ for the generic constant $C$ given above, and $A \lesssim_\Omega B$ means $A \leq C(\Omega) B$ for the generic constant $C(\Omega)$ given above.

\subsection{Main results}
Before mentioning our main results, we introduce a result in \cite{knnz-10}
which showed the unique existence of planar stationary solutions
$(\rt,\ut,\tt)(x_1)=(\rt,\ut_1,0,0,\tt)(x_1)$
over the half-space 
$\mathbb R_+^3:=\{ x \in \mathbb R^3 \,;\, x_1>0\}$.
The planar stationary solution $(\rt(x_1), \ut_1(x_1), \tt(x_1))$ solves the ordinary differential equations
\begin{subequations}
\label{ste}
\begin{gather}
(\rt \ut_1)_{x_1} = 0,
\label{ste1}
\\
(\rt \ut_1^2 + R\rt\tt)_{x_1} = \mu_1 \ut_{1 x_1 x_1},
\\
\left(\rt \ut_1\left(c_v\tt+\frac{\ut_1^2}{2}\right)+ R \rt \tt \ut_1\right)_{x_1} = \kappa \tt_{ x_1 x_1}+
\mu_1 (\ut_1\ut_{1 x_1})_{ x_1}
\label{ste2}
\end{gather}
with $\tilde p=R\rt\tt$
and $\mu_1$ is a positive constant defined by $\mu_1:=2\mu+\lambda$.
The boundary data are prescribed as
 \begin{gather}
  (\ut_1,\tt)(0)=(\ut_b,\tt_{b}), 
  \label{bb1}
  \\
\lim_{x_1\to\infty}(\rt,\ut_1,\tt)(x_1)=(\rho_+,u_+,\theta_{+}).
 \label{bb2}
 \end{gather}
 \end{subequations}

The following quantity $\tilde{\delta}$ plays an important role in stability analysis:
\[
\tilde{\delta} := |\tilde{u}_{b}-u_+|+|\tilde\theta_{b}-\theta_+|.
\]
We call it the {\it boundary strength}.

\begin{proposition}[\cite{knnz-10}] \label{ex-st}
For the supersonic case, i.e. when \eqref{super1} holds, if $\tilde{\delta}<\varepsilon_0$ for a certain positive constant $\varepsilon_0$, there exists a unique smooth solution $(\rt,\ut_1,\tt)$ to the problem \eqref{ste}.
Moreover, there exist a positive constant $\alpha$
such that the stationary solution $(\rt,\ut_1,\tt)$ satisfies
\begin{equation}
|\pd_{x_1}^k (\rt(x_1) - \rho_+, \ut_1(x_1) - u_+,\tt(x_1) - \theta_+)|
\lesssim \tilde{\dels} e^{- \alpha x_1}
 \ \; \text{for} \ \;
k = 0,1,2,\dots.
\label{stdc1}
\end{equation}
\end{proposition}

From now on we discuss our main results.
We first show the unique existence of stationary solutions 
$(\rho^s, u^s,\theta^s)=(\rho^s,u^s_1,u^s_2,u^s_3,\theta^s)$ over the domain $\Omega$
by regarding $(\rho^s, u^s,\theta^s)(x)$ as a perturbation of $(\rt,\ut,\tt)(\tilde{M}(x))$, where
\begin{equation}\label{tM1}
\tilde{M}(x):=x_1-M(x'). 
\end{equation}
The stationary solutions satisfy the equations
\begin{subequations}
\label{snse}
\begin{gather}
\div (\rho^s  u^s) = 0,
\label{snse1}
\\
\rho^s \{ ( u^s \cdot \nabla) u^s \}
= \mu \Delta  u^s + (\mu + \lambda) \nabla (\div  u^s) - \nabla (R\rho^{s}\theta^{s}),
\label{snse2}
\\
c_{v}\rho^s{u}^s\cdot\nabla\theta^s+R\rho^{s}\theta^{s}\textrm{div}{u}^s
=\kappa\Delta\theta^s+2\mu|\mathcal{D}( u^s)|^2+\lambda(\textrm{div} u^s)^2.
\end{gather}
with the conditions
\begin{gather}
 u^s(M(x'),x') = { u}_b(x'),~~\theta^s(M(x'),x') = \theta_b(x'),
\\
\lim_{x_1 \to \infty}  \rho^s(x)
= \rho_+,
\quad
\lim_{x_1 \to \infty}  u^s(x)
=  (u_+,0,0),\quad
\lim_{x_1 \to \infty}  \theta^s(x)
= \theta_+,
\\
\inf_{x\in\Omega}\rho^s(x)>0,
\quad
\inf_{x\in\Omega}\theta^s(x)>0.
\end{gather}
\end{subequations}
To state the existence theorem, we use the notation
\[
\dels := \|\ub - (\tilde{u}_b,0,0)\|_{H^{13/2}(\partial \Omega)}+\|\theta_b - \tilde\theta_b\|_{H^{13/2}(\partial \Omega)}+\tilde{\delta}.
\]
We also employ an extension $U=U(x)$ of $u_b-(\tilde{u}_b,0,0)$, which satisfies 
\begin{subequations}\label{ExBdry0}
\begin{gather}
U(M(x'),x')=u_b(x')-(\tilde{u}_b,0,0),
\label{ExBdry1} \\
U(x_1,x')=0 \quad \text{if $x_1 >M(x')+1$},
\label{ExBdry2} \\
\|U\|_{H^7(\Omega)} \lesssim  \delta,
\label{ExBdry3}
\end{gather}
\end{subequations}
and an extension $\Theta=\Theta(x)$ of $\theta_b-\tilde{\theta}_b$, which satisfies 
\begin{subequations}\label{ExBd0}
\begin{gather}
\Theta(M(x'),x')=\theta_b(x')-\tilde{\theta}_b,
\label{ExBd1} \\
\Theta(x_1,x')=0 \quad \text{if $x_1 >M(x')+1$},
\label{ExBd2} \\
\|\Theta\|_{H^7(\Omega)} \lesssim  \delta.
\label{ExBd3}
\end{gather}
\end{subequations}
The existence result is summarized in the following theorem.
\begin{theorem}\label{th1}
Let \eqref{outf} and \eqref{super1} hold, and $m=3,4,5$.
There exist positive constants $\beta \leq \alpha/2$, 
where $\alpha$ is defined in Proposition \ref{ex-st},
and $\ep_0=\ep_0(\beta,\Omega)$ depending on $\beta$ and $\Omega$, such that if $\dels \le \ep_0$, then
the stationary problem \eqref{snse} has a unique solution
$(\rho^s,u^s,\theta^s)$ that satisfies
\begin{gather*}
(\rho^s-\rt\circ\tilde{M},u^s-\ut\circ\tilde{M}-U,\theta^s-\tt\circ\tilde{M}-\Theta) 
\in \lteasp{\beta}(\Omega)\cap H^m(\Omega),
\\
\|(\rho^s-\rt\circ\tilde{M},u^s-\ut\circ\tilde{M}-U,\theta^s-\tt\circ\tilde{M}-\Theta)\|_{\lteasp{\beta}}^2
\leq C_0\delta,
\\
\|(\rho^s-\rt\circ\tilde{M},u^s-\ut\circ\tilde{M}-U,\theta^s-\tt\circ\tilde{M}-\Theta)\|_{H^m}^2 \leq C_0\delta,
\end{gather*}
where $C_0=C_0(\beta,\Omega)$ is a positive constant depending on $\beta$ and $\Omega$.
\end{theorem}
We also state the stability theorem.
\begin{theorem}\label{th2}
Let  \eqref{outf} and \eqref{super1} hold.
There exist positive constants $\beta \leq \alpha/2$, 
where $\alpha$ is defined in Proposition \ref{ex-st},
and $\ep_0=\ep_0(\beta,\Omega)$ depending on $\beta$ and $\Omega$ such that if 
$\|(\rho_0-\rho^s, u_0-u^s,\theta_0-\theta^s)\|_{\lteasp{\beta}}
+\hs{3}{(\rho_0-\rho^s, u_0-u^s,\theta_0-\theta^s)} + \dels \le \ep_0$
and $(\rho_0,u_0,\theta_0)$ satisfies the compatibility conditions of order $0$ and $1$,
then the initial-boundary value problem \eqref{ns} has 
a unique time-global solution $(\rho, u,\theta)$ such that
$(\rho-\rho^s,u-u^s,\theta-\theta^s) \in X^{\text{\rm e}}_{3,\beta} (0,T)$.
Moreover, it holds
\begin{equation*}
\|(\rho-\rho^s,u-u^s,\theta-\theta^s)(t)\|_{L^\infty}
\leq C_0 e^{-\sigma t},
\end{equation*}
where $C_0=C_0(\beta,\Omega)$ and $\sigma=\sigma(\beta,\Omega)$ are 
positive constants depending on $\beta$ and $\Omega$
but independent of $\delta$ and $t$.
\end{theorem}

Theorem \ref{th2} requires 
the condition $(\rho_0-\rho^s, u_0-u^s,\theta_0-\theta^s) \in \lteasp{\beta}(\Omega)$.
Without assuming this condition, we can prove the following stability theorem:

\begin{theorem}\label{th3}
Let  \eqref{outf} and \eqref{super1} hold.
There exists a positive constant $\ep_0=\ep_0(\beta,\Omega)$ 
depending on $\beta$ and $\Omega$ such that if 
$\hs{3}{(\rho_0-\rho^s, u_0-u^s,\theta_0-\theta^s)} + \dels  \le \ep_0$
and $(\rho_0,u_0,\theta_0)$ satisfies the compatibility conditions of order 0 and 1,
then the initial-boundary value problem \eqref{ns} has 
a unique time-global solution $(\rho, u,\theta)$ such that
$(\rho-\rho^s,u-u^s,\theta-\theta^s) \in X_{3} (0,T)$.
Moreover, 
\begin{equation*}
\|(\rho-\rho^s,u-u^s,\theta-\theta^s)(t)\|_{L^\infty}\to 0 \quad \text{as $t\to\infty$}.
\end{equation*}
\end{theorem}

If the domain $\Omega$ is sufficiently flat, in the above theorems,
we can take the constants $\ep_0$, $C_0$, and $\sigma$ independent of $\Omega$.
Namely, the following corollary holds:

\begin{corollary}\label{cor}
Suppose that $\|M\|_{H^9} \leq \eta_0$ holds for $\eta_0$ being in Lemma \ref{CattabrigaEst}.
Then Theorems \ref{th1}--\ref{th3} hold with constants $\ep_0$, $C_0$, and $\sigma$
independent of $\Omega$. %
\end{corollary}

\begin{remark}
It is worth noticing that in Theorems \ref{th1}--\ref{th3}, the existence and stability are established as long as the boundary of domain is given by a graph, so a domain boundary with a large curvature is allowed. The works \cite{kg06,nn09} adopted the boundary condition as $u_b(x')=(\tilde{u}_b,0,0)$ for the half-space $\mathbb R^3_+$, which is covered in our theorems. 

It is also worth pointing out that Corollary \ref{cor} covers the boundary condition $u_b(x')=\tilde{u}_b n(x')$, where $n(x')$ is the unit outer normal vector given in \eqref{nvector}. Indeed, if $\varepsilon_0$ is independent of $\Omega$, there is no issue to take $n(x')$ depending on $\Omega$ so that $\|\tilde{u}_b n-(\tilde{u}_b,0,0)\|_{H^{13/2}(\partial\Omega)} + |\tilde{u}_b-u_+|=\delta\leq \varepsilon_0/2$ holds. This boundary condition describes the situation where the fluid is flowing out of the domain from only the normal direction of the boundary, which seems more reasonable from a physical point of view. 
\end{remark}

The stationary problem \eqref{snse} is hard to solve directly, in contrast to the case when the spatial domain $\Omega = \mathbb{R}^3_+$ and a planar stationary solution only depending on $x_1$ is considered, since in the latter case the system \eqref{snse} reduces to an ODE \eqref{ste}. It is also different from the stationary incompressible Navier-Stokes equation, in which the system is elliptic. Our stationary equations are not categorized as elliptic equations. To overcome these difficulties and solve \eqref{snse}, we first prove the existence of a time-global solution to the problem \eqref{ns}, and then construct a stationary solution making use of this time-global solution. 

Let us further specify the idea to construct the time-global solution. We use a continuous argument combining time-local solvability and an a priori estimate. The key of the proof is the dissipative structure which makes solutions of the corresponding homogeneous equations decay exponentially fast as time tends to infinity. On the other hand, we expect from the stability theorem in \cite{nn09}, \cite{knnz-10} introduced above that the solution $(\rho,u, \theta)$ to the problem \eqref{ns} with $u_b=(u_+,0,0)$ will converge to the constant state $(\rho_+,u_+,\theta_+)$ at an exponential rate as time tends to infinity. When $u_b\neq(u_+,0,0)$, after suitable reformulation, all effects coming from $u_b\neq(u_+,0,0)$ are represented by inhomogeneous terms in the equations. We define a perturbation 
$$\Phi (t, x):=(\vp,\psi,\zeta)(t,x) :=(\rho, u, \theta)(t,x)-(\rt, \ut, \tt)(\tilde{M}(x)) - (0, U, \Theta) (x)$$ 
and reformulate \eqref{ns} into a problem for $\Phi$. We are then able to derive a priori estimates of the solutions $\Phi$ to the reformulated problem utilizing the dissipative structure. For the construction of stationary solutions, a similar approach as in \cite{SZ1,Va83}  is taken. More precisely, we define the translated time-global solutions $\Phi^k(t,x):=\Phi(t+kT^*,x)$ for any $T^* > 0$ and $k=1,2,3,\ldots$. It can then be proved that the sequence $\{\Phi^k\}$ converges to a certain time-periodic solution with a period $T^*$. After that, the uniqueness of time-periodic solution and the arbitrariness of $T^*$ allow us to show that the time-periodic solution is actually time-independent, yielding a stationary solution to our problem. 

Before closing this section, we sketch the outline of this paper.
In Section \ref{sec2}, we reformulate the initial-boundary value problem \eqref{ns} into an initial-boundary value problem for a perturbation 
from the stationary solution $(\rt\circ\tilde{M},\ut\circ\tilde{M}, \tt\circ\tilde{M})$ in the half-space, as stated in \eqref{eq-pv}. In Section \ref{sec3}, we show the unique existence of the time-global solution to the reformulated problem \eqref{eq-pv} (see Theorem \ref{global1}) by proving an a priori estimate in Proposition \ref{apriori1}. The derivation of the a priori estimate is based on the energy form in \cite{kagei05,knz03}, combined with the Matsumura--Nishida energy method in \cite{m-n83} and the weighted energy method in \cite{nny07,nn09}. The stationary solutions are constructed in Section \ref{S5} (in particular, Subsections \ref{S5.1} and \ref{S5.2}), by using the method mentioned just above. Subsection \ref{S5.3} is devoted to the proof of the asymptotic stability of the stationary solution in the weighted space $L^2_{e, \beta} (\Omega)$, where the exponential convergence rate is obtained. For initial data not in $L^2_{e, \beta} (\Omega)$, the asymptotic stability of the stationary solution is shown in Section \ref{S6}. In Appendix \ref{Appenx0}, we give the derivation of the equations of the energy forms \eqref{ea1} and \eqref{sea1}. 
Appendix \ref{BasciIneq} provides basic inequalities and estimates, which are often used in this paper.

\section{Reformulation}\label{sec2}
For the proof of Theorems \ref{th1} and \ref{th2},
we begin by reformulating the initial-boundary value problem \eqref{ns}. 
Let us introduce the perturbations
\begin{gather*}
(\vp,\psi,\zeta)(t,x)
:=
(\rho, u,\theta)(t,x)
-
(\rt, \ut,\tt)(\tilde{M}(x)) - (0, U, \Theta) (x),
\ \; \text{where} \ \;
\psi = (\psi_1,\psi_2,\psi_3). 
\end{gather*}
Here $\tilde{M}(x)$ is defined in \eqref{tM1}.

Owing to equations in (\ref{ns}) and (\ref{ste}),
the perturbation $(\vp,\psi)$ satisfies the system of
equations
\begin{subequations}
\label{eq-pv}
\begin{gather}
\vp_t + u \cdot \nabla \vp + \rho \div \psi
= f+F,
\label{eq-pv1}
\\
\rho \{ \psi_t + (u \cdot \nabla) \psi \}
- L \psi + R\theta\nabla \vp+R\rho\nabla\zeta
= g + G,
\label{eq-pv2}
\\
c_v\rho\big(\zeta_t+{u}\cdot\nabla\zeta\big)+R\rho\theta\textrm{div}\psi
-\kappa\Delta\zeta=2\mu|\mathcal{D}(\psi)|^2+\lambda(\textrm{div}\psi)^2+h+H.
\label{eq-pv3}
\end{gather}
The boundary and initial conditions for $(\vp,\psi,\zeta)$ 
follow from (\ref{ini1})--(\ref{bc3}) and (\ref{bb1})--(\ref{bb2}) as
\begin{gather}
\psi(t,M(x'),x') = 0,~~\zeta(t,M(x'),x') = 0,
\label{pbc}
\\
(\vp,\psi,\zeta)(0,x)
=(\vp_0, \psi_0, \zeta_0)(x)
:= (\rho_0, u_0,\theta_0)(x) - (\rt, \ut, \tt)(\tilde{M}(x)) - (0, U, \Theta) (x).
\label{pic}
\end{gather}
\end{subequations}
Here $L \psi$, $f$, $F$, $g$, $G$, $h$, and $H$ are defined by
\begin{align*}
L \psi &:= \mu \Delta \psi + (\mu + \lambda) \nabla \div \psi,
\\
f &:= - \nabla\tilde{\rho} \cdot \psi  -  \tilde{u}_1'\varphi - \vp \div U ,
\\
F &:=  - \nabla \tilde{\rho} \cdot U - \tilde{\rho} \div U,
\\
g&:= - \rho (\psi \cdot \nabla) (\tilde{u}+U) - \varphi ((\tilde{u}+U) \cdot \nabla) (\tilde{u}+U)
-R \zeta\nabla \tilde{\rho}-R\vp\nabla(\tt+ \Theta),
\\
G&:=-\rt ((\tilde{u}+U) \cdot \nabla) U-\rt (U \cdot \nabla)\tilde{u} + LU-R \rt\nabla \Theta-R \Theta\nabla \tilde{\rho} 
\\
&\qquad +\begin{bmatrix}
\mu \tilde{u}''_1 \sum^3_{j=2} (\pd_{j} M)^2  - \mu \tilde{u}'_1 \sum^3_{j=2}  \pd^2_{j} M
\\
\left(R\tt\tilde{\rho}' +R\tilde{\rho}\tt'-(\mu+\lambda) \tilde{u}''_1 \right)\pd_{2} M
\\
\left(R\tt\tilde{\rho}' +R\tilde{\rho}\tt'-(\mu+\lambda) \tilde{u}''_1 \right) \pd_{3} M
\end{bmatrix},
\\
h&:= - c_v\rho (\psi \cdot \nabla) (\tt+\Theta) - c_v\vp ((\ut+U)\cdot \nabla) (\tt+\Theta)- R\rt\zeta\div(\ut+U)
\\
&\qquad-R\vp\theta\div(\ut+ U)
+4\mu\mathcal{D}(\psi):\mathcal{D}(\ut+U)+2\lambda\div\psi\div(\ut+U),
\\
H&:=-c_v\rt (U \cdot \nabla)\tilde{\theta}-c_v\rt ((\tilde{u}+U) \cdot \nabla) \Theta-R \rt \Theta\div\ut-R \tilde{\rho} (\tt+\Theta)\div U
\\
&\qquad+\kappa \tt''\sum^3_{j=2} (\pd_{j} M)^2  - \kappa \tt' \sum^3_{j=2}  \pd^2_{j} M+\kappa\Delta \Theta+2\mu|\mathcal{D}(U)|^2+4\mu\mathcal{D}(\ut):\mathcal{D}(U)
\\
&\qquad+\lambda(\div U)^2+2\lambda\div\ut:\div U+\mu(\tilde{u}'_1)^2\left((\pd_{2} M)^2+(\pd_{3} M)^2 \right).
\end{align*}
Note that $L$ is a differential operator; $f$, $g$, and $h$ are homogeneous terms for $(\vp,\psi,\zeta)$; 
$F$,  $G$, and $H$ are inhomogeneous terms independent of $t$.
Furthermore, $F$,  $G$, and $H$
can be estimated by using $M \in H^9(\mathbb R^2)$, \eqref{stdc1}, and \eqref{ExBdry0} as
\begin{equation}\label{h1}
\| (F,G,H) \|_{L^2_{e, 3\alpha/2}} \lesssim  \delta, \quad 
\|(F,G,H)\|_{H^5} \lesssim \delta.
\end{equation}


We denote the perturbation as
\[
\Phi := (\vp,\psi,\zeta),
\quad
\Phi_0 := (\vp_0,\psi_0,\zeta_0).
\]
In order to establish the local existence of the solution
in strong sense,
we assume compatibility conditions for the initial data.
It is necessary to assume the compatibility conditions of order 0, 1, and 2:
\begin{subequations}\label{cmpa0}
\begin{gather} 
\psi_0|_{x_1=M(x')} = 0,
\quad
\zeta_0|_{x_1=M(x')} = 0,
\label{cmpa1}
\\
\left\{
\rho_0 (u_0 \cdot \nabla) \psi_0
- L \psi_0
+R\theta_0\nabla \vp_0+R\rho_0\nabla\zeta_0
- (g+G)|_{t=0}
\right\}|_{x_1 = M(x')}
= 0,
\label{cmpa3} \\
\left\{
c_v\rho_0 u_0\cdot\nabla\zeta_0+P_0\textrm{div}\psi_0
-\kappa\Delta\zeta_0-2\mu|\mathcal{D}(\psi_0)|^2-\lambda(\textrm{div}\psi_0)^2-(h+H)|_{t=0}
\right\}|_{x_1 = M(x')}
= 0,
\label{cmpa44} \\
\left[\partial_t\left\{
\rho (u \cdot  \nabla) \psi
- L \psi
+ R\theta\nabla \vp+R\rho\nabla\zeta
- g
\right\}|_{t=0}\right]_{x_1 = M(x')}
= 0,
\label{cmpa2}
\\
\left[\partial_t\left\{
c_v\rho u\cdot\nabla\zeta+P\textrm{div}\psi
-\kappa\Delta\zeta-2\mu|\mathcal{D}(\psi)|^2-\lambda(\textrm{div}\psi)^2-h
\right\}|_{t=0}\right]_{x_1 = M(x')}
= 0.
\label{cmpa55}
\end{gather}
\end{subequations}
Note that the equations \eqref{cmpa2}--\eqref{cmpa55}, which are of order 2, can be written into a form which only contains spatial-derivatives of the initial data by using \eqref{eq-pv} 

It suffices to show Theorems \ref{th4}--\ref{th5} and Corollary \ref{cor2} below for the completion of the
proof of Theorems \ref{th1}--\ref{th2} and the claims corresponding to Theorems \ref{th1}--\ref{th2} in Corollary \ref{cor}, respectively. 

\begin{theorem}\label{th4}
Let  \eqref{outf} and \eqref{super1} hold, and $m=3,4,5$.
There exist positive constants $\beta \leq \alpha/2$, 
where $\alpha$ is defined in Proposition \ref{ex-st},
and $\ep_0=\ep_0(\beta,\Omega)$ depending on $\beta$ and $\Omega$
such that if $\dels \le \ep_0$, 
the stationary problem corresponding to \eqref{eq-pv} has a unique solution
$\Phi^s \in \lteasp{\beta}(\Omega)\cap H^m(\Omega)$ with
\begin{gather*}
\|\Phi^s\|_{\lteasp{\beta}}^2+\|\Phi^s\|_{H^m}^2\leq C_0 \delta,
\end{gather*}
where $C_0=C_0(\beta,\Omega)$ is a positive constant depending on $\beta$ and $\Omega$.
\end{theorem}

\begin{theorem}\label{th5}
Let \eqref{outf} and \eqref{super1} hold.
There exist positive constants $\beta \leq \alpha/2$, 
where $\alpha$ is defined in Proposition \ref{ex-st},
and $\ep_0=\ep_0(\beta,\Omega)$ depending on $\beta$ and $\Omega$
such that if 
$\|\Phi_0-\Phi^s\|_{\lteasp{\beta}}+\hs{3}{\Phi_0-\Phi^s} + \dels \le \ep_0$
and $\Phi_0$ satisfies the compatibility condition \eqref{cmpa1}--\eqref{cmpa44},
then the initial-boundary value problem \eqref{eq-pv} has 
a unique time-global solution 
$\Phi \in X^{\text{\rm e}}_{3,\beta} (0,\infty)$.
Moreover, it satisfies
\begin{equation*}
\|(\Phi-\Phi^s)(t)\|_{L^\infty}\leq C_0 e^{-\sigma t},
\end{equation*}
where $C_0=C_0(\beta,\Omega)$ and $\sigma=\sigma(\beta,\Omega)$
are positive constant depending on $\beta$ and $\Omega$.
\end{theorem}

\begin{corollary}\label{cor2}
Suppose that $\|M\|_{H^9} \leq \eta_0$ for $\eta_0$ being in Lemma \ref{CattabrigaEst}.
Then Theorems \ref{th4} and \ref{th5} hold with constants $\ep_0$, $C_0$, and $\sigma$
independent of $\Omega$. 
\end{corollary}

\section{Time-global solvability}\label{sec3}

This section provides the time-global solvability
of the initial--boundary value problem \eqref{eq-pv}.

\begin{theorem}\label{global1}
Let \eqref{outf} and \eqref{super1} hold, and $m=3,4,5$.
There exist positive constants $\beta \leq \alpha/2$, 
where $\alpha$ is defined in Proposition \ref{ex-st},
and $\ep_0=\ep_0(\beta,\Omega)$ depending on $\beta$ and $\Omega$ such that if 
$\|\Phi_0\|_{\lteasp{\beta}}+\hs{m}{\Phi_0} + \dels \le \ep_0$
and $\Phi_0$ satisfies the compatibility condition \eqref{cmpa1}--\eqref{cmpa44} for $m=3,4$,
\eqref{cmpa0} for $m=5$,
then the initial-boundary value problem \eqref{eq-pv} has 
a unique time-global solution 
$\Phi \in X^{\text{\rm e}}_{m,\beta} (0,T)$.
Moreover, it satisfies
\begin{equation}\label{bound1}
\|\Phi(t)\|_{\lteasp{\beta}}^2+\|\Phi(t)\|_{H^m}^2+\|\pd_t \Phi(t)\|_{H^{m-2}}^2
\\
\leq C_0(\|\Phi_0\|_{\lteasp{\beta}}^2+\hs{m}{\Phi_0}^2)e^{-\sigma t} + C_0\dels,
\quad t \in [0,\infty),
\end{equation}
where $C_0=C_0(\beta,\Omega)$ and $\sigma=\sigma(\beta,\Omega)$
are positive constant depending on $\beta$ and $\Omega$
but independent of $\delta$ and $t$.
\end{theorem}

The time-global solution $\Phi$ with \eqref{bound1} can be constructed 
by a standard continuation argument (see \cite{m-n83})
using the time-local solvability in Lemma \ref{local1} 
and the a priori estimate in Proposition \ref{apriori1} below.

\begin{lemma}\label{local1}
Let $m=3,4,5$. 
Suppose that the initial data $\Phi_0 \in H^m(\Omega)$ satisfies
the compatibility condition \eqref{cmpa1}--\eqref{cmpa44} for $m=3,4$,
\eqref{cmpa0} for $m=5$.
Then there exists a positive constant $T$ depending on 
$\hs{m}{\Phi_0}$ such that
initial-boundary value problem \eqref{eq-pv}
has a unique solution $\Phi \in X_{m}(0,T)$.
Moreover, if the initial data satisfies
$\Phi_0 \in \lteasp{\beta} (\Omega)$,
it holds $\Phi \in X_{m,\beta}^{\text{\rm e}}(0,T)$.
\end{lemma}

For the notational convenience, 
we define a norm $E_{m,\beta}(t)$ and a dissipative norm $D_{m,\beta}(t)$ by
\begin{align} 
E_{m,\beta}(t) &:=\|\Phi\|_{\lteasp{\beta}}^2+\|\Phi\|_{H^m}^2 \quad \text{for $m\geq 0$,}  \label{Ekbeta-def}
\\
D_{m,\beta}(t)
&:=\left\{\begin{array}{ll}
\beta\|\Phi\|_{\lteasp{\beta}}^2
+ \|(\nabla\psi, \nabla\zeta)\|_{\lteasp{\beta}}^2
+ \|{\frac{d}{dt}} {\vp} \|^2 
+ \|\vp(\tau,M(\cdot),\cdot)\|_{L^2(\mathbb R^2)}^2 & \text{if $m=0$},
\\
D_{0,\beta}(t)^2+\|(\nabla\Phi,\nabla^2\psi,\nabla^2\zeta)\|_{H^{m-1}}^2
+\left\|{\frac{d}{dt}} {\vp} \right\|_{H^{m}}^2 & 
\\
\displaystyle \quad +\sum_{i=1}^{[(m+1)/2]} \lt{(\pd_t^{i} \psi,\pd_t^{i}\zeta)}^2_{H^{m+1-2i}}
& \text{if $m\geq 1$},
\end{array}\right.  \label{Dkbeta-def}
\end{align}
where $\frac{d}{dt}:=\partial_{t}+u\cdot\nabla$. Furthermore, we denote
\[
N_{\beta}(T):= \sup_{t\in[0,T]}(\|\Phi(t)\|_{\lteasp{\beta}}+\|\Phi(t)\|_{H^3}).
\]

\begin{proposition}\label{apriori1}
Let \eqref{outf} and \eqref{super1} hold, and $m=3,4,5$.
Suppose that $\Phi \in X^{\text{\rm e}}_{m,\beta} (0,T)$
is a solution to the initial-boundary value problem \eqref{eq-pv}
for some positive constant $T$.
Then there exist positive constants $\beta \leq \alpha/2$, 
where $\alpha$ is defined in Proposition \ref{ex-st},
and $\ep_0=\ep_0(\beta,\Omega)$ depending on $\beta$ and $\Omega$ such that if 
$\sup_{t\in[0,T]}(\|\Phi(t)\|_{\lteasp{\beta}}+\|\Phi(t)\|_{H^m}) + \dels \le \ep_0$, 
the following estimate holds
\begin{gather}
e^{\sigma t} E_{m, \beta} (t) + \int^t_0 e^{\sigma \tau} D_{m,\beta} (\tau ) \, d \tau
\leq C_0(\|\Phi_0\|_{\lteasp{\beta}}^2 + \| \Phi_0\|^2_{H^m})
+ C_0 \delta e^{\sigma \tau},
\label{apes0}\\
\|\Phi(t)\|_{\lteasp{\beta}}^2+\|\Phi(t)\|_{H^m}^2+\|\pd_t \Phi(t)\|_{H^{m-2}}^2
\leq C_0(\|\Phi_0\|_{\lteasp{\beta}}^2+\hs{m}{\Phi_0}^2)e^{-\sigma t} + C_0\dels
\label{apes1}
\end{gather}
for $t \in [0,T]$. Here $C_0=C_0(\beta,\Omega)$ and $\sigma=\sigma(\beta,\Omega)$
are positive constant depending on $\beta$ and $\Omega$
but independent of $\delta$ and $t$.
\end{proposition}

\begin{corollary}\label{cor1}
Suppose that $\|M\|_{H^9} \leq \eta_0$ for $\eta_0$ being in Lemma \ref{CattabrigaEst}.
Then Theorems \ref{global1} and Proposition \ref{apriori1} hold 
with constants $\ep_0$, $C_0$, and $\sigma$
independent of $\Omega$. 
\end{corollary}

Lemma \ref{local1} can be proved in much the same way as in \cite{kg06-loc}.
Therefore, we omit the proof. 
In the remainder of this section, we prove Proposition \ref{apriori1} 
only for the case $m=3$, since the cases $m=4,5$ can be shown similarly.
We derive the $L^2$-norm of $\Phi$ by following the method in \cite{kg06,knz03,nn09,w22}.
To estimate the derivatives of $\Phi$, 
we use essentially the Matsumura--Nishida energy method in \cite{m-n83}.

\subsection{$L^2$ estimate} \label{ss-L2}

This subsection is devoted to the derivation of the 
estimate of the perturbation $(\vp,\psi,\zeta)$ in $\lteasp{\beta}(\Omega)$.
To do this, we introduce an energy form $\cale$,
similarly as in \cite{knnz-10,w22}, by
\[
\rho\cale := R\rho\tt\eta \left(\frac{\tilde\rho}{\rho}\right) + \frac{\rho}{2} |\psi|^2
+ c_v\rho(\tt+\Theta)\eta \left(\frac{\theta}{\tilde\theta + \Theta}\right),
\quad
\eta(r) := r - \ln r-1.
\]
Under the smallness assumption on $N_\beta (T)$, 
we have $\li{\Phi(t)} \ll 1$ by Sobolev's inequality \eqref{sobolev2}.
Hence, the energy form $\cale$ is equivalent to the 
square of the perturbation $(\vp,\psi, \zeta)$:
\begin{equation}
c (\vp^2 + |\psi|^2+\zeta^2)
\le
\cale
\le
C (\vp^2 + |\psi|^2+\zeta^2).
\label{sqr}
\end{equation}

Moreover, using $N_\beta(T)+\dels \ll 1$, we can derive the following uniform bounds of solutions:
\begin{equation}
0 < c \le (\rho,\theta)(t,x) \le C,
\quad
|u(t,x)| \le C .
\label{bdd}
\end{equation}
Using the time and space weighted energy method,
we obtain the energy inequality in $L^2$ framework.

\begin{lemma}
\label{lm1}
Under the same conditions as in Proposition \ref{apriori1} with $m=3$, 
it holds that
\begin{align}
&e^{\sigma t} \|\Phi(t)\|_{\lteasp{\beta}}^2
+ \int_0^t e^{\sigma \tau} D_{0,\beta}(\tau) \, d \tau
\notag\\
&\lesssim \|\Phi_0\|_{\lteasp{\beta}}^2
+  \sigma \int_0^t e^{\sigma \tau}\|\Phi(\tau)\|_{\lteasp{\beta}}^2 \, d \tau
+  \dels \int_0^t e^{\sigma \tau} \lt{\nabla \vp(\tau)}^2 \, d \tau
+  \dels \int_0^t e^{\sigma \tau} d \tau
\label{ea0}
\end{align}
for $t \in [0,T]$ and $\sigma>0$.
\end{lemma}

\begin{proof}
Following the computation in \cite{knnz-10,w22}, we see that
the energy form $\cale$ satisfies
\begin{equation}
(\rho \cale)_t
-\div (G_1 + B_1)
+ \mu |\nabla \psi|^2
+ (\mu + \lambda) (\div \psi)^2+\frac{\kappa}{\theta}|\nabla \zeta|^2
=
R_{11},
\label{ea1}
\end{equation}
where 
\begin{align*}
G_1
&:=
-\rho u \cale
-R\tt \vp\psi-R \rho \zeta\psi,
\nonumber
\\
B_1
&:=
\mu \nabla \psi \cdot \psi
+ (\mu + \lambda) \psi \div \psi
+\kappa\frac{\zeta}{\theta}\nabla\zeta, 
\nonumber
\\
R_{11}
&:=\!\!
\frac{R\tt\vp}{\rho}(f+F)+\psi \cdot (g+G)+\frac{\zeta}{\theta}(h+H)
+R\vp \psi \cdot \nabla\tt-R\Theta \psi \cdot\nabla \vp + R\zeta \psi \cdot \nabla \tilde\rho
\\
& \quad + R\rho\eta\left(\frac{\tilde\rho}{\rho}\right)u\cdot \nabla\tt
-  R\tt\frac{\vp^2}{\rho\tilde\rho}u\cdot \nabla\tilde\rho
 +c_v\rho\eta\left(\frac{\theta}{\tt+\Theta}\right)u\cdot \nabla(\tt+\Theta) 
 \\
& \quad  
+c_v\rho\frac{\tt+\Theta}{\theta} \zeta^{2} u\cdot \nabla \frac{1}{\tt+\Theta}
+\kappa\frac{\zeta}{\theta^2}\nabla\theta\cdot\nabla\zeta
+\frac{\zeta}{\theta}(2\mu|\mathcal{D}(\psi)|^2+\lambda(\textrm{div}\psi)^2).
\nonumber 
\end{align*}
For more details of the derivation, see Appendix \ref{Appenx0}.

Multiplying (\ref{ea1}) by a weight function 
$w = w(x_1,t) := e^{\beta x_1} e^{\sigma t}$, we get
\begin{multline}
(w \rho \cale)_t
- \div
\bigl\{
w (G_1 + B_1)
\bigr\}
+ \nabla w \cdot G_1
+ w 
(\mu  |\nabla \psi|^2 + (\mu + \lambda)  (\div \psi)^2)
+w\frac{\kappa}{\theta}|\nabla \zeta|^2
\\
=
w_t \rho \cale
- \nabla w \cdot B_1
+ w R_{11}.
\label{ea2}
\end{multline}

We integrate this equality over $\Omega$. 
The second term on the left hand side is estimated from below by
using the divergence theorem with \eqref{outf} and \eqref{pbc} 
as well as (\ref{sqr}) and (\ref{bdd}):
\begin{equation}
- \int_{\Omega}  \div \bigl\{ w (G_1 + B_1)\bigr\} \, dx  
 =  \int_{\pd \Omega}
(w \rho  \cale)(u_b\cdot n)  \, d \sigma \,
\gtrsim
e^{\sigma t}   \|\vp|_{\partial \Omega} \|^2_{L^2_{x'}}. 
\label{ea3}
\end{equation}
Next we derive the lower estimate of the third term on the
left hand side of (\ref{ea2}).
Taking into account the fact that
 $\eta(s) = \frac{1}{2} (s-1)^2 + O(|s-1|^3)$ for $|s-1| \ll 1$,
we compute the term $\rho u_1 \cale$  in $G_1$ as
\begin{gather}
\rho u_1 \cale
=
\frac{R\theta_{+}u_{+}}{2\rho_{+}} \vp^2
+ \frac{\rho_+ u_+}{2} |\psi|^2
+\frac{c_{v}\rho_{+}u_{+}}{2\theta_{+}} \zeta^{2}
+ R_{12},
\label{ea6}
\\
\begin{aligned}
R_{12}
&:=
(\rho u_1 - \rho_+ u_+) \cale
\\
&\quad + \rho_+ u_+
\left[
\frac{R\tilde\theta}{2} \Bigl( \frac{\rt}{\rho} - 1 \Bigr)^2
-\frac{R\theta_{+}}{2\rho_{+}^{2}} \vp^2
+ R\tilde\theta
\Bigl\{
\eta \Bigl( \frac{\rt}{\rho} \Bigr)
- \frac{1}{2} \Bigl( \frac{\rt}{\rho} - 1 \Bigr)^2
\Bigr\}
\right]
\\
&\quad + \rho_+ u_+
\left[
\frac{c_{v}\tilde\theta}{2} \Bigl( \frac{\theta}{\tilde\theta+\Theta} - 1 \Bigr)^2
-\frac{c_{v}}{2\theta_{+}} \zeta^{2}
+ c_{v}\tilde\theta
\Bigl\{
\eta \Bigl( \frac{\theta}{\tilde\theta+\Theta} \Bigr)
- \frac{1}{2} \Bigl( \frac{\theta}{\tilde\theta+\Theta} - 1 \Bigr)^2
\Bigr\}
\right].
\end{aligned}
\nonumber
\end{gather}
Thus, using (\ref{ea6}), the third term in \eqref{ea2} 
is rewritten as
\begin{gather*}
\nabla w \cdot G_1
=
w_{x_1}
\Bigl(
F_1(\vp,\psi_1, \zeta) + \frac{\rho_+ |u_+|}{2} |\psi'|^2
 + R_{12} - R(\tt - \theta_{+})\vp\psi_{1}-R (\rho -\rho_{+}) \zeta\psi_{1}
\Bigr),
\\
F_1(\vp,\psi_1,\zeta)
:=
\frac{R\theta_{+}|u_{+}|}{2\rho_{+}} \vp^2
+ \frac{\rho_+ |u_+|}{2} \psi_{1}^{2}
+\frac{c_{v}\rho_{+}|u_{+}|}{2\theta_{+}} \zeta^{2}
-R\theta_{+}\vp\psi_{1}-R \rho_{+} \zeta\psi_{1},
\nonumber
\end{gather*}
where $\psi'$ is the second and third components of $\psi$ defined by 
$\psi' := (\psi_2,\psi_3)$.
Owing to the supersonic condition \eqref{super1}, 
the quadratic form $F_1(\vp,\psi_1,\zeta)$ is positive definite.
On the other hand,
the remaining terms satisfy
\begin{equation*}
|R_{12} - R(\tt - \theta_{+})\vp\psi_{1}-R (\rho -\rho_{+}) \zeta\psi_{1}| 
\lesssim  (N_\beta(t) + \dels ) |\Phi|^2.
\end{equation*}
Therefore we obtain the following lower bound of the integration of
 $\nabla w \cdot G_1$ 
\begin{equation}
\int_{\Omega} \nabla w \cdot G_1 \, dx
\geq \beta e^{\sigma t} 
\bigl\{
 c - C (N_\beta(t) + \dels)
\bigr\}
\| \Phi \|^2_{L^2_{e, \beta} (\Omega)}.
\label{ea10}
\end{equation} 
The first and the second terms on the right hand side of (\ref{ea2})
are estimated by using (\ref{sqr}), (\ref{bdd}), and the
Schwarz inequality as
\begin{gather}
\int_{\Omega}
|w_t \rho \cale| \, dx
\lesssim  \sigma e^{\sigma t} \| \Phi \|^2_{L^2_{e, \beta}},  
\label{ea4}
\\
\int_{\Omega}
|\nabla w \cdot B_1| \, dx
\lesssim
\beta e^{\sigma t}
\big(\ep \| (\psi,\zeta) \|^2_{L^2_{e, \beta}} + \ep^{-1}  \| (\nabla \psi,\nabla \zeta) \|^2_{L^2_{e, \beta}} \big),
\label{ea5}
\end{gather}
where $\ep$ is a positive constant to be determined later.
For the term involving $R_{11}$, we observe 
\begin{equation*}
\begin{split}
|R_{11}| 
& \lesssim |(\nabla \tilde{\rho}, \nabla \tilde{u}, \nabla \tilde{\theta},\nabla U,\nabla \Theta)||(\Phi,\nabla \varphi,\nabla \zeta)|^2 + |\zeta||(\nabla \psi,\nabla \zeta)|^{2} +|\Phi||(F,G,H)|\\
&  \lesssim |(\nabla \tilde{\rho},\nabla \tilde{u}, \nabla \tilde{\theta}, \nabla U,\nabla \Theta)||(\Phi,\nabla \varphi,\nabla \zeta)|^2 + |\zeta||(\nabla \psi,\nabla \zeta)|^{2} \\
& \quad +\delta e^{-\alpha x_1}|\Phi|^2 +\delta^{-1}e^{\alpha x_1}|(F,G,H)|^2. 
\end{split}
\end{equation*}
We apply Hardy's inequality \eqref{hardy} to the first and third terms
with $\beta \leq \alpha/2$, \eqref{stdc1}, and \eqref{ExBdry0}, 
and estimate the last term by \eqref{h1}, there then holds that
\begin{equation}
\int_\Omega w |R_{11} | dx 
\lesssim e^{\sigma t}  (N_\beta(t) + \dels) ( \| \nabla \Phi \|^2 +  \|\vp|_{\partial \Omega} \|^2_{L^2_{x'}} +1).
\label{ea9}
\end{equation}

We integrate (\ref{ea2}) over $(0,t) \times \Omega$,
substitute the estimates (\ref{ea3}) and (\ref{ea10})--(\ref{ea9}) into the resulting equality
and then let $\ep$, \footnote{Hereafter we fix this $\beta$ in our whole proof.}{$\beta$}, and $N_\beta(T)+\dels$ be suitably small.
Furthermore, we use the fact that
\begin{align*}
\left\| \frac{d}{dt} \vp  \right\|^2 
= \| \rho \div \psi + f + F\|^2 
\lesssim \| \nabla \psi\|^2+\delta (\| \nabla \vp\|^2+ \|\vp|_{\partial \Omega} \|^2_{L^2_{x'}}+1),
\end{align*}
which follows from \eqref{eq-pv1}, \eqref{stdc1}, \eqref{ExBdry0}, and \eqref{hardy}.
These computations yield the desired inequality.
\end{proof}

\subsection{Time-derivative estimates}

In this section we derive time-derivative estimates.
To this end, we apply the differential operator $\partial_t^k$ for $k=0,1$ to 
(\ref{eq-pv1})--(\ref{eq-pv3}) and obtain the following equations:
\begin{gather}
\partial_t^k \vp_t
+ u \cdot \nabla \partial_t^k \vp
+ \rho \div \partial_t^k \psi
= f_{0,k},
\label{ec1}
\\
\rho \{
\partial_t^k \psi_t
+ (u \cdot \nabla) \partial_t^k \psi
\}
- L(\partial_t^k \psi)
+ R\theta\nabla \partial_t^k \vp+R\rho\nabla \partial_t^k\zeta
= g_{0,k},
\label{ec2}\\
c_v\rho\big(\partial_t^k\zeta_t+{u}\cdot\nabla\partial_t^k\zeta\big)+R\rho\theta\textrm{div}\partial_t^k\psi
-\kappa\Delta\partial_t^k\zeta=h_{0,k},
\label{ec3}
\end{gather}
where
\begin{align*}
f_{0,k}
&:=
\partial_t^k (f + F)
- [\partial_t^k, u] \nabla \vp
- [\partial_t^k, \rho] \div \psi,
\nonumber
\\
g_{0,k}
&:= \partial_t^k (g + G)
- [\partial_t^k, \rho] \psi_t
- [\partial_t^k, \rho u] \nabla \psi
- [\partial_t^k, R\theta] \nabla \vp
- [\partial_t^k, R\rho] \nabla \zeta,
\nonumber
\\
h_{0,k}
&:= \partial_t^k \left(2\mu|\mathcal{D}(\psi)|^2+\lambda(\textrm{div}\psi)^2+h+H\right)
- [\partial_t^k, c_v\rho] \zeta_t
- [\partial_t^k, c_v\rho u] \nabla \zeta
- [\partial_t^k, R\rho\theta] \div \psi.
\end{align*}
Here $[T,u]v := T(uv) - u T v$ is a commutator.
We also often use the following two inequalities:
\begin{gather}
\|\partial_t \varphi\|_{H^1} \lesssim  \sqrt{D_{2,\beta}} + \delta ,
\label{PhiH1}
\\
\|\partial_t\vp\|_{L^\infty} + \|\partial_t\Phi \|_{H^{1}} \lesssim  \|\Phi \|_{H^{3}} + \delta \lesssim  N_\beta(T) + \delta.
\label{PhiSup}
\end{gather}
Indeed, we can derive these from \eqref{eq-pv1}
by using \eqref{stdc1}, \eqref{ExBdry0}, \eqref{h1}, and Hardy's inequality \eqref{hardy}.

We first estimate $\pd_t \Phi$ in the next lemma. 
\begin{lemma}
\label{lm2}
Under the same conditions as in Proposition \ref{apriori1} with $m=3$, 
it holds that
\begin{multline}
e^{\sigma t} \|\pd_t \Phi(t)\|^2
+ \int_0^t e^{\sigma \tau} 
\|(\pd_t \nabla\psi,\pd_t \nabla\zeta)(\tau)\|^2
\, d \tau
\\
\lesssim \|\Phi_0\|_{H^3}^2
+ (N_\beta(T)+\dels+\sigma) \int_0^t e^{\sigma\tau} D_{3,\beta}(\tau) \, d \tau
+ \dels \int_0^t e^{\sigma \tau} d \tau  
\label{ec0}
\end{multline}
for $t \in [0,T]$ and $\sigma>0$, 
where $C$ is a positive constant 
independent of $\delta$, $t$, and $\sigma$.
\end{lemma}

\begin{proof}
Multiplying \eqref{ec1} with $k=1$ by $\frac{R}{\rho} \, \partial_t \vp$, and using the facts that 
${\rho}_t = - {\div} ({\rho} {u})$,
we get
\begin{gather}
\Bigl(
\frac{1}{2}\frac{R}{\rho} | \partial_t \vp |^2
\Bigr)_t
+ \div \Bigl(
\frac{1}{2} \frac{R}{\rho} u
|\partial_t \vp|^2
\Bigr)
+ R \div (\partial_t \psi) \partial_t \vp
=
\frac{R}{\rho} f_{0,1} \, \partial_t \vp
+
\frac{R\div u}{\rho} \, |\partial_t \vp|^2.
\label{ec31}
\end{gather}
Multiply (\ref{ec2}) by $\frac{\partial_t \psi}{\theta}$ successively to get
\begin{gather}
\begin{aligned}
&\Bigl(
\frac{1}{2} \frac{\rho}{\theta} |\partial_t \psi|^2
\Bigr)_t
+ \div B_2
+ \frac{\mu}{\theta}|\nabla \partial_t \psi|^2
+ \frac{\mu+\lambda}{\theta}|\div \partial_t \psi|^2+R\partial_t\psi\cdot\nabla \partial_t\varphi
+\frac{R\rho}{\theta}\partial_t\psi\cdot\nabla\partial_t \zeta
\\
&= g_{0,1} \cdot\frac{ \partial_t \psi}{\theta}
-\frac{\rho\partial_t\theta}{2\theta^2}  |\partial_t \psi|^2- 
\frac{\rho u\cdot\nabla\theta}{2\theta^2}  |\partial_t \psi|^2
+\frac{\mu \partial_t \psi}{\theta^2}\nabla \partial_t \psi\nabla\theta
+\frac{(\mu+\lambda)\div \partial_t \psi}{\theta^2}\partial_t \psi \cdot \nabla\theta,
\end{aligned}
\label{ec4}
\\
\begin{aligned}
B_2 &:=
\frac{1}{2} \frac{\rho u}{\theta} |\partial_t \psi|^2
- \frac{\mu}{\theta} \nabla (\partial_t \psi)  \cdot \partial_t \psi
-\frac{\mu+\lambda}{\theta}\div (\partial_t \psi) \partial_t \psi.
\end{aligned}
\nonumber
\end{gather}
Multiply (\ref{ec3}) by $\frac{\partial_t \zeta}{\theta^2}$ to obtain
\begin{align}
&\Bigl(
\frac{c_v}{2} \frac{\rho}{\theta^2} |\partial_t \zeta|^2\Bigr)_t
+\div\left(c_v|\partial_t\zeta|^2\frac{\rho u}{2\theta^2}-\kappa\frac{\nabla\partial_t\zeta}{\theta^2}\partial_t\zeta\right)
+\kappa\frac{|\nabla\partial_t\zeta|^2}{\theta^2}+\frac{R\rho}{\theta}\div(\partial_t\psi)\partial_t\zeta
\notag
\\
&= h_{0,1}\frac{\partial_t \zeta}{\theta^2}
-\frac{c_v\rho\partial_t\theta}{\theta^3}  |\partial_t \zeta|^2- 
\frac{c_v\rho u\cdot\nabla\theta}{\theta^3}  |\partial_t \zeta|^2
+\frac{2\kappa \partial_t \zeta}{\theta^3}\nabla \partial_t \zeta \cdot \nabla\theta.
\label{ec41}
\end{align}
Summing up  (\ref{ec31})--(\ref{ec41}) yields
\begin{gather}
\Bigl(
\frac{1}{2}\frac{R}{\rho} | \partial_t \vp |^2
+ \frac{1}{2} \frac{\rho}{\theta} |\partial_t \psi|^2
+\frac{c_v}{2} \frac{\rho}{\theta^2} |\partial_t \zeta|^2
\Bigr)_t
+ \div B_3
\mspace{250mu}
\nonumber
\\
\mspace{120mu}
{}+  \frac{\mu}{\theta}|\nabla (\partial_t \psi)|^2
+ \frac{\mu+\lambda}{\theta}|\div (\partial_t \psi)|^2
+\kappa\frac{|\nabla\partial_t\zeta|^2}{\theta^2}
=
R_{21}+R_{22}+R_{23},
\label{ec5}
\\
\begin{aligned}
B_3&:=
\frac{1}{2} \frac{R}{\rho} u
|\partial_t \vp|^2
+ R \partial_t \psi \partial_t \vp+ B_2+c_v|\partial_t\zeta|^2
\frac{\rho u}{2\theta^2}-\kappa\frac{\nabla\partial_t\zeta}{\theta^2}\partial_t\zeta
+\frac{R\rho}{\theta}\partial_t\psi\partial_t\zeta,
\\
R_{21}
&:=
\frac{R}{\rho} f_{0,1} \, \partial_t \vp
+
\frac{R\div u}{\rho} \, |\partial_t \vp|^2
+ g_{0,1} \cdot\frac{ \partial_t \psi}{\theta}
- \frac{\rho u\cdot\nabla\theta}{2\theta^2}  |\partial_t \psi|^2
\\
&\quad +\nabla(\frac{R\rho}{\theta})\cdot\partial_t\psi\partial_t\zeta
-\frac{c_v\rho u\cdot\nabla\theta}{\theta^3}  |\partial_t \zeta|^2+h_{0,1}\frac{\partial_t \zeta}{\theta^2},
\\
R_{22}
&:= -\frac{\rho\partial_t\theta}{2\theta^2}  |\partial_t \psi|^2
-\frac{c_v\rho\partial_t\theta}{\theta^3}  |\partial_t \zeta|^2,
\\
R_{23}
&:=\frac{\mu \partial_t \psi}{\theta^2}\nabla \partial_t \psi\nabla\theta
+\frac{(\mu+\lambda)\div( \partial_t \psi)}{\theta^2}\partial_t \psi \cdot \nabla\theta
+\frac{2\kappa \partial_t \zeta}{\theta^3}\nabla \partial_t \zeta \cdot \nabla\theta.
\end{aligned}
\nonumber
\end{gather}
Owing to \eqref{outf} and (\ref{bdd}),
we have the nonnegativity of the second term on the
left hand side of (\ref{ec5}) as
\begin{gather}
\int_{\Omega}
\div B_3  \, dx
=\int_{\partial\Omega}
\frac{R }{2\rho}
|\partial_t \vp|^2 (u_{b} \cdot n)
\, d\sigma
\geq 0.
\label{ec6}
\end{gather}
Notice that here we used $\partial_t \psi = \partial_t \zeta  = 0$ on $\partial \Omega$, which holds because of \eqref{pbc}. By \eqref{PhiSup} and Sobolev's inequality,
the term $R_{21}$ is estimated as
\begin{equation}
\int_{\Omega}|R_{21}|dx  \lesssim (N_{\beta}(T) + \dels)\|(\Phi_t,\nabla\Phi)\|^2.
\end{equation}
The term $R_{22}$ could be controlled as 
\begin{equation}
\int_{\Omega}|R_{22}|dx  \lesssim
\|\zeta_t\|_{L^2}\|\|(\psi_t,\zeta_{t})\|_{L^4}^2 
\lesssim (N_{\beta}(T) + \dels) \|(\psi_t,\zeta_{t})\|_{H^{1}}^2 ,
\end{equation}
where the last inequality holds due to (\ref{ec3}) and Sobolev's inequality. 
We estimate the term $R_{23}$ as
\begin{equation}
\int_{\Omega}|R_{23}|dx  \lesssim
\|\nabla \theta\|_{L^\infty} \|(\psi_t,\zeta_{t})\|_{L^2}\|(\nabla \partial_t \psi,\nabla\partial_t \zeta)\|_{L^2} 
\lesssim (N_{\beta}(T) + \dels) \|(\psi_t,\zeta_{t})\|_{H^1}^{2}.
\label{ec7}
\end{equation}
Now one can have the desired inequality (\ref{ec0}) as follows.
Multiply (\ref{ec5}) by a time weight function $e^{\zeta t}$, 
integrate the resulting equality over $(0,t) \times \Omega$,
and substitute in (\ref{ec6})--(\ref{ec7}).
Then applying \eqref{PhiH1} and \eqref{PhiSup}
yields the desired inequality (\ref{ec0}).
\end{proof}

Next we estimate $\pd_t^k\nabla\psi$ for $k=0,1$.

\begin{lemma}
\label{lm3}
Under the same conditions as in Proposition \ref{apriori1} with $m=3$, 
it holds that
\begin{align}
{}&
e^{\sigma t} \lt{(\pd_t^k\nabla\psi,\pd_t^k\nabla\zeta)(t)}^2
+ \int_0^t e^{\sigma \tau} \lt{(\pd_t^{k+1}\psi,\pd_t^{k+1}\zeta)(\tau)}^2 \, d \tau
\notag\\
&\lesssim \hs{3}{\Phi_0}^2
+ \ep \calh_{k}^\sigma(t)
+ \ep^{-1} \calp_{k}^\sigma(t)
+ (N_\beta(t)+\dels+\sigma) \int_0^t e^{\sigma \tau} D_{3,\beta}(\tau) \, d \tau
+  \dels \int_0^t e^{\sigma \tau} d \tau 
\label{ed0}
\end{align}
for $t \in [0,T]$, $\sigma>0$, $\ep \in (0,1)$, and $k=0,1$, where $C$ is a positive constant 
independent of $\delta$, $t$, and $\sigma$.
Furthermore, $\calh_{k}^\sigma(t)$ and $\calp_{k}^\sigma(t)$ are defined by
\begin{align*}
\calh_{k}^\sigma(t)
& :=
e^{\sigma t} \lt{\pd_t^k \nabla \vp(t)}^2
+ \int_0^t e^{\sigma \tau} \lt{\pd_t^k \nabla \vp(\tau)}^2 \, d \tau,
\\
\calp_{k}^\sigma(t)
& :=
e^{\sigma t} \lt{(\pd_t^k \psi, \pd_t^k \zeta)(t)}^2
+ \int_0^t e^{\sigma \tau} \lt{(\pd_t^k\nabla \psi, \pd_t^k\nabla \zeta)(\tau)}^2 \, d \tau.
\end{align*}
\end{lemma}

\begin{proof}
Multiplying (\ref{ec2}) by $\partial_t^k \psi_t$, we get
\begin{align}
&
\rho |\partial_t^k \psi_t|^2
+ \rho (u \cdot \nabla)  \partial_t^k \psi \cdot \partial_t^k \psi_t
- L (\partial_t^k \psi) \cdot \partial_t^k \psi_t\nonumber\\
&
+ R\theta \nabla \partial_t^k \vp \cdot \partial_t^k \psi_t
+ R\rho\nabla \partial_t^k \zeta \cdot \partial_t^k \psi_t
= g_{0,k} \cdot \partial_t^k \psi_t.
\label{ed3}
\end{align}
The third and the forth terms on the left hand side of (\ref{ed3}) 
are rewritten as
\begin{align}
&
- L (\partial_t^k \psi) \cdot \partial_t^k \psi_t
=
\Bigl(
\frac{\mu}{2} |\nabla \partial_t^k \psi|^2
+ \frac{\mu+\lambda}{2} |\div \partial_t^k \psi|^2
\Bigr)_t
\nonumber
\\
&
\mspace{150mu}
- \div
\Bigl(
\mu \nabla \partial_t^k \psi \cdot \partial_t^k \psi_t
+ (\mu+\lambda) (\div \partial_t^k \psi )\partial_t^k \psi_t
\Bigr),
\label{ed4}
\end{align}
and
\begin{align}
R\theta\nabla \partial_t^k \vp \cdot \partial_t^k \psi_t
&=
( R\theta \nabla \partial_t^k \vp \cdot \partial_t^k \psi )_t
-  R\partial_t\theta \nabla \partial_t^k \vp \cdot \partial_t^k \psi
- \div( R\theta  \partial_t^k \vp_t  \partial_t^k \psi )
\nonumber
\\
&\quad +(R\nabla\theta   \cdot \partial_t^k \psi
+R\theta   \partial_t^k \div\psi)\cdot(-u \cdot \nabla \partial_t^k \vp-
\rho \div \partial_t^k \psi
+f_{0,k}).\label{ed5}
\end{align}
Multiplying (\ref{ec3}) by $\partial_t^k \zeta_t$, we get
\begin{align}
&
c_v\rho |\partial_t^k \zeta_t|^2-\kappa\div(\nabla\partial_t^k\zeta\partial_t^k\zeta_t)
+\kappa\partial_t\frac{|\nabla\partial_t^k\zeta|^2}{2}\nonumber\\
&
=-c_v\rho {u}\cdot\nabla\partial_t^k\zeta\partial_t^k \zeta_t-R\rho\theta\textrm{div}\partial_t^k\psi\partial_t^k \zeta_t+h_{0,k}\partial_t^k \zeta_t.
\label{ed33}
\end{align}
Substituting (\ref{ed4}) and (\ref{ed5}) into (\ref{ed3}) and adding up the resulting equation and (\ref{ed33}) yield
\begin{equation}
\dt E_{3}
- \div B_4
+ \rho |\partial_t^k \psi_t|^2+c_v\rho |\partial_t^k \zeta_t|^2
= G_3 + R_3,
\label{ed1}
\end{equation}
where $E_3$, $B_4$, $G_3$, and $R_3$ are defined by
\begin{align*}
E_3 := {}
&
\frac{\mu}{2} |\nabla \partial_t^k \psi|^2
+ \frac{\mu + \lambda}{2} |\div \partial_t^k \psi|^2
+  R\theta \nabla \partial_t^k \vp \cdot \partial_t^k \psi 
+\kappa\frac{|\nabla\partial_t^k\zeta|^2}{2} ,
\\
B_4 := {}
&
\mu \nabla \partial_t^k \psi \cdot \partial_t^k \psi_t
+ (\mu + \lambda) \partial_t^k \psi_t \div \partial_t^k \psi
+  R\theta  \partial_t^k \vp_t  \partial_t^k \psi
-\kappa\nabla\partial_t^k\zeta\partial_t^k\zeta_t,
\\
G_3 := {}
&-\rho (u \cdot \nabla) \partial_t^k \psi \cdot \partial_t^k \psi_t
+R\theta   \partial_t^k \div\psi \, (u \cdot \nabla \partial_t^k \vp 
  + \rho \div \partial_t^k \psi) \\
  &
-R\rho\nabla \partial_t^k \zeta \cdot \partial_t^k \psi_t
-c_v\rho {u}\cdot\nabla\partial_t^k\zeta\partial_t^k \zeta_t-R\rho\theta\textrm{div}\partial_t^k\psi\partial_t^k \zeta_t,
\\
R_3 := {}
&
R\partial_t\theta \nabla \partial_t^k \vp \cdot \partial_t^k \psi+
R\nabla\theta   \cdot \partial_t^k \psi
 \, (u \cdot \nabla \partial_t^k \vp 
  + \rho \div \partial_t^k \psi)
  \\
&
-(R\nabla\theta   \cdot \partial_t^k \psi
+R\theta   \partial_t^k \div\psi)\cdot f_{0,k}
+  g_{0,k}  \cdot \partial_t^k \psi_t+h_{0,k}\partial_t^k \zeta_t.
\end{align*}

Owing to $\partial_t \psi =\partial_t \zeta=0$ on $\partial \Omega$, we have  
\begin{gather}
\int_{\Omega}
\div B_4
 \, dx= 0.
\label{ed10}
\end{gather}
For arbitrary positive constants $\ep\le 1$,
the integrals of $E_3$ and $G_3$ over $\Omega$ are estimated as
\begin{gather}
c( \lt{\nabla \partial_t^k \psi}^2+ \lt{\nabla \partial_t^k \zeta}^2)
- \ep \lt{\nabla \partial_t^k \vp}^2
- C \ep^{-1} \lt{\partial_t^k \psi}^2
\le
\int_{\Omega} E_3 \, d x
\lesssim
\ho{\partial_t^k \Phi}^2,
\label{ed6}
\\
\int_{\Omega} |G_3| \, d x
\lesssim
\ep(\lt{\nabla \partial_t^k \vp}^2
+ \lt{\partial_t^k \psi_t}^2+ \lt{\partial_t^k \zeta_t}^2)
+ C\ep^{-1}(\lt{\nabla \partial_t^k \psi}^2+\lt{\nabla \partial_t^k \zeta}^2),
\label{ed2}
\end{gather}
by using \eqref{bdd}.
Using \eqref{PhiSup}, it is straightforward to check that
\begin{equation}
|R_3| \lesssim \left\{
\begin{array}{ll}
(N_\beta(T)+\delta)|(\nabla\Phi,\pd_t \Phi)|^2+|(\rt',\ut',\tt', \nabla U, \nabla \Theta)||\Phi|^2 +\delta^{-1}|(F, G, H)|^2 & \text{if $k=0$,}
\\
(N_\beta(T)+\delta)|(\nabla\Phi,\pd_t \Phi,\nabla^2\Phi,\pd_t \nabla \Phi,\pd_{tt} \psi,\pd_{tt} \zeta)|^2+|\partial_t\theta \nabla \partial_t \vp \cdot \partial_t \psi| & \text{if $k=1$.}
\end{array} 
\right.
\label{ed9}
\end{equation}
It is also seen from \eqref{PhiSup} and Sobolev's inequality \eqref{sobolev0} that
\begin{align*}
\int_{\Omega} |\partial_t\theta \nabla \partial_t \vp \cdot \partial_t \psi| dx
&\lesssim \|\nabla \partial_t \vp\|_{L^{2}}(\|\partial_t\zeta\|_{L^{4}}^{2}+\|\partial_t \psi\|_{L^{4}}^{2})
\\
&\lesssim (N_\beta(T)+\delta)(\|\partial_t\zeta\|_{H^{1}}^{2}+\|\partial_t \psi\|_{H^{1}}^{2}).
\end{align*}

Finally, we multiply (\ref{ed1}) by $e^{\sigma t}$,
integrate the resulting equality over $(0,t) \times \Omega$ and
substitute in the estimates (\ref{ed10})--(\ref{ed9}). Making use of \eqref{PhiH1}, \eqref{PhiSup}, and \eqref{hardy} and letting $\ep$ be suitably small lead to the desired estimate (\ref{ed0}).
\end{proof}

\subsection{Spatial-derivative estimates}\label{Sptial-deriv}

In order to flatten the boundary and obtain tangential derivatives, we introduce the following change of variables:
\begin{equation}\label{CV1}
\Gamma : 
\begin{cases}
& x_1 = y_1 + M(y_2, y_3), \\
& x_2 = y_2, \\
& x_3 = y_3,  \\
\end{cases}
\end{equation} 
and its inverse 
\begin{equation} \label{CV11}
\hat{\Gamma} : 
\begin{cases}
& y_1 = x_1 - M(x_2, x_3), \\
& y_2 = x_2, \\
& y_3 = x_3.  \\
\end{cases}
\end{equation}
We notice that $\Gamma(\mathbb{R}_3^+)=\Omega$, where 
$\mathbb{R}_3^+ : = \{ (y_1, y_2, y_3) \in \mathbb{R}^3 : y_1 >0 \}$.
We set $y' =(y_2, y_3)$ and denote the matrix
\begin{equation}\label{CV2}
A(y') : = 
\left( \begin{array}{ccc} 
1 & 0 & 0   \\
-\partial_{y_2} M(y') & 1 & 0 \\
-\partial_{y_3} M(y') & 0 & 1 \\ 
 \end{array} \right).
\end{equation}

Let us define
\begin{gather*}
\hat{\vp} (t,y) := \vp(t,\Gamma(y)), \quad \hat{\psi} (t,y) : = \psi (t,\Gamma (y)), 
\quad \hat{\zeta} (t,y) : = \zeta (t,\Gamma (y)), 
\\ \hat{\rho} (t,y) := \rho(t,\Gamma(y)), \quad \hat{u} (t,y) := u(t,\Gamma(y)), \quad \hat{U} (t,y) := U(t,\Gamma(y)),
\\
\quad \hat{\theta} (t,y) := \theta(t,\Gamma(y)), \quad \hat{\Theta} (t,y) := \Theta(t,\Gamma(y)).
\end{gather*}
Note that $(\hat{\vp}, \hat{\psi}, \hat\zeta)(t,y)$ is a vector-valued function defined on $\{t \geq 0\} \times \overline{\mathbb{R}_3^+}$, 
and $(\rt,\ut,\tt)(\tilde{M}(\Gamma(y)))=(\rt,\ut,\tt)(y_1)$ holds. We now have
\begin{equation} \label{CV3}
\nabla_x \vp (t,\Gamma(y)) = \hat{\nabla} \hat{\vp}(t,y) := A \nabla_y \hat{\vp} (t,y) = \big(\sum_{j=1}^3 A_{kj} \partial_{y_j} \hat{\vp}  \big)_{k=1, 2, 3} (t,y), 
\end{equation}
\begin{equation} \label{CV7}
\mathcal{D}_x(\psi) (t,\Gamma(y)) = \hat{\mathcal{D}}(\hat\psi)(t,y) := \frac{ \hat{\nabla}\hat\psi+( \hat{\nabla}\hat\psi)^{tr}}{2} (t,y), 
\end{equation}
\begin{equation} \label{CV4}
\div \!{}_x  \psi (t,\Gamma(y)) = \hat{\div} \  \hat{\psi}(t,y) := (A \nabla_y) \cdot \hat{\psi} (t,y) =  \sum_{k=1}^3 \sum_{j=1}^3 A_{kj} \partial_{y_j} \hat{\psi}_k (t,y), 
\end{equation}
\begin{equation} \label{CV5}
\Delta_x  \psi (t,\Gamma(y)) = \hat{\Delta} \hat{\psi}(t,y) := (A \nabla_y) \cdot (A \nabla_y \hat{\psi}) (t,y) =  \sum_{k=1}^3 \big( \sum_{j=1}^3 A_{kj} \partial_{y_j} \big)^2 \hat{\psi} (t,y),
\end{equation}
\begin{equation} \label{CV6}
\frac{d}{dt} \vp (t,\Gamma(y)) = \hat{\frac{d}{dt} } \hat{\vp} (t,y) := \partial_t \hat{\vp}(t,y) + \hat{u} \cdot \hat{\nabla} \hat{\vp} (t,y).
\end{equation}

From \eqref{eq-pv}, we obtain the equation for $(\hat{\vp}, \hat{\psi}, \hat\zeta)$
\begin{subequations}
\label{eq-pv-hat}
\begin{gather}
 \hat{\vp}_t + \hat{u} \cdot \hat{\nabla} \hat{\vp} + \hat{\rho} \hat{\div} \hat{\psi}
= \hat{f} + \hat{F},
\label{eq-pv1-hat}
\\
\hat{\rho} \{ \hat{\psi}_t  + (\hat{u} \cdot \hat{\nabla}) \hat{\psi} \}
- \hat{L} \hat{\psi} + R\hat\theta\hat{\nabla} \hat\vp+R\hat\rho\hat{\nabla}\hat\zeta 
= \hat{g} + \hat{G},
\label{eq-pv2-hat}
\\
c_v\hat\rho\big(\hat\zeta_t+\hat{u}\cdot\hat{\nabla}\hat\zeta\big)+R\hat\rho\hat\theta\hat{\textrm{div}}\hat\psi
-\kappa\hat{\Delta}\hat\zeta=2\mu|\hat{\mathcal{D}}(\hat\psi)|^2+\lambda(\hat{\textrm{div}}\hat\psi)^2+\hat h+\hat H.
\label{eq-pv3-hat}
\end{gather}
and the initial and boundary conditions
\begin{gather}
(\hat{\vp},\hat{\psi}, \hat\zeta)(0,y)
=(\hat{\vp}_0, \hat{\psi}_0, \hat\zeta_0)(y)
= (\rho_0, u_0, \theta_0)(\Gamma(y)) - (\rt, \ut, \tt)(y_1) - (0, U(\Gamma (y)), \Theta(\Gamma(y))),
\label{pic-hat}
\\
\hat{\psi}(t,0,y') = 0, \quad \hat{\zeta}(t,0,y') = 0.
\label{pbc-hat}
\end{gather}
\end{subequations}
Here $\hat{L} \hat{\psi}$, $\hat{f}$, $\hat{F}$, $\hat{g}$, $\hat{G}$, $\hat{h}$ and $\hat{H}$ are defined by
\begin{gather*}
\hat{L} \hat{\psi}(t,y) := \mu \hat{\Delta} \hat{\psi} (t,y) + (\mu + \lambda) \hat{\nabla} \hat{\div} \hat{\psi} (t,y),
\\
\hat{f} (t,y) :=f(t,\Gamma(y)),
\quad
\hat{F} (y) :=F(\Gamma(y)),
\quad
\hat{g} (t,y) := g(t,\Gamma(y)), 
\quad
\hat{G} (y) := G (\Gamma(y)),
\\
\hat{h} (t,y) :=h(t,\Gamma(y)),
\quad
\hat{H} (y) :=H(\Gamma(y)).
\end{gather*}

We now derive the estimate on the spatial-derivatives for the tangential directions. 
\begin{lemma}
\label{lm2-hat}
Under the same conditions as in Proposition \ref{apriori1} with $m=3$, 
it holds that 
\begin{align}
{}&e^{\sigma t} \|\nabla^l_{y'} \hat{\Phi}(t)\|_{L^2(\mathbb R^3_+)}^2
+ \int_0^t e^{\sigma \tau} 
\left(\|  (\nabla\nabla^l_{y'} \hat{\psi}, \nabla\nabla^l_{y'} \hat{\zeta})(\tau)\|_{L^2(\mathbb R^3_+)}^2
+ \left\| \nabla^l_{y'} \hat{\frac{d}{dt}} \hat{\vp} (\tau) \right\|_{L^2(\mathbb R^3_+)}^2 
\right) \, d \tau 
\notag \\
&\lesssim\|\Phi_0\|_{H^3}^2
+ \epsilon \int^t_0 e^{\sigma \tau}\left( \| \nabla \vp(\tau) \|^2_{H^{l-1}} + \|  (\nabla \psi, \nabla\zeta)(\tau) \|^2_{H^{l}} \right) d \tau 
+ \epsilon^{-1} \int^t_0 e^{\sigma \tau}\|  (\nabla \psi, \nabla \zeta)(\tau) \|^2_{H^{l-1}} d \tau 
\notag \\
&\quad + (N_\beta(T)+\dels+\sigma) \int_0^t e^{\sigma \tau} D_{3,\beta}(\tau) \, d \tau
+  \dels   \int_0^t e^{\sigma \tau} d \tau
\label{ec0-hat}
\end{align}
for $t \in [0,T]$, $\sigma>0$, $\epsilon \in (0,1)$, and $l = 1, 2, 3$.
\end{lemma}

\begin{proof}
Applying the differential operator $\nabla^l_{y'}$ to 
(\ref{eq-pv1-hat}) and (\ref{eq-pv2-hat}), we have the following
two equations:
\begin{gather}
\nabla^l_{y'} \hat{\vp}_t
+ \hat{u} \cdot \hat{\nabla} \nabla^l_{y'} \hat{\vp}
+ \hat{\rho} \hat{\div} \nabla^l_{y'} \hat{\psi}
= \hat{f}_{l,0},
\label{ec1-hat}
\\
\hat{\rho} \{
\nabla^l_{y'} \hat{\psi}_t
+ (\hat{u} \cdot \hat{\nabla}) \nabla^l_{y'} \hat{\psi}
\}
- \hat{L}(\nabla^l_{y'} \hat{\psi})
 + R\hat\theta\hat{\nabla}\nabla^l_{y'} \hat\vp+R\hat\rho\hat{\nabla}\nabla^l_{y'}\hat\zeta 
= \hat{g}_{l,0},
\label{ec2-hat}
\\
c_v\hat\rho\big(\nabla^l_{y'}\hat\zeta_t+(\hat{u}\cdot\hat{\nabla})\nabla^l_{y'}\hat\zeta\big)+R\hat\rho\hat\theta\hat{\textrm{div}}\nabla^l_{y'}\hat\psi
-\kappa\hat{\Delta}\nabla^l_{y'}\hat\zeta=\hat h_{l,0},
\label{ec03-hat}
\end{gather}
where
\begin{align}
\hat{f}_{l,0}
&:=
\nabla^l_{y'} (\hat{f} + \hat{F})
- [\nabla^l_{y'}, \hat{u} \cdot \hat{\nabla}] \hat{\vp}
- [\nabla^l_{y'},  \hat{\rho} \hat{\div}] \hat{\psi},\label{fl0-hat}
\\
\hat{g}_{l,0}
&:= \nabla^l_{y'} (\hat{g} + \hat{G})
- [ \nabla^l_{y'}, \hat{\rho}] \hat{\psi}_t  
- [ \nabla^l_{y'}, \hat{\rho} ( \hat{u} \cdot \hat{\nabla})] \hat{\psi}
- [\nabla^l_{y'},\hat{L}] \hat{\psi}
 - [ \nabla^l_{y'} , R\hat\theta\hat{\nabla}] \hat\vp 
- [ \nabla^l_{y'},  R\hat\rho\hat{\nabla}] \hat\zeta ,\label{gl0-hat}
\\
\hat h_{l,0}
&:= \nabla^l_{y'}  \left(2\mu|\hat{\mathcal{D}}(\hat{\psi})|^2+\lambda(\hat{\textrm{div}}\hat{\psi})^2+\hat h+\hat H\right)
- [ \nabla^l_{y'} , c_v\hat{\rho}] \hat{\zeta}_t
- [ \nabla^l_{y'} , c_v\hat{\rho}(\hat u \cdot\hat{\nabla})] \hat{\zeta}\nonumber\\
& \qquad - [ \nabla^l_{y'} , R\hat{\rho}\hat{\theta} \hat{\div}]\hat{ \psi}
+[\nabla^l_{y'},\kappa\hat{\Delta}]\hat\zeta.\label{hl0-hat}
\end{align}
Multiplying \eqref{ec1-hat} by $\frac{R}{\hat{\rho}} \, \nabla^l_{y'} \hat{\vp}$, and using the facts that 
 $\hat{\rho}_t = - \hat{\div} (\hat{\rho} \hat{u})$,
we get
\begin{multline}
\Bigl(
\frac{1}{2}\frac{R}{\hat{\rho}} | \nabla^l_{y'} \hat{\vp} |^2
\Bigr)_t
+ \hat\div \Bigl(
\frac{1}{2} \frac{R}{\hat{\rho}} \hat{u}
|\nabla^l_{y'} \hat{\vp}|^2
\Bigr)
+ R \hat{\div} (\nabla^l_{y'} \hat{ \psi}) \nabla^l_{y'} \hat{\vp}
=
\frac{R}{\hat{\rho}} \, \nabla^l_{y'} \hat{\vp} \hat{f}_{l,0}
+
\frac{R \hat{\div} \hat{ u}}{ \hat{\rho}} \, |\nabla^l_{y'} \hat{\vp}|^2.
\label{ec31hat}
\end{multline}
Multiply (\ref{ec2-hat}) by $\frac{\nabla^l_{y'} \hat{\psi}}{ \hat{\theta}}$ successively to get
\begin{gather}
\Bigl(
\frac{1}{2} \frac{ \hat{\rho}}{ \hat{\theta}} |\nabla^l_{y'} \hat{\psi}|^2
\Bigr)_t
+ \hat{\div}  \hat{B}_2
+ \frac{\mu}{ \hat{\theta}}| \hat{\nabla} (\nabla^l_{y'} \hat{\psi})|^2
+ \frac{\mu+\lambda}{ \hat{\theta}}|\hat{\div} (\nabla^l_{y'} \hat{\psi})|^2
\mspace{200mu} \notag
\\ 
\mspace{200mu}
+R\nabla^l_{y'} \hat{\psi}\cdot\hat{\nabla} \nabla^l_{y'} \hat{\varphi}
+\frac{R\hat{\rho}}{\hat{\theta}}\nabla^l_{y'} \hat{\psi}\cdot\hat{\nabla}\nabla^l_{y'} \hat{ \zeta}
=\hat{R}_{2},
\label{ec32hat}
\\
\begin{aligned}
\hat{B}_2 &:=
\frac{1}{2} \frac{\hat{\rho}\hat{ u}}{\hat{\theta}} |\nabla^l_{y'} \hat{\psi}|^2
- \frac{\mu}{\hat{\theta}} \hat{\nabla} (\nabla^l_{y'} \hat{\psi})  \cdot \nabla^l_{y'} \hat{\psi}
-\frac{\mu+\lambda}{\hat{\theta}}\hat{\div} (\nabla^l_{y'} \hat{\psi}) \nabla^l_{y'} \hat{\psi},
\\
\hat{R}_{2}&:=\hat{g}_{l,0} \cdot \frac{ \nabla^l_{y'} \hat{\psi}}{\hat{\theta}}
 -\frac{\hat{\rho}\partial_t\hat{\theta}}{2\hat{\theta}^2}  |\nabla^l_{y'} \hat{\psi}|^2- 
\frac{\hat{\rho}\hat{ u}\cdot\hat{\nabla}\hat{\theta}}{2\hat{\theta}^2}  |\nabla^l_{y'} \hat{\psi}|^2
+\frac{\mu \nabla^l_{y'} \hat{\psi}}{\hat{\theta}^2}\hat{\nabla}\nabla^l_{y'} \hat{\psi}\hat{\nabla}\hat{\theta}
\\
& \quad +\frac{(\mu+\lambda)\hat{\div}( \nabla^l_{y'} \hat{\psi})}{\hat{\theta}^2}\nabla^l_{y'} \hat{\psi}\cdot \hat{\nabla}\hat{\theta}.
\end{aligned}
\nonumber
\end{gather}
Multiply (\ref{ec03-hat}) by $\hat{\theta}^{-2}{\nabla^l_{y'}\hat\zeta}$ to obtain
\begin{align}
&\Bigl(
\frac{c_v}{2} \frac{\hat{\rho}}{\hat{\theta}^2} |\nabla^l_{y'}\hat\zeta|^2\Bigr)_t+\hat{\div}\left(c_v|\nabla^l_{y'}\hat\zeta|^2
\frac{\hat{\rho}\hat{ u}}{2\hat{\theta}^2}-\kappa\frac{\hat\nabla\nabla^l_{y'}\hat\zeta}{\hat{\theta}^2}\nabla^l_{y'}\hat\zeta\right)+\kappa\frac{|\hat\nabla\nabla^l_{y'}\hat\zeta|^2}{\hat{\theta}^2}+\frac{R\hat{\rho}}{\hat{\theta}}\hat{\div}(\nabla^l_{y'}\hat\psi)\nabla^l_{y'}\hat\zeta
\notag
\\
&= \hat{h}_{l,0}\frac{\nabla^l_{y'}\hat\zeta}{\hat{\theta}^2}-\frac{c_v\hat{\rho}\partial_t\hat{\theta}}{\hat{\theta}^3}  |\nabla^l_{y'}\hat\zeta|^2- 
\frac{c_v\hat{\rho}\hat{ u}\cdot\hat{\nabla}\hat{\theta}}{\hat{\theta}^3}  |\nabla^l_{y'}\hat\zeta|^2+\frac{2\kappa \nabla^l_{y'}\hat\zeta}{\hat{\theta}^3}\hat{\nabla} \nabla^l_{y'}\hat\zeta\hat{\nabla}\hat{\theta}.
\label{ec33hat}
\end{align}
Summing up (\ref{ec31hat})--(\ref{ec33hat}) yields
\begin{gather}
\Bigl(
\frac{1}{2}\frac{R}{\hat{\rho}} | \nabla^l_{y'} \hat{\vp} |^2
+ \frac{1}{2} \frac{ \hat{\rho}}{ \hat{\theta}} |\nabla^l_{y'} \hat{\psi}|^2
+\frac{c_v}{2} \frac{\hat{\rho}}{\hat{\theta}^2} |\nabla^l_{y'}\hat\zeta|^2
\Bigr)_t
+ \hat\div \hat{B}_3
\mspace{270mu}
\nonumber
\\
\mspace{100mu}
{}+  \frac{\mu}{ \hat{\theta}}| \hat{\nabla} (\nabla^l_{y'} \hat{\psi})|^2
+ \frac{\mu+\lambda}{ \hat{\theta}}|\hat{\div} (\nabla^l_{y'} \hat{\psi})|^2
+\kappa\frac{|\hat{\nabla}\nabla^l_{y'}\hat\zeta|^2}{\hat{\theta}^2}
=
\hat{R}_{3}, 
\label{ec5hat}
\\
\begin{aligned}
\hat{B}_3&:=
\frac{1}{2} \frac{R}{\hat{\rho}} \hat{u}
|\nabla^l_{y'} \hat{\vp}|^2
+ R \nabla^l_{y'} \hat{ \psi} \nabla^l_{y'} \hat{\vp}+ \hat{B}_2+c_v|\nabla^l_{y'}\hat\zeta|^2
\frac{\hat{\rho}\hat{ u}}{2\hat{\theta}^2}-\kappa\frac{\hat{\nabla}\nabla^l_{y'}\hat\zeta}{\hat{\theta}^2}\nabla^l_{y'}\hat\zeta
+\frac{R\hat{\rho}}{\hat{\theta}}\nabla^l_{y'} \hat{\psi}\cdot\nabla^l_{y'} \hat{ \zeta},\nonumber
\\
\hat{R}_{3}
&:=
\frac{R}{\hat{\rho}} \, \nabla^l_{y'} \hat{\vp} \hat{f}_{l,0}
+
\frac{R \hat{\div} \hat{ u}}{ \hat{\rho}} \, |\nabla^l_{y'} \hat{\vp}|^2 + \hat{R}_{2}
+\hat{h}_{l,0}\frac{\nabla^l_{y'}\hat\zeta}{\hat{\theta}^2}\nonumber
\\
&\quad -\frac{c_v\hat{\rho}\partial_t\hat{\theta}}{\hat{\theta}^3}  |\nabla^l_{y'}\hat\zeta|^2
-\frac{c_v\hat{\rho}\hat{ u}\cdot\hat{\nabla}\hat{\theta}}{\hat{\theta}^3}  |\nabla^l_{y'}\hat\zeta|^2
+\frac{2\kappa \nabla^l_{y'}\hat\zeta}{\hat{\theta}^3}\hat{\nabla} \nabla^l_{y'}\hat\zeta\hat{\nabla}\hat{\theta}
+\hat{\nabla}\left(\frac{R\hat{\rho}}{\hat{\theta}}\right)\nabla^l_{y'} \hat{\psi}\cdot\nabla^l_{y'} \hat{ \zeta}.
\end{aligned}
\end{gather}
Let us look at the left hand side of \eqref{ec5hat}.
Owing to the divergence theorem with \eqref{outf} and \eqref{pbc-hat},
we have the nonnegativity of the second terms on the
left hand side of (\ref{ec5hat}) as
\begin{align}
\int_{\mathbb{R}^3_+}
\hat{\div} \Bigl(
\frac{1}{2} \frac{R}{\hat{\rho}} \hat{u}
|\nabla^l_{y'} \hat{\vp}|^2 + \hat{B}_2
\Bigr) \, dy
=\int_{\mathbb{R}^2}
\frac{1}{2} \frac{R}{\hat{\rho}}
|\nabla^l_{y'} \hat{\vp}|^2
 (\ub \cdot n) \sqrt{1+|\nabla M|^2}  \, dy' \geq 0.
\label{ec6-hat}
\end{align}
Using the fact that 
$|\nabla \nabla^l_{y'} \hat{\psi}| \lesssim |\hat{\nabla} \nabla^l_{y'} \hat{\psi}|$,
we also have the good contribution from the third and fourth terms in \eqref{ec5hat} as
\begin{equation}\label{ec10-hat}
\int_{\mathbb{R}^3_+}  \mu |\hat{\nabla} \nabla^l_{y'} \hat{\psi}|^2
+ (\mu + \lambda) |\hat{\div} \nabla^l_{y'} \hat{\psi}|^2 \,dy
\gtrsim \int_{\mathbb{R}^3_+} |\nabla \nabla^l_{y'} \hat{\psi}|^2 \,dy.
\end{equation}
Similarly, 
\begin{gather}
\int_{\mathbb{R}^3_+}  \kappa\frac{|\hat\nabla\nabla^l_{y'}\hat\zeta|^2}{\hat{\theta}^2} \,dy
\gtrsim \int_{\mathbb{R}^3_+} |\nabla \nabla^l_{y'} \hat{\zeta}|^2 \,dy.
\end{gather}

We are going to show that $\hat{R}_{3}$ satisfies
\begin{multline}
\int_{\Omega}|\hat{R}_{3}|dx  \lesssim (N_{\beta}(T) + \dels)D_{3,\beta} 
+ \epsilon \left(\|\nabla_y \hat{\vp}\|_{H^{l-1}(\mathbb R^3_+)}^2
+ \|(\nabla_y \hat{\psi}, \nabla_y \hat{\zeta})\|_{H^l(\mathbb R^3_+)}^2\right)
\\
+ \epsilon^{-1} \|(\nabla_y \hat{\psi}, \nabla_y \hat{\zeta})\|_{H^{l-1}(\mathbb R^3_+)}^2 
+ \delta^{-1}\|(\hat{F},\hat{G},\hat{H})\|_{H^3(\mathbb R^3_+)}^2.
\label{ec7-hat}
\end{multline}
Let us first estimate the integrals of 
$\hat{f}_{l,0} \, \nabla^l_{y'} \hat{\vp},~$
$\hat{g}_{l,0} \cdot \nabla^l_{y'} \hat{\psi},~$ and 
$\hat{h}_{l,0} \nabla^l_{y'} \hat{\zeta}$ in $\hat{R}_{3}$, 
where $\hat{f}_{l,0},~$ $\hat{g}_{l,0},~$  and $\hat{h}_{l,0}$ are defined in \eqref{fl0-hat}--\eqref{hl0-hat}.
Noting that $(\rt',\ut')(\tilde{M}(\Gamma(y)))=(\rt',\ut')(y_1)$ and applying
Sobolev's inequalities \eqref{sobolev0}--\eqref{sobolev2} with \eqref{stdc1} and \eqref{ExBdry3},
we have
\begin{align}
&\|(\nabla_{y'}^l \hat{f},\nabla_{y'}^l \hat{g},\nabla_{y'}^l \hat{h})\|_{L^2(\mathbb R^3_+)}\notag\\
& \lesssim \delta  \left( \|\nabla_y \hat{\vp}\|_{H^{2}(\mathbb R^3_+)} + \|(\nabla_y \hat{\psi}, \nabla_y \hat{\zeta})\|_{H^3(\mathbb R^3_+)} \right) 
\notag \\
& \quad + \left\||(\hat{\vp}, \hat{\psi}, \hat{\zeta})| (|(\rt',\ut',\tt')||\nabla^l_{y'} \nabla M| + |\nabla^l_{y'} \nabla U| + |\nabla^l_{y'} \nabla \Theta|)\right\|_{L^2(\mathbb R^3_+)}
\notag\\
& \lesssim \delta  \left( \|\nabla_y \hat{\vp}\|_{H^{2}(\mathbb R^3_+)} + \|(\nabla_y \hat{\psi}, \nabla_y \hat{\zeta})\|_{H^3(\mathbb R^3_+)} \right) 
\notag \\
& \quad + \| (\hat{\vp}, \hat{\psi},\hat\zeta)\|_{L^6(\mathbb R^3_+)} \left( \left\||(\rt',\ut',\tt')| |\nabla^l_{y'} \nabla M| \right\|_{L^3(\mathbb R^3_+)} + \|\nabla^l_{y'} \nabla U \|_{L^3(\mathbb R^3_+)} 
 + \|\nabla^l_{y'} \nabla \Theta \|_{L^3(\mathbb R^3_+)}\right)
\notag\\
&  \lesssim \delta  D_{3,\beta}. 
\label{ec8-hat}
\end{align}
Using \eqref{ec8-hat} and Schwarz's inequality, one can see that the integrals of
$| (\nabla_{y'}^l \hat{f} + \nabla_{y'}^l \hat{F}) \nabla^l_{y'} \hat{\vp}|$,
$|(\nabla_{y'}^l \hat{g} + \nabla_{y'}^l \hat{G})  \cdot \nabla^l_{y'} \hat{\psi}|,$  and 
$|(\nabla_{y'}^l \hat{h} + \nabla_{y'}^l \hat{H})  \nabla^l_{y'} \hat{\zeta}|$ are bounded 
from above by the right hand side of \eqref{ec7-hat}.
Notice that the other terms in $ \hat{f}_{l,0} \, \nabla^l_{y'} \hat{\vp}$,
$\hat{g}_{l,0} \cdot \nabla^l_{y'} \hat{\psi},$ and 
$\hat{h}_{l,0} \nabla^l_{y'} \hat{\zeta}$ are just commutator terms.
Using suitably Lemma \ref{CommEst} with the facts that 
\begin{gather*}
\hat{\rho}(y)= \tilde{\rho}(y_1) + \hat{\vp}(y), \quad 
\nabla^l_{y'} (\rt(y_1) \, \cdot \, ) =\rt(y_1)(\nabla^l_{y'} \, \cdot\, ),
\\
\hat{u}(y)= \tilde{u}(y_1) + \hat{\psi}(y) + \hat{U}(y), \quad
\nabla^l_{y'} (\ut(y_1) \, \cdot \, ) =\ut(y_1)(\nabla^l_{y'} \, \cdot \, ),
\\
\hat{\theta}(y)= \tilde{\theta}(y_1) + \hat{\zeta}(y) + \hat{\Theta}(y), \quad
\nabla^l_{y'} (\tt(y_1) \, \cdot \, ) =\tt(y_1)(\nabla^l_{y'} \, \cdot \, ),
\end{gather*}
we can see that the commutator terms are bounded by the right hand side of \eqref{ec7-hat}.
Now we have completed the estimation of all terms 
in $ \hat{f}_{l,0} \, \nabla^l_{y'} \hat{\vp}$ ,
$\hat{g}_{l,0} \cdot \nabla^l_{y'} \hat{\psi},$ and 
$\hat{h}_{l,0} \nabla^l_{y'} \hat{\zeta}$. 
It is quite straightforward to handle the other terms in $\hat{R}_{3}$ 
with the aid of \eqref{stdc1} and \eqref{ExBdry3}.
Therefore we conclude \eqref{ec7-hat}.

Applying $\nabla^l_{y'}$ to \eqref{eq-pv1-hat}, we arrive at
\begin{equation*}
\nabla^l_{y'} \hat{\frac{d}{dt}} \hat{\vp}=
-\nabla^l_{y'} \big( \hat{\rho} \hat{\div} \hat{\psi} \big)
+ \nabla^l_{y'} (\hat{f} + \hat{F}). 
\end{equation*}
We take the $L^2$-norm and estimate the terms on the right hand side 
with the aid of \eqref{ec8-hat} as

\begin{equation}
\left\| \nabla^l_{y'} \hat{\frac{d}{dt}} \hat{\vp}  \right\|_{L^2(\mathbb R^3_+)}^2 
\lesssim \| {\nabla} \nabla_{y'}^l \hat{\psi}\|_{L^2(\mathbb R^3_+)}^2  
+\delta D_{3,\beta}+\|\nabla_y \hat{\psi}\|_{H^{l-1}(\mathbb R^3_+)}^2
+\|\hat{F}\|_{H^{l}(\mathbb R^3_+)}^2.
\label{ec9-hat}
\end{equation}

We multiply (\ref{ec5hat}) by the time weight function $e^{\sigma t}$, 
integrate the resulting equality over $(0,t) \times \mathbb{R}^3_+$,
and substitute (\ref{ec6-hat})--(\ref{ec7-hat}) into the result. 
Using \eqref{h1} and \eqref{ec9-hat} and
performing the change of variables $y \rightarrow x$ on the right hand side, 
we arrive at the desired inequality (\ref{ec0-hat}). 
This completes the proof of the lemma. 
\end{proof}

Next we estimate the spatial-derivatives for the normal direction.
To simplify the notations, we denote $\partial_j := \partial_{y_j}$ for $j =1,2,3$. 
Applying $\partial_1$ to \eqref{eq-pv1-hat} and multiplying the result 
by $\mu_1 := 2 \mu + \lambda$ yields
\begin{equation}  
\mu_1 \partial_1  \hat{\frac{d}{dt}} \hat{\vp}  + \mu_1 \partial_1 \hat{\rho} \hat{\div} \hat{\psi} + \mu_1 \hat{\rho} \partial_1 \hat{\div} \hat{\psi} = \mu_1 \partial_1 (\hat{f} + \hat{F}) . 
\label{ee1-hat}
\end{equation}
We need to make some cancellation on the term $\mu_1 \hat{\rho} \partial_1 \hat{\div} \hat{\psi}$ so as to avoid the highest order derivative in the normal direction $y_1$. Denote
\begin{align} 
& \mathcal{A}_1 := \frac{\mu_1}{\mu (1 + |\nabla M|^2) + \mu + \lambda}, \quad
 \mathcal{A}_j := -\frac{\partial_j M}{\mu (1 + |\nabla M|^2)}\{\mu_1-(\mu+ \lambda)\mathcal{A}_1\},\quad  j=2,3,
\notag \\
& \tilde{\mathcal{A}}_1:=\mathcal{A}_1-\mathcal{A}_2\partial_2 M-\mathcal{A}_3\partial_3 M>0, \quad
\mathfrak{D} := \tilde{\mathcal{A}}_1\partial_1 + {\mathcal{A}}_2\partial_2 + {\mathcal{A}}_3\partial_3.  
\label{Aj-def}
\end{align}
Taking an inner product of \eqref{eq-pv2-hat} with 
$(\hat{\rho} \mathcal{A}_1, \hat{\rho}\mathcal{A}_2,\hat{\rho} \mathcal{A}_3 )^{tr}$, we obtain
\begin{multline}
\hat{\rho}^2 \left( \sum_{j=1}^3 \mathcal{A}_j \hat{\psi}_{jt} + \sum_{j=1}^3 \mathcal{A}_j (\hat{u} \cdot \hat{\nabla}) \hat{\psi}_j  \right) - \mu \hat{\rho} \sum_{j=1}^3 \mathcal{A}_j \hat{\Delta} \hat{\psi}_j 
\\
- (\mu + \lambda) \hat{\rho} \sum_{j=1}^3 \mathcal{A}_j \partial_j \hat{\div} \hat{\psi} + R\hat{\rho} \hat{\theta} \mathfrak{D}\hat{\vp} + R\hat{\rho}^2 \mathfrak{D}\hat{\zeta}  = \hat{\rho} \sum_{j=1}^3 \mathcal{A}_j (\hat{g}_j + \hat{G}_j), 
\label{ee2-hat}
\end{multline}
where $\hat{g}_j$ and $ \hat{G}_j$ are the $j$-th components of $\hat{g}$ and $\hat{G}$, respectively.
Adding \eqref{ee1-hat} and \eqref{ee2-hat} together gives
\begin{multline}
\mu_1 \partial_1  \hat{\frac{d}{dt}} \hat{\vp}  + \mu_1 \partial_1 \hat{\rho} \hat{\div} \hat{\psi}  + \hat{\rho}^2 \left( \sum_{j=1}^3 \mathcal{A}_j \hat{\psi}_{jt} + \sum_{j=1}^3 \mathcal{A}_j (\hat{u} \cdot \hat{\nabla}) \hat{\psi}_j \right) \\
+ R\hat{\rho} \hat{\theta} \mathfrak{D}\hat{\vp} + R\hat{\rho}^2 \mathfrak{D}\hat{\zeta}  + I + II + III 
 = \mu_1 \partial_1 (\hat{f} + \hat{F} ) + \hat{\rho} \sum_{j=1}^3 \mathcal{A}_j (\hat{g}_j + \hat{G}_j), 
\label{ee3-hat}
\end{multline}
where
\begin{align*}
I &: = \mu_1 \hat{\rho} \partial_1 \hat{\div} \hat{\psi} = \mu_1 \hat{\rho} (\partial_1^2 \hat{\psi}_1
-\partial_2 M \partial_1^2 \hat{\psi}_2 - \partial_3 M \partial_1^2 \hat{\psi}_3)+I', 
\\
II &:= - \mu \hat{\rho} \sum_{j=1}^3 \mathcal{A}_j \hat{\Delta} \hat{\psi}_j = - \mu \hat{\rho} \sum_{j=1}^3 \mathcal{A}_j (1 + |\nabla M|^2)\partial_1^2 \hat{\psi}_j+II',
\\
III &:= - (\mu + \lambda) \hat{\rho} \sum_{j=1}^3 \mathcal{A}_j \partial_j \hat{\div} \hat{\psi} = - (\mu + \lambda) \hat{\rho} \mathcal{A}_1 (\partial_1^2 \hat{\psi}_1
-\partial_2 M \partial_1^2 \hat{\psi}_2 - \partial_3 M \partial_1^2 \hat{\psi}_3) + III',  
\end{align*}
where $I'$, $II'$, and $III'$ do not have terms 
with second-order normal derivative $\partial_{1}^2$.
Due to the choice of $\mathcal{A}_j$, it is straightforward to check that
\begin{equation}
I+II+III= I' + II' + III' = \hat{\rho}\sum_{\substack{1\leq |\bm{b}|\leq 2, \ b_1\neq2, \\ j=1,2,3}} a_{\bm{b}}\partial^{\bm{b}}_y\hat{\psi}_j,
\label{ee4-hat}
\end{equation}
where $a_{\bm{b}}$ denotes scalar-valued functions 
$a_{\bm{b}} = a_{\bm{b}} (\mu, \lambda, \nabla M, \nabla^2 M)$.
Substituting \eqref{ee4-hat} into \eqref{ee3-hat}, multiplying the result by $\tilde{\mathcal{A}}_1$,
and using $\tilde{\mathcal{A}}_1\partial_1=\mathfrak{D}-\mathcal{A}_2\partial_2-\mathcal{A}_3\partial_3$,
we arrive at
\begin{align}
& \mu_1 \mathfrak{D}  \hat{\frac{d}{dt}} \hat{\vp}  - \mu_1(\mathcal{A}_2\partial_2+\mathcal{A}_3\partial_3)  \hat{\frac{d}{dt}} \hat{\vp}  + \mu_1 \tilde{\mathcal{A}}_1 \partial_1 \hat{\rho} \hat{\div} \hat{\psi} + \tilde{\mathcal{A}}_1 \hat{\rho}^2 \left( \sum_{j=1}^3 \mathcal{A}_j \hat{\psi}_{jt} + \sum_{j=1}^3 \mathcal{A}_j (\hat{u} \cdot \hat{\nabla}) \hat{\psi}_j  \right)\notag \\
& \quad  
+ \tilde{\mathcal{A}}_1 (R\hat{\rho} \hat{\theta} \mathfrak{D}\hat{\vp} + R\hat{\rho}^2 \mathfrak{D}\hat{\zeta}) 
 + \tilde{\mathcal{A}}_1 \hat{\rho}\sum_{\substack{1\leq |\bm{b}|\leq 2, \ b_1\neq2, \\ j=1,2,3}} a_{\bm{b}} \partial^{\bm{b}}_y\hat{\psi}_j
\notag \\
& = \mu_1 \tilde{\mathcal{A}}_1 \partial_1 (\hat{f} + \hat{F} ) + \tilde{\mathcal{A}}_1 \hat{\rho} \sum_{j=1}^3 \mathcal{A}_j (\hat{g}_j + \hat{G}_j). 
\label{ee5-hat}
\end{align}

\begin{lemma}
\label{lm4-hat}
Suppose that the same conditions as in Proposition \ref{apriori1} with $m=3$ hold. Define the index $\bm{a} = (a_1, a_2, a_3)$ with $a_1, a_2, a_3 \geq 0$ and $|\bm{a}| := a_1 + a_2 + a_3$. Let $\partial^{\bm{a}} := \partial^{a_1}_{y_1} \partial^{a_2}_{y_2} \partial^{a_3}_{y_3}$. Then it holds that 
\begin{align}
& e^{\sigma t} \| \partial^{\bm{a}} \partial_1 \hat{\vp} (t) \|_{L^2(\mathbb R^3_+)}^2 
+ \int^t_0 e^{\sigma \tau}\left(  \|\partial^{\bm{a}} \mathfrak{D} \hat{\vp} (\tau) \|_{L^2(\mathbb R^3_+)}^2 
+ \left\| \partial^{\bm{a}} \partial_1 \hat{\frac{d}{dt}} \hat{\vp}(\tau) \right\|_{L^2(\mathbb R^3_+)}^2  \right) \, d\tau 
\notag\\
&  \lesssim  \|\vp_0\|_{H^3}^2 
+ \int_0^t e^{\sigma \tau} \left(
\left\|\partial^{\bm{a}} \nabla_{y'} \hat{\frac{d}{dt}} \hat{\vp}(\tau) \right\|_{L^2(\mathbb R^3_+)}^2 
+ \|\partial^{\bm{a}} \nabla_y \nabla_{y'} \hat{\psi}(\tau)\|_{L^2(\mathbb R^3_+)}^2 \right)\, d \tau
\notag\\
&\quad + \int_0^t e^{\sigma \tau} \left(|\bm{a}|\|\nabla {\vp}(\tau) \|^2_{H^{|\bm{a}| -1}}
+|\bm{a}|\left\|\nabla {\frac{d}{dt}} {\vp}(\tau) \right\|_{H^{|\bm{a}|-1}}^2
+\|{\psi}_t(\tau) \|^2_{H^{|\bm{a}|}} 
+\| (\nabla{\psi},\nabla\zeta)(\tau)\|^2_{H^{|\bm{a}|}} \right)\, d \tau
\notag\\
& \quad 
+ (N_\beta (T)+\dels+\sigma) \int_0^t e^{\sigma \tau} D_{3, \beta}(\tau) \, d \tau 
+ \dels\int_0^t e^{\sigma \tau} d \tau
\label{ee0-hat}
\end{align}
for $t \in [0, T]$, $\sigma > 0$, and $0\leq |\bm{a}| \leq 2$.
\end{lemma}

\begin{proof}
Applying $\partial^{\bm{a}}$ to \eqref{ee5-hat} yields
\begin{equation}
\mu_1 \partial^{\bm{a}} \mathfrak{D} \hat{\frac{d}{dt}} \hat{\vp} +\tilde{\mathcal{A}}_1 R\hat{\rho} \hat{\theta}\partial^{\bm{a}} \mathfrak{D}\hat{\vp} 
= \hat{I}_1 ,
\label{ee6-hat}
\end{equation}
where 
\begin{equation*}
\begin{split}
 \hat{I}_1 := 
&  -[\partial^{\bm{a}},\tilde{\mathcal{A}}_1 R\hat{\rho} \hat{\theta}] \mathfrak{D}\hat{\vp}
\\
& + \partial^{\bm{a}} \Bigg\{\mu_1(\mathcal{A}_2\partial_2+\mathcal{A}_3\partial_3) \hat{\frac{d}{dt}} \hat{\vp} - \mu_1 \tilde{\mathcal{A}}_1 \partial_1 \hat{\rho} \hat{\div} \hat{\psi} - \tilde{\mathcal{A}}_1 \hat{\rho}^2 \left( \sum_{j=1}^3 \mathcal{A}_j \hat{\psi}_{jt} + \sum_{j=1}^3 \mathcal{A}_j \hat{u} \cdot \hat{\nabla} \hat{\psi}_j \right) 
\\
&  \qquad + \tilde{\mathcal{A}}_1 (R\hat{\rho}^2 \mathfrak{D}\hat{\zeta})  + \tilde{\mathcal{A}}_1 \hat{\rho}\sum_{\substack{1\leq |\bm{b}|\leq 2, \ b_1\neq2, \\ j=1,2,3}} a_{\bm{b}} \partial^{\bm{b}}_y\hat{\psi}_j + \mu_1 \tilde{\mathcal{A}}_1 \partial_1 (\hat{f} + \hat{F} ) + \tilde{\mathcal{A}}_1 \hat{\rho} \sum_{j=1}^3 \mathcal{A}_j (\hat{g}_j + \hat{G}_j) \Bigg\} . \\
\end{split}
\end{equation*}
Multiplying \eqref{ee6-hat} by $\partial^{\bm{a}}\mathfrak{D} \hat{\vp}$ and $\partial^{\bm{a}} \mathfrak{D} \hat{\frac{d}{dt}} \hat{\vp}$, respectively, and adding up the two resulting equalities, we obtain
\begin{multline}
 \left(\frac{1}{2} \mu_1 |\partial^{\bm{a}}\mathfrak{D} \hat{\vp}|^2 + \frac{1}{2} \tilde{\mathcal{A}}_1 R\hat{\rho} \hat{\theta} |\partial^{\bm{a}} \mathfrak{D} \hat{\vp}|^2  \right)_t + \mu_1 \left|\partial^{\bm{a}} \mathfrak{D} \hat{\frac{d}{dt}} \hat{\vp}\right|^2 + \tilde{\mathcal{A}}_1 R\hat{\rho} \hat{\theta} |\partial^{\bm{a}} \mathfrak{D} \hat{\vp}|^2 \\
+ \hat{\div}\left(\frac{1}{2} \mu |\partial^{\bm{a}}\mathfrak{D} \hat{\vp}|^2 \hat{u} + \frac{1}{2} \tilde{\mathcal{A}}_1 R\hat{\rho} \hat{\theta} |\partial^{\bm{a}} \mathfrak{D} \hat{\vp}|^2\hat{u}  \right)
= \hat{R}_3, 
\label{ee7-hat}
\end{multline}
where
\begin{equation*}
\begin{split}
\hat{R}_3&:=\big\{(\hat{u} \cdot \hat{\nabla} \partial^{\bm{a}} \mathfrak{D} \hat{\vp} ) - \partial^{\bm{a}}\mathfrak{D} (\hat{u} \cdot \hat{\nabla} \hat{\vp}) \big\} (\mu_1 \partial^{\bm{a}}\mathfrak{D} \hat{\vp} + \tilde{\mathcal{A}}_1 R\hat{\rho} \hat{\theta}\partial^{\bm{a}} \mathfrak{D} \hat{\vp})
\\
& \quad + \frac{1}{2}\mu_1 |\partial^{\bm{a}}\mathfrak{D} \hat{\vp}|^2 \hat{\div} \hat{u}  
+ \frac{1}{2} \tilde{\mathcal{A}}_1 R(\hat{\rho}\hat{\theta} )_t|\partial^{\bm{a}} \mathfrak{D} \hat{\vp}|^2 
+ \frac{1}{2} |\partial^{\bm{a}} \mathfrak{D} \hat{\vp}|^2 \hat{\div} \left(\tilde{\mathcal{A}}_1 R\hat{\rho} \hat{\theta} \hat{u}\right) \\
& \quad + \hat{I}_1 \left(\partial^{\bm{a}} \mathfrak{D} \hat{\vp} + \partial^{\bm{a}} \mathfrak{D} \hat{\frac{d}{dt}} \hat{\vp} \right) . 
\end{split}
\end{equation*}

Let us estimate the integrals of some terms in \eqref{ee7-hat}. 
We first find the good contribution for $\partial^{\bm{a}} \partial_1 \hat{\frac{d}{dt}} \hat{\vp}$
from the second and third terms on the left hand side. 
Indeed, using the fact $\partial_1=\tilde{\mathcal{A}}_1^{-1}\mathfrak{D}-\tilde{\mathcal{A}}_1^{-1}\mathcal{A}_2\partial_2-\tilde{\mathcal{A}}_1^{-1}\mathcal{A}_3\partial_3$, 
we see that
\begin{gather} 
\int_{\mathbb R^3_+} \left|\partial^{\bm{a}} \partial_1 \hat{\frac{d}{dt}} \hat{\vp}\right|^2 \!dy 
\lesssim \int_{\mathbb R^3_+}
 \mu_1\left|\partial^{\bm{a}}\mathfrak{D} \hat{\frac{d}{dt}} \hat{\vp}\right|^2 \!dy  
+\left\|\partial^{\bm{a}} \nabla_{y'} \hat{\frac{d}{dt}} \hat{\vp}\right\|_{L^2(\mathbb R^3_+)}^2 
\!+|\bm{a}| \left\|\hat{\frac{d}{dt}} \hat{\vp} \right\|_{H^{|\bm{a}|}(\mathbb R^3_+)}^2.
\label{ee10-hat}
\end{gather}
Owing to the {Divergence} Theorem with \eqref{outf} and \eqref{pbc},
we have the nonnegativity of the fourth term on the left hand side as
\begin{align}
&{}
\int_{\mathbb{R}^3_+}
\hat{\div}\left(\frac{1}{2} \mu_1 |\partial^{\bm{a}}\mathfrak{D} \hat{\vp}|^2 \hat{u} + \frac{1}{2} \tilde{\mathcal{A}}_1 R\hat{\rho} \hat{\theta} |\partial^{\bm{a}}\mathfrak{D} \hat{\vp}|^2\hat{u}  \right) dy
\notag \\
&=\int_{\mathbb{R}^2}
\frac{\ub \cdot n}{2}
\left(\mu_1 |\partial^{\bm{a}}\mathfrak{D} \hat{\vp}|^2+\tilde{\mathcal{A}}_1 R\hat{\rho} \hat{\theta} |\partial^{\bm{a}}\mathfrak{D} \hat{\vp}|^2\right)
 \sqrt{1+|\nabla M|^2} \, dy' \geq 0.
\label{ee8-hat}
\end{align}
Furthermore, we claim that the integral of $\hat{R}_3$ in \eqref{ee7-hat} is estimated as
\begin{align}
\int_{\mathbb R^3_+} |\hat{R}_3 | \,dy 
& \lesssim (\epsilon + N_\beta(T) + \delta )  \| \partial^{\bm{a}} \partial_1 \hat{\vp} \|_{L^2(\mathbb R^3_+)}^2 
+( \epsilon + N_\beta(T) + \delta ) \left\|\partial^{\bm{a}} \partial_1 \hat{\frac{d}{dt}} \hat{\vp}  \right\|_{L^2(\mathbb R^3_+)}^2
\notag\\
&\quad +( N_\beta(T) + \delta ) D_{3,\beta} 
+\epsilon^{-1}\left\|\partial^{\bm{a}} \nabla_{y'} \hat{\frac{d}{dt}} \hat{\vp} \right\|_{L^2(\mathbb R^3_+)}^2
+\epsilon^{-1}\|\partial^{\bm{a}} \nabla_y \nabla_{y'} \hat{\psi}\|_{L^2(\mathbb R^3_+)}^2
\notag\\
&\quad +\epsilon^{-1}|\bm{a}|\|\nabla_y \hat{\vp} \|_{H^{|\bm{a}| -1}(\mathbb R^3_+)}^{2}
+\epsilon^{-1}|\bm{a}|\left\|\hat{\frac{d}{dt}} \hat{\vp}\right\|_{H^{|\bm{a}|}(\mathbb R^3_+)}^{2}
+\epsilon^{-1}\| \hat{\psi}_t \|_{H^{|\bm{a}|}(\mathbb R^3_+)}^{2}
\notag\\
&\quad +\epsilon^{-1}\| (\nabla_y\hat{\psi},\nabla_y\hat\zeta)\|_{H^{|\bm{a}|}(\mathbb R^3_+)}^{2}
+\epsilon^{-1}\delta,
\label{ee11-hat}
\end{align}
where $\epsilon$ is a positive constant to be determined later.
To show this, we start from the estimation of $\hat{I}_1$:
\begin{align}
& \| \hat{I}_1\|_{L^2(\mathbb R^3_+)} 
\notag \\
& \lesssim
( N_\beta(T) + \delta )  D_{3,\beta}^{1/2} 
+\left\|\partial^{\bm{a}} \nabla_{y'} \hat{\frac{d}{dt}} \hat{\vp} \right\|_{L^2(\mathbb R^3_+)}\!
+\|\partial^{\bm{a}} \nabla_y \nabla_{y'} \hat{\psi}\|_{L^2(\mathbb R^3_+)}
+|\bm{a}|\|\nabla_y \hat{\vp} \|_{H^{|\bm{a}| -1}(\mathbb R^3_+)}
\notag \\
&\quad +|\bm{a}|\left\|\hat{\frac{d}{dt}} \hat{\vp}\right\|_{H^{|\bm{a}|}(\mathbb R^3_+)} \!
+\| \hat{\psi}_t \|_{H^{|\bm{a}|}(\mathbb R^3_+)} 
+\|(\nabla_y\hat{\psi},\nabla_y\hat\zeta)\|_{H^{|\bm{a}|}(\mathbb R^3_+)}
+\|(\hat{F},\hat{G})\|_{H^2(\mathbb R^3_+)}.
\label{ee12-hat}
\end{align}
It is straightforward to check that all terms except those involving $\hat{f}$ and $\hat{g}$
can be estimated by the right hand side of \eqref{ee12-hat}.
Let us handle the terms involving $\hat{f}$ and $\hat{g}$.
Using Hardy's inequality \eqref{hardy} and
Sobolev's inequalities \eqref{sobolev1} and \eqref{sobolev2}
together with \eqref{stdc1} and \eqref{ExBdry3}, we have
\begin{align}
{}& \| \mu_1 \partial^{\bm{a}}  \tilde{\mathcal{A}}_1 \partial_1 \hat{f}\|_{L^2(\mathbb R^3_+)}
+ \left\| \partial^{\bm{a}} \tilde{\mathcal{A}}_1  \hat{\rho} \sum_{j=1}^3 \mathcal{A}_j \hat{g}_j \right\|_{L^2(\mathbb R^3_+)} 
\notag \\
&\lesssim \delta \left( \|\nabla\hat{\vp} \|_{H^2(\mathbb R^3_+)} 
+ \|(\nabla\hat{\psi},\nabla\hat\zeta) \|_{H^3(\mathbb R^3_+)}\right)
\notag \\
&\quad + \sum_{i=1}^{|\bm{a}|+1}\left\||(\rt^{(i)},\ut^{(i)},\tt^{(i)})||\hat{\Phi}|\right\|_{L^2(\mathbb R^3_+)}\!
+ \|\hat{\Phi}\|_{L^6(\mathbb R^3_+)} \sum_{i=0}^{|\alpha|+1} \|(\nabla^i \hat{U},\nabla^i \hat{\Theta})\|_{L^3(\mathbb R^3_+)}
\notag \\
&\lesssim \delta D_{3,\beta}^{1/2}. 
\label{ee13-hat}
\end{align}
Therefore we conclude that \eqref{ee12-hat} holds.
We now can estimate the integral of $\hat{R}_3$.
It is easy to show by using \eqref{ee12-hat}, Schwarz's inequality, and Sobolev's inequality \eqref{sobolev2} that the last four terms in $\hat{R}_3$ are bounded by the right hand side of \eqref{ee11-hat}. It remains to handle only the first term, that is, the commutator term.
The $L^2$-norm of the commutator $\hat{u} \cdot \hat{\nabla} \partial^{\bm{a}} \mathfrak{D} \hat{\vp} - \partial^{\bm{a}} \mathfrak{D} (\hat{u} \cdot \hat{\nabla} \hat{\vp} )$ can be estimated as
\begin{align*}
&\|\hat{u} \cdot \hat{\nabla} \partial^{\bm{a}} \mathfrak{D} \hat{\vp} - \partial^{\bm{a}} \mathfrak{D} (\hat{u} \cdot \hat{\nabla} \hat{\vp} )\| 
\notag\\
&\leq \|\ut_1 \partial_1 \partial^{\bm{a}} \mathfrak{D} \hat{\vp} - \partial^{\bm{a}} \mathfrak{D}(\ut_1\partial_1 \hat{\vp} )\|
+\|(\hat{\psi}+\hat{U}) \cdot \hat{\nabla} \partial^{\bm{a}} \mathfrak{D} \hat{\vp} - \partial^{\bm{a}} \mathfrak{D} ((\hat{\psi}+\hat{U}) \cdot \hat{\nabla} \hat{\vp} )\| 
\notag\\
&= \|(\ut_1-u_+) \partial^{\bm{a}}\mathfrak{D} \partial_1 \hat{\vp} - \partial^{\bm{a}} \mathfrak{D} ((\ut_1-u_+)\partial_1 \hat{\vp} )\|
+\|(\hat{\psi}+\hat{U}) \cdot \hat{\nabla} \partial^{\bm{a}} \mathfrak{D} \hat{\vp} - \partial^{\bm{a}} \mathfrak{D} ((\hat{\psi}+\hat{U}) \cdot \hat{\nabla} \hat{\vp} )\| 
\notag\\
&\lesssim (N_\beta(T)+\delta)D_{3,\beta}^{1/2},
\end{align*}
where we have written explicitly $\hat{u}$ and used the triangular inequality in deriving the first inequality; we have expanded the derivative operators and applied Sobolev's inequalities \eqref{sobolev0} and \eqref{sobolev2} to derive the last inequality.
Using this, one can check that the integrals of the first two terms in $\hat{R}_3$ 
is also bound by the right hand side of \eqref{ee11-hat}.
Therefore we conclude that \eqref{ee11-hat} holds.


We multiply (\ref{ee7-hat}) by the time weight function $e^{\sigma t}$, 
integrate the resulting equality over $(0,t) \times \mathbb{R}^3_+$,
substitute (\ref{ee8-hat}) and (\ref{ee11-hat}) into the result,
let $\epsilon + N_\beta(T) + \delta$ be small enough,
and use (\ref{ee10-hat}).
Performing the change of variables $y \rightarrow x$ for some terms on right hand side, 
we arrive at the desired inequality (\ref{ee0-hat}). 
This completes the proof of the lemma.
\end{proof}

\subsection{Cattabriga estimates} \label{ss-CattabrigaEst}

To complete the a priori estimate, we apply the Cattabriga estimate in Lemma \ref{CattabrigaEst}.
We remark that the Cattabriga estimate has crucial dependence on $\Omega$. The other estimates rely on Hardy's inequality, Sobolev's inequalities, Gagliardo-Nirenberg inequality and the commutator estimates, which depend on Sobolev's norms of $M$. Recall that in Subsection \ref{ss-Notation} all the constants depend on the Sobolev's norms of $M$. 

\begin{lemma} \label{lm1-C}
Under the same assumption as in Proposition \ref{apriori1} with $m=3$,
it holds, for $k=0, 1, 2$, 
\begin{align}
&\| \nabla^{k+2} \psi \|^2 + \| \nabla^{k+1} \vp \|^2 
\notag\\
& \lesssim_{\Omega} \| \psi_t \|^2_{H^{k}} 
+ \| (\nabla \psi,\nabla \zeta)\|^2_{H^{k}} 
+ \left\| \frac{d}{dt} \vp \right\|^2_{H^{k+1}} \!
+ (N_\beta (T) + \delta ) D_{3,\beta}+\delta.
\label{ef0-C}
\end{align}
\end{lemma}
\begin{proof}
From \eqref{eq-pv}, and recalling $\frac{d}{dt} = \partial_t + u \cdot \nabla $, we obtain
a boundary value problem of the Stokes equation:
\begin{align}
 \rho_+\div\psi = V , \quad
 - \mu \Delta \psi + R\theta_+\nabla \vp 
 = W,\quad
 \psi|_{\partial \Omega} =0 , \quad \lim_{|x| \rightarrow \infty} |\psi| =0,\label{ef7-C} 
\end{align}
where
\begin{align*}
V : =& f + F - \frac{d}{dt}\vp - (\rho-\rho_+)\div \psi , \\
W : =& - \rho \{ \psi_t + ( u \cdot \nabla ) \psi \} + (\mu + \lambda) \nabla \div \psi + g + G -R (\theta-\theta_+) \nabla \vp- R\rho \nabla \zeta
\\
=& - \rho \{ \psi_t + ( u \cdot \nabla ) \psi \} + (\mu + \lambda) \rho_+^{-1} \nabla V + g + G -R (\theta-\theta_+) \nabla \vp- R\rho\nabla \zeta.
\end{align*}

Applying the Cattabriga estimate \eqref{Cattabriga} to problem \eqref{ef7-C}, we have 
\begin{equation*}
\| \nabla^{k+2} \psi \|^2 + \| \nabla^{k+1} \vp \|^2 \lesssim_{\Omega} \| V \|^2_{H^{k+1}} + \| W \|^2_{H^k} + \| \nabla \psi \|^2.
\end{equation*}
It is straightforward to show that $\| V \|^2_{H^{k+1}}$ and $\| W \|^2_{H^k}$ are bound by the right hand side of \eqref{ef0-C}. Indeed, we can use the same method as in the derivation of \eqref{ee13-hat} to estimates the terms $f$ and $g$. The other terms can be estimated by using \eqref{h1} and Sobolev's inequalities \eqref{sobolev0} and \eqref{sobolev2}. Therefore we conclude \eqref{ef0-C}.
\end{proof}

We also show similar estimates for $(\hat{\vp},\hat{\psi})(t,y)$, where $y \in \mathbb R^3_+$.
For notational convenience, we denote
\begin{equation}\label{checkOp}
\check{\partial}_{y_j}: = \sum_{i=1}^3 (A^{-1}(x'))_{ij} \partial_{x_i}. 
\end{equation}
where $A$ is defined in \eqref{CV2}; $(A^{-1}(x'))_{ij}$ means the $(i,j)$-component of $A^{-1}(x')$;
$\check{\partial}_{y_j} \vp (t,\Gamma(y)) = \partial_{y_j} \hat{\vp}(t,y)$ holds.
Furthermore, $\check{\nabla}^l_{y'}$ means the totality of all $l$-times tangential derivatives 
$\check{\partial}_{y_j}$ only for $j=2,3$.
Then applying $\check{\nabla}^l_{y'}$ to \eqref{ef7-C}, we obtain a boundary value problem of the Stokes equation: 
\begin{equation}\label{checkEq}
\rho_+\div \check{\nabla}^l_{y'} \psi = \check{V}, \ \
- \mu \Delta \check{\nabla}^l_{y'} \psi +R\theta_+\nabla \check{\nabla}^l_{y'} \vp = \check{W}, \ \  
\check{\nabla}^l_{y'} \psi|_{\partial \Omega} =0 , \ \ \lim_{|x| \rightarrow \infty} | \check{\nabla}^l_{y'} \psi| =0,
\end{equation}
where
\begin{align*}
\check{V} : = 
& \check{\nabla}^l_{y'} (f+F) - \check{\nabla}^l_{y'} \frac{d}{dt}\vp - \check{\nabla}^l_{y'} \Big((\rho-\rho_+)\div \psi \Big) - \rho_{+}[\check{\nabla}^l_{y'},\div] \psi, \\
\check{W} : = 
&- \check{\nabla}^l_{y'}  \Big(\rho \{ \psi_t + ( u \cdot \nabla ) \psi \} \Big) + (\mu + \lambda) \check{\nabla}^l_{y'}  \nabla \div \psi + \check{\nabla}^l_{y'} g + \check{\nabla}^l_{y'} G \\
& - \check{\nabla}^l_{y'} \Big( R (\theta-\theta_+) \nabla \vp + R\rho\nabla \zeta\Big) 
+ \mu [\check{\nabla}^l_{y'},\Delta]\psi 
- R \theta_+ [\check{\nabla}^l_{y'}, \nabla]\vp.
\end{align*}

\begin{lemma} \label{lm2-C}
Under the same assumption as in Proposition \ref{apriori1} with $m=3$,
it holds, for $k=0,  1$, $k+l = 1, 2$, 
\begin{align}
&\| \nabla^{k+2}  \nabla^l_{y'} \hat{\psi} \|_{L^2(\mathbb R^3_+)}^2 
+ \|\nabla^{k+1} \nabla^l_{y'} \hat{\vp} \|_{L^2(\mathbb R^3_+)}^2 
\notag \\
& \lesssim_{\Omega} 
\left\|{\nabla}^l_{y'} \hat{\frac{d}{dt}} \hat{\vp} \right\|^2_{H^{k+1}(\mathbb R^3_+)} \!
+ \| \psi_t \|^2_{H^{k+l}} 
+ \| (\nabla \psi,\nabla \zeta) \|^2_{H^{k+l}} 
+ \| \nabla \vp \|^2_{H^{k+l-1}}
+ (N_\beta (T) + \delta ) D_{3,\beta}+ \delta.
\label{eg0-C}
\end{align}
\end{lemma}
\begin{proof}
We apply the Cattabriga estimate \eqref{Cattabriga} to \eqref{checkEq} and obtain 
\begin{equation}
\| \nabla^{k+2} \check{\nabla}^l_{y'} \psi \|^2 + \|\nabla^{k+1} \check{\nabla}^l_{y'} \vp \|^2 \lesssim_{\Omega} \|\check{V} \|^2_{H^{k+1}} + \|\check{W}\|^2_{H^k} + \|\nabla \check{\nabla}^l_{y'} \psi \|^2.
\label{eg1-C}
\end{equation}
The terms $\|\check{V} \|^2_{H^{k+1}}$ and $\|\check{W}\|^2_{H^k}$ can be estimated 
in the same way as in the proof of Lemma \ref{lm1-C}. Now we conclude from \eqref{eg1-C} that  
\begin{align*}
&\| \nabla^{k+2} \check{\nabla}^l_{y'} \psi \|^2 
+ \|\nabla^{k+1} \check{\nabla}^l_{y'} \vp \|^2  
\\
&\lesssim_{\Omega} 
\left\|\check{\nabla}^l_{y'} \frac{d}{dt} \vp \right\|^2_{H^{k+1}} \!
+ \| \psi_t \|^2_{H^{k+l}} 
+ \| (\nabla \psi,\nabla \zeta) \|^2_{H^{k+l}} 
+ \| \nabla \vp \|^2_{H^{k+l-1}}
+ (N_\beta (T) + \delta ) D_{3,\beta}+ \delta.
\end{align*}
Then, changing the coordinate $x \in \Omega$ to the coordinate $y \in \mathbb R_+^3$ 
in the left hand side of this inequality 
and also using $\check{\partial}_{y_j} \vp (t,\Gamma(y)) = \partial_{y_j} \hat{\vp}(t,y)$, 
we arrive at \eqref{eg0-C}.
\end{proof}

\subsection{Elliptic estimates} \label{EllipticEst}

Using the elliptic estimate (Lemmas \ref{ellipticEst} and \ref{ellipticEst2}),
we rewrite some terms for the time-derivatives into terms for the spatial-derivatives.

\begin{lemma} \label{lm1-E}
Under the same assumption as in Proposition \ref{apriori1} with $m=3$, it holds that
\begin{align}
\| \nabla^{k+2} \psi \| & \lesssim  \| \psi_t \|_{H^{k}} + E_{k+1,\beta}+ \delta, \quad k=0,1,
\label{eh0-E} \\
\| \nabla^{k+2} \zeta \| & \lesssim  \| \zeta_t \|_{H^{k}} + E_{k+1,\beta}+ \delta, \quad k=0,1,
\label{eh0-Ez} \\
\| \nabla^{k+2} \zeta \| & \lesssim  \| \zeta_t \|_{H^{k}} + D_{k,\beta}+ \delta, \quad k=0,1,2,
\label{eh0-Ezz} \\
\| \nabla^{2}   \psi_t \| & \lesssim  \|\psi_{tt} \| +D_{2,\beta},
\label{ej0-E} \\
\| \nabla^{2}  \zeta_t \| & \lesssim  \|\zeta_{tt} \| +D_{2,\beta}.
\label{ej0-Ez}
\end{align}
\end{lemma}
\begin{proof}
Let us first show \eqref{eh0-E} and \eqref{eh0-Ez} together.
From \eqref{eq-pv}, we have the elliptic boundary value problem:
\begin{subequations}
\label{eh1-E}
\begin{align} 
&-\mu \Delta \psi -(\mu + \lambda) \nabla \div \psi = - \rho \psi_t  + \widetilde{ G}, \quad 
\psi|_{\partial\Omega}=0, \quad \lim_{|x|\to\infty}\psi=0,\label{eh1-E1}\\
&-\kappa\Delta\zeta=-c_v\rho \zeta_t+\widetilde H,\quad
\zeta|_{\partial\Omega}=0, \quad \lim_{|x|\to\infty}\zeta=0,\label{eh1-E2}
\end{align}
\end{subequations}
where
\begin{align*}
&\widetilde{ G} := - R\theta \nabla \vp -\rho (u \cdot \nabla) \psi-
R\rho  \nabla\zeta +g +G, \\
&\widetilde H:=-c_v\rho{u}\cdot\nabla\zeta-R\rho\theta\textrm{div}\psi
+2\mu|\mathcal{D}(\psi)|^2+\lambda(\textrm{div}\psi)^2+h+H.
\end{align*}
Applying Lemmas \ref{ellipticEst}  to \eqref{eh1-E1} and \eqref{eh1-E2}, respectively, we obtain 
\begin{align*}
\| \nabla^{k+2} \psi \| & \lesssim  \|\psi_t \|_{H^{k}}  +\|\widetilde{ G} \|_{H^k} + \|  \psi \|  
\quad \text{for $k=0, 1$},
\\
\| \nabla^{k+2} \zeta \| & \lesssim  \|\zeta_{t} \|_{H^{k}}  +\|\widetilde H \|_{H^k} + \|  \zeta \|  
\quad \text{for $k=0, 1$}.
\end{align*}
It is straightforward to see that the terms $\|\widetilde{ G} \|_{H^k}$ and $\|\widetilde{ H} \|_{H^k}$ are bounded by the right hand sides of \eqref{eh0-E} and \eqref{eh0-Ez}, respectively.
Therefore, we conclude \eqref{eh0-E} and \eqref{eh0-Ez}.
Similarly, applying Lemma \ref{ellipticEst2} to \eqref{eh1-E2} leads to \eqref{eh0-Ezz}.

Finally we show \eqref{ej0-E} and \eqref{ej0-Ez} together.
We apply $\partial_t$ to \eqref{eh1-E} to obtain
\begin{align*} 
&-\mu \Delta \psi_t -(\mu + \lambda) \nabla \div \psi_t = - (\rho \psi_{t})_{t}  + \widetilde{ G}_t, \quad 
\psi_t|_{\partial\Omega}=0, \quad \lim_{|x|\to\infty}\psi_t=0,
\\
&-\kappa\Delta\zeta_t =-c_v(\rho \zeta_{t})_{t}+{\widetilde H}_t,\quad
\zeta_t|_{\partial\Omega}=0, \quad \lim_{|x|\to\infty}\zeta_t =0,
\end{align*}
Applying Lemma \ref{ellipticEst} with $k=0$ to these boundary value problems gives
\begin{gather*}
\| \nabla^{ 2} \psi_t \|  \lesssim  \|\psi_{tt} \|    +\|\widetilde{ G}_{t}\| + \|  \psi_t \|,
\\
\| \nabla^{ 2} \zeta_t \|  \lesssim  \| \zeta_{tt} \|    +\| \widetilde{ H}_t \| + \| \zeta_t \|.
\end{gather*}
Furthermore, $\|\widetilde{ G}_{t} \|_{H^k}$ and $\|\widetilde{ H}_{t} \|_{H^k}$ are bounded by the right hand sides of \eqref{ej0-E} and \eqref{ej0-Ez}, respectively.
From the discussion above, we obtain \eqref{ej0-E} and \eqref{ej0-Ez}.
\end{proof}

\subsection{Completion of the a priori estimates} \label{ss-comp-apriori}

In this subsection, we complete the derivation of the a priori estimate.
To this end, we show the following lemma.

\begin{lemma} \label{lm1-comp}
Under the same assumption as in Proposition \ref{apriori1} with $m=3$,
it holds that
\begin{align}
&   e^{\sigma t} \| \nabla^{p+1}  \vp (t) \|^2 + \int^t_0 e^{\sigma \tau} \left\| \nabla^{p+1}   \frac{d}{dt} \vp (\tau) \right\|^2  d \tau   
\notag \\
&\lesssim_{\Omega} \| \Phi_0\|^2_{H^3} + e^{\sigma t} \| \vp (t) \|_{H^p}^2 + \int^t_0 e^{\sigma \tau} \|\psi_t(\tau) \|^2_{H^p} \, d \tau  
+ \epsilon^{-1}\int^t_0 e^{\sigma \tau} D_{p, \beta} (\tau) \, d \tau 
\notag \\
& \qquad + \epsilon \int^t_0  e^{\sigma \tau}(\| \nabla \vp (\tau) \|^2_{H^{p}} +\| (\nabla \psi,\nabla \zeta) (\tau) \|^2_{H^{p+1}} ) d \tau 
\notag \\
& \qquad + (N_\beta (T) + \delta + \sigma) \int^t_0 e^{\sigma \tau} D_{3, \beta} (\tau) \, d \tau
+ \delta  \int^t_0 e^{\sigma \tau}  \, d \tau
\label{el0-comp}
\end{align}
for $\epsilon \in (0,1)$ and $p=0,1,2$.
\end{lemma}

\begin{proof}
As in Lemma \ref{lm4-hat}, we define the index $\bm{a} = (a_1, a_2, a_3)$ with $a_1, \ a_2, \ a_3 \geq 0$ and $|\bm{a}| := a_1 + a_2 + a_3$. Let $\partial^{\bm{a}} := \partial^{a_1}_{y_1} \partial^{a_2}_{y_2} \partial^{a_3}_{y_3}$. It suffices to prove that for $j = 1, 2, \ldots , p+1$ and $\bm{a} = (a_1, a_2, a_3)$,
\begin{equation}
\sum_{|\bm{a}| = p+1, a_1 \leq j} \left( e^{\sigma t} \| \partial^{\bm{a}}_y \hat{\vp} (t) \|_{L^2(\mathbb R^3_+)}^2 
+ \int^t_0 e^{\sigma \tau} \left\| \partial^{\bm{a}}_y \hat{\frac{d}{dt}}\hat{\vp} (\tau) \right\|_{L^2(\mathbb R^3_+)}^2 d \tau \right)
\lesssim_{\Omega} \mathcal{R}_{p,\epsilon},
\label{ek0}
\end{equation}
where 
\begin{equation*}
\begin{split}
\mathcal{R}_{p,\epsilon} :=& \| \Phi_0\|^2_{H^3} + \int^t_0 e^{\sigma \tau} \|\psi_t (\tau) \|^2_{H^p} \, d \tau  
+ \epsilon^{-1}\int^t_0 e^{\sigma \tau} D_{p, \beta} (\tau) \, d \tau \\
& + \epsilon \int^t_0  e^{\sigma \tau}(\| \nabla \vp (\tau) \|^2_{H^{p}} +\| (\nabla \psi,\nabla \zeta) (\tau) \|^2_{H^{p+1}} ) d \tau \\
& + (N_\beta (T) + \delta + \sigma) \int^t_0 e^{\sigma \tau} D_{3, \beta} (\tau) \, d \tau
+ \delta  \int^t_0 e^{\sigma \tau}  \, d \tau. 
\end{split}
\end{equation*}
Indeed the desired estimate \eqref{el0-comp} follows from changing the coordinate $y \in \mathbb R_+^3$ to the coordinate $x \in \Omega$ in the left hand side of \eqref{ek0} with $j=p+1$.

To obtain \eqref{ek0} with $j=1$, we add up \eqref{ec0-hat} with $l = p+1$ 
and \eqref{ee0-hat} with $a_1=0$, $a_2 + a_3 = p$, 
and estimate 
$\int^t_0 e^{\sigma \tau} \| \partial^{\bm{a}} \nabla_{y'} \hat{\frac{d}{dt}} \hat{\vp} \|_{L^2}^2  \, d \tau$ 
and 
$\int^t_0 e^{\sigma \tau} \| \partial^{\bm{a}} \nabla_y \nabla_{y'} \hat{\psi} \|_{L^2}^2  \, d \tau$ 
by using \eqref{ec0-hat}.
Now, assuming \eqref{ek0} holds for $j=q$, 
we show that it holds for $j=q+1$. 
We take the weighted-in-time integral of \eqref{eg0-C} with $l = p+1-q$ and $k=q-1$, 
and use \eqref{ek0} with $j = q$ 
to estimate the highest-order term in $\|{\nabla}^{p+1-q}_{y'} \hat{\frac{d}{dt}} \hat{\vp} \|^2_{H^q} $. 
Then we arrive at
\begin{equation}
\int^t_0 e^{\sigma \tau}  \left( \| \nabla^{q+1}  \nabla^{p+1-q}_{y'} \hat{\psi}(\tau) \|_{L^2(\mathbb R^3_+)}^2 + \| \nabla^q  \nabla^{p+1-q}_{y'} \hat{\vp} (\tau) \|_{L^2(\mathbb R^3_+)}^2 \right) \, d \tau 
\lesssim_{\Omega} \mathcal{R}_{p,\epsilon}.
\label{ek2}
\end{equation}
Taking $a_1 =q$ and $ a_2 + a_3 = p-q$ in \eqref{ee0-hat}, and using \eqref{ek0} with $j = q$ and \eqref{ek2} to estimate the second term on the right hand side of \eqref{ee0-hat},
we obtain
\begin{equation}
e^{\sigma t} \| \partial^{q+1}_1 \nabla^{p-q}_{y'} \hat{\vp} (t) \|_{L^2(\mathbb R^3_+)}^2 + \int^t_0 e^{\sigma \tau}  \left\|\partial^{q+1}_1 \nabla^{p-q}_{y'} \hat{\frac{d}{dt}} \hat{\vp}(\tau)  \right\|_{L^2(\mathbb R^3_+)}^2  \, d \tau 
\lesssim_{\Omega} \mathcal{R}_{p,\epsilon}.
\label{ek3}
\end{equation} 
Combining \eqref{ek3} and \eqref{ek0} with $j=q$, we obtain \eqref{ek0} with $j=q+1$. The proof is completed by induction. 
\end{proof}

Moreover, we have the following lemma.
\begin{lemma} \label{lm3-comp}
Under the same assumption as in Proposition \ref{apriori1} with $m=3$,
it holds, for $p=0,1,2$, 
\begin{align}
&   e^{\sigma t} \| \nabla^{p+1}  \vp (t) \|^2 + \int^t_0 e^{\sigma \tau}  \|( \nabla^{p+1} \vp, \nabla^{p+2} \psi, \nabla^{p+2} \zeta )(\tau)\|^2   \, d \tau  + \int^t_0 e^{\sigma \tau} \left\| \nabla^{p+1}   \frac{d}{dt} \vp (\tau) \right\|^2 d \tau   
\notag \\
&\lesssim_{\Omega} \| \Phi_0\|^2_{H^3} + e^{\sigma t} \| \vp (t) \|_{H^p}^2 + \int^t_0 e^{\sigma \tau} \|(\psi_t,\zeta_t)(\tau) \|^2_{H^p} \, d \tau  + \int^t_0 e^{\sigma \tau} D_{p, \beta} (\tau) \, d \tau 
\notag\\
& \qquad + (N_\beta (T) + \delta + \sigma) \int^t_0 e^{\sigma \tau} D_{3, \beta} (\tau) \, d \tau
+ \delta  \int^t_0 e^{\sigma \tau}  \, d \tau. 
\label{em0-comp}
\end{align}
\end{lemma}
\begin{proof}
Take the weighted-in-time integral of \eqref{ef0-C} with $k =p$ and \eqref{eh0-Ezz} with $k =p$, and 
combine them together with \eqref{el0-comp}. 
Then, letting $\epsilon$ be suitably small, we arrive at \eqref{em0-comp}.
\end{proof}

Now we can complete the a priori estimates. 

\begin{proof}[Proof of Proposition \ref{apriori1}]
From \eqref{ea0}, we know that
\begin{align}
& e^{\sigma t} E_{0, \beta} (t) + \int^t_0 e^{\sigma \tau} D_{0,\beta} (\tau ) \, d \tau  
\notag \\
& \lesssim \|\Phi_0\|_{\lteasp{\beta}}^2 + \| \Phi_0\|^2_{H^3} 
+ (N_\beta (T) + \delta +\sigma) \int^t_0 e^{\sigma\tau} D_{3, \beta} (\tau) \, d \tau
+ \delta  \int^t_0 e^{\sigma \tau}  \, d \tau.   
\label{en0-comp}
\end{align}
We next show that for $l=1,2,3$,
\begin{align}
& e^{\sigma t} E_{l, \beta} (t) + \int^t_0 e^{\sigma \tau} D_{l,\beta} (\tau ) \, d \tau  
\notag \\
& \lesssim_{\Omega}  \| \Phi_0\|^2_{H^3} 
+ e^{\sigma t} E_{l-1, \beta} (t) + \int^t_0 e^{\sigma \tau} D_{l-1,\beta} (\tau ) \, d \tau
\notag \\
& \qquad + (N_\beta (T) + \delta +\sigma) \int^t_0 e^{\sigma \tau} D_{3, \beta} (\tau) \, d \tau
+ \delta  \int^t_0 e^{\sigma \tau}  \, d \tau + (l-1) \delta e^{\sigma \tau} .   
\label{en1-comp}
\end{align}
Hereafter $\nu$ denotes a positive constant.

Let us first treat the case $l=1$. 
We multiply \eqref{em0-comp} with $p =0$ by $\nu$ and add it to \eqref{ed0} with $k=0$. 
Taking $\ep$ and $\nu$ suitably small yields
\eqref{en1-comp} with $l=1$.

For the case $l=2$, we multiply \eqref{em0-comp} with $p =1$, 
\eqref{eh0-E} with $k=0$, and \eqref{eh0-Ez} with $k=0$ by $\nu$, $\nu e^{\sigma t}$, and $\nu e^{\sigma t}$, respectively.
Adding up the three results and \eqref{ec0}
and then taking $\nu$ small, we have \eqref{en1-comp} with $l=2$.

Finally, for the case $l=3$, we multiply \eqref{ej0-E} and \eqref{ej0-Ez} by $\nu e^{\sigma \tau}$, respectively, and integrate it over $(0,t)$.
We then multiply \eqref{eh0-E} with $k=1$, \eqref{eh0-Ez} with $k=1$, 
and \eqref{em0-comp} with $p =2$ by $\nu$, $\nu$, and $\nu^2$, respectively.
Add up the five results and \eqref{ed0} with $k=1$. 
Then taking $\ep$ and $\nu$ suitably small yields \eqref{en1-comp} with $l=3$.

The estimates \eqref{en0-comp} and \eqref{en1-comp} imply that
\begin{equation*}
\begin{split}
& e^{\sigma t} E_{3, \beta} (t) + \int^t_0 e^{\sigma \tau} D_{3,\beta} (\tau ) \, d \tau  \\
& \lesssim_{\Omega}  \|\Phi_0\|_{\lteasp{\beta}}^2 + \| \Phi_0\|^2_{H^3} 
+ (N_\beta (T) + \delta +\sigma) \int^t_0 e^{\sigma \tau} D_{3, \beta} (\tau) \, d \tau
+ \delta  \int^t_0 e^{\sigma \tau}  \, d \tau + \delta e^{\sigma \tau}.   \\  
\end{split}
\end{equation*}
Letting $N_\beta (T) + \delta +\sigma$ be small enough, one can have \eqref{apes0}.
Dividing \eqref{apes0} by $e^{\sigma \tau}$ and using \eqref{PhiSup}, we conclude \eqref{apes1}.
\end{proof}

We discuss briefly the proof of Corollary \ref{cor1}.

\begin{proof}[Proof of Corollary \ref{cor1}]
If $\|M\|_{H^{9}(\mathbb R^2)} \leq \eta_{0}$ holds for $\eta_{0}$ being in Lemma \ref{CattabrigaEst},
we can replace $\lesssim_{\Omega}$ by $\lesssim$ 
in inequalities \eqref{ef0-C}, \eqref{eg0-C}, and \eqref{em0-comp}.
Then following the proof of Proposition \ref{apriori1} with these improved inequalities, 
we conclude Corollary \ref{cor1}.
\end{proof}

\section{Construction of stationary solutions}\label{S5}
For the construction of stationary solutions, 
we make use of the time-global solution $\Phi$ in Theorem \ref{global1}. 
We first prove a unique result Proposition \ref{5.1} 
for the time-periodic solutions to \eqref{eq-pv1}--\eqref{pbc}. 
Then we consider $\Phi$ and its translated version $\Phi^k(t,x):=\Phi(t+kT^*,x)$ 
for any $T^* > 0$ and $k=1,2,3,\ldots$. 
We prove in Proposition \ref{5.3} that $\{\Phi^k\}$ is a Cauchy sequence 
in the Banach space $C([0,T^*];H^{m-1}(\Omega)) \cap C^1([0,T^*];H^{m-3}(\Omega))$ 
and obtain a limit $\Phi^*$ from it, 
then we show in Proposition \ref{5.3} 
that $\Phi^*$ is a time-periodic solution to problem \eqref{eq-pv1}--\eqref{pbc} with period $T^*>0$. 
In Subsection \ref{S5.2}, 
using uniqueness of time-periodic solutions, 
we prove that $\Phi^*$ is actually time-independent and 
therefore gives a stationary solution to \eqref{eq-pv1}--\eqref{pbc}.

We remark that it is also possible to show directly that
$\Phi^k_t$ converges to zero as $k\to\infty$ by differentiating equations \eqref{eq-pv}
with respect to $t$ and then applying the energy method used in Section \ref{sec3} 
to the resulting equations. 
However the computations are very long.
Therefore we adopt the method mentioned above to construct the stationary solution.

\subsection{Time-periodic solutions}\label{S5.1}
\subsubsection{Uniqueness}\label{S5.1.1}
In this subsection, we show the uniqueness of time-periodic solutions to the problem 
of equations \eqref{eq-pv1}--\eqref{eq-pv3} 
with boundary condition \eqref{pbc} in the solution space 
\begin{equation*}
\mathcal{X}^{\text{\rm e}}_{m,\beta} (0,T)
=X^{\text{\rm e}}_{m-1,\beta} (0,T)
\cap L^\infty(0,T ; H^m(\Omega)).
\end{equation*}

\begin{proposition}\label{5.1}
Let \eqref{outf} and \eqref{super1} hold.
For $\beta>0$ being in Theorem \ref{global1},
there exists $\epsilon>0$ depending on $\|M\|_{H^9}$ but independent of $\Omega$ such that 
if a time-periodic solution $\Phi^* \in \mathcal{X}^{\text{\rm e}}_{3,\beta} (0,T)$
with a period $T^*>0$ to the problem \eqref{eq-pv1}--\eqref{pbc}
exists and satisfies the following inequality, then it is unique:
\begin{equation}\label{uniasp1}
\sup_{t\in[0,T^*]}(\|\Phi^*(t)\|_{H^3}
+\|\partial_t \vp^*(t)\|_{H^{2}}
+\| (\partial_t \psi^*,\partial_t\zeta^*)(t)\|_{H^{1}})
+\delta \leq \epsilon.
\end{equation}
\end{proposition}

Let $\Phi^*=(\vp^*,\psi^*,\zeta^*)$ and $\Phi^\#=(\vp^\#,\psi^\#,\zeta^\#)$ 
be time-periodic solutions to \eqref{eq-pv1}--\eqref{pbc}.
It is straightforward to see that
$\overline{\Phi} = (\overline{\vp}, \overline{\psi} , \overline{\zeta} )= \Phi^* - \Phi^\#$ 
satisfies the system   
\begin{subequations} \label{5.1-eq1}
\begin{gather}
\overline{\vp}_{t} + (\tilde{u} + U + \psi^* ) \cdot \nabla \overline{\vp}  + (\tilde{\rho} + \vp^* ) \div \overline{\psi}  
=\overline{f}, \label{5.1-eq1-1}
\\
(\tilde{\rho} + \vp^*) \{ \overline{\psi}_{t} + (\tilde{u} + U + \psi^*) \cdot \nabla \overline{\psi}  \} - L \overline{\psi}  + R(\tilde\theta+\Theta+\zeta^*)\nabla \overline{\vp}+R(\tilde{\rho} + \vp^*)\nabla\overline\zeta 
= \overline{g}, \label{5.1-eq1-2}
\\
c_v(\tilde{\rho} + \vp^*)\big( \overline\zeta_t+(\tilde{u} + U + \psi^*)\cdot\nabla \overline\zeta\big)+R(\tilde{\rho} + \vp^*)(\tilde\theta+\Theta+\zeta^*)\textrm{div} \overline\psi
-\kappa\Delta \overline\zeta= \overline{h}.
\label{5.1-eq1-3}
\end{gather}
The boundary condition for $(\overline{\psi}, \overline{\zeta} )$ 
follows from (\ref{pbc}) as
\begin{gather*}
\overline{\psi} (t,M(x'),x') = 0,\quad\overline{\zeta} (t,M(x'),x') = 0.
\end{gather*}
\end{subequations}
Here $L  \overline\psi$,  $\overline{f}$, $\overline{g}$ and $ \overline{h}$ are defined by
\begin{align*}
L  \overline\psi &:= \mu \Delta  \overline\psi + (\mu + \lambda) \nabla \div  \overline\psi,
\\
\overline{f} &:=
- \nabla\tilde{\rho} \cdot \overline{\psi}  - \ut'_1 \overline{\vp}  - \overline{\vp} \div U - \overline{\psi} \cdot \nabla \vp^\# - \overline{\vp} \div \psi^\#,
\\
\overline{g} &:= - \left(\{(\rt\!+\!\vp^*)\overline{\psi}+\overline{\vp} \psi^\# \}\!\cdot\! \nabla \right) (\tilde{u}\!+\!U) - \overline{\varphi} ((\tilde{u}\!+\!U) \cdot \nabla) (\tilde{u}\!+\!U) -R \overline\zeta\nabla\rt\!-R \overline\vp\nabla (\tilde{\theta}+\Theta)
\\
&\qquad - \overline{\vp}  \psi^\#_t - \{(\rt+\vp^*)(\ut+U+\psi^*)-(\rt+\vp^\#)(\ut+U+\psi^\#) \}\cdot \nabla \psi^\# 
\\ 
&\qquad - R \overline\zeta\nabla\vp^\#-R \overline\vp\nabla \zeta^\#,
\\
\overline{h}&:= - c_v(\rt\!+\!\vp^*) (\overline\psi \cdot \nabla) (\tt+\Theta) - c_v\overline\vp\! (\psi^\# \cdot \nabla) (\tt+\Theta)- c_v\overline\vp ((\ut+U)\cdot \nabla) (\tt+\Theta)
\\ 
&\qquad - R\rt\overline\zeta\div(\ut+U)
-R\overline\vp(\tilde\theta+\Theta+\zeta^*)\div(\ut+ U)-R\vp^\#\overline\zeta\div(\ut+ U)
\\ 
&\qquad +4\mu\mathcal{D}(\overline\psi):\mathcal{D}(\ut+U)+2\lambda\div\overline\psi\div(\ut+U)-c_v\overline{\vp}  \zeta^\#_t
\\ 
&\qquad - c_v\{(\rt+\vp^*)(\ut+U+\psi^*)-(\rt+\vp^\#)(\ut+U+\psi^\#) \}\cdot \nabla \zeta^\# 
\\ 
&\qquad -R\{(\rt+\vp^*)(\tilde\theta+\Theta+\zeta^*)-(\rt+\vp^\#)
(\tilde\theta+\Theta+\zeta^\#) \}\textrm{div} \psi^\# 
\\ 
&\qquad +2\mu\left(\mathcal{D}( \psi^*)+\mathcal{D}(\psi^\#)\right):\mathcal{D}(\overline\psi)+\lambda\left(\textrm{div} \psi^*+\div\psi^\#\right)\div(\overline\psi).
\end{align*}

It is easy to check from \eqref{uniasp1} that 
\begin{equation*}
|\overline{f}| \lesssim \epsilon |\overline{\Phi}|, \quad |\overline{g}| \lesssim \epsilon |\overline{\Phi}| + |\overline{\vp}| |\psi^\#_t|
, \quad |\overline{h}| \lesssim \epsilon (|\overline{\Phi}|+|\nabla\overline{\psi}|) + |\overline{\vp}| |\zeta^\#_t| . 
\end{equation*}

\begin{proof}[Proof of Proposition \ref{5.1}]

Multiplying \eqref{5.1-eq1-1} by $e^{\beta x_1}R\tilde\theta \overline{\vp}/(\rt+\vp^*)$ and using $\rho_{t}=-{\rm div}(\rho u)$, we have
\begin{align}\label{5.1-eq2}
&\left(\frac{1}{2} e^{\beta x_1} \frac{R\tilde\theta}{\rt+\vp^*} \overline{\vp}^2\right)_t  
+ \div\left(\frac{1}{2}  e^{\beta x_1} R\tilde\theta \overline{\vp}^2 \frac{\tilde{u} + U + \psi^*}{\rt+\vp^*} \right)
- \frac{\beta}{2}   e^{\beta x_1}  R\tilde\theta \frac{\tilde{u}_1 + U_1 + \psi^*_1}{\rt+\vp^*} \overline{\vp}^2 
\notag \\
&= - e^{\beta x_1} R\tilde\theta \overline{\vp} \div\overline{\psi}
+e^{\beta x_1} R\tilde\theta\frac{ \overline{\vp}}{\rt+\vp^*} \overline{f}
+\frac{1}{2}  e^{\beta x_1}  R \overline{\vp}^2 \frac{(\tilde{u} + U + \psi^* )\cdot\nabla\tilde\theta}{\rt+\vp^*}
\notag \\
&\quad + e^{\beta x_1}R\tilde\theta \frac{\div (\tilde{u} + U + \psi^* )}{\rt+\vp^*} \overline{\vp}^2.
\end{align}
Multiplying \eqref{5.1-eq1-2} by $e^{\beta x_1} \overline{\psi}$ gives
\begin{align}
&\left(\frac{1}{2}  e^{\beta x_1} (\tilde{\rho} + \vp^*) |\overline{\psi}|^2\right)_t 
\notag \\
&\quad+ \div \left( \frac{1}{2} e^{\beta x_1} |\overline{\psi}|^2 (\tilde{\rho} + \vp^*) (\tilde{u} + U  + \psi^*)
+e^{\beta x_1}  R(\tilde\theta+\Theta+\zeta^*)\overline{\vp} \overline{\psi} 
+e^{\beta x_1}  R(\tilde\rho+\vp^*)\overline{\zeta}  \overline{\psi}  \right) 
\notag \\
&\quad - \div \left(\mu e^{\beta x_1} (\nabla \overline{\psi}) \overline{\psi}
+(\mu + \lambda) e^{\beta x_1} (\div \overline{\psi}) \overline{\psi} \right)
\notag \\
&\quad - \frac{\beta}{2} e^{\beta x_1}(\tilde{\rho} + \vp^*) (\tilde{u}_1 + U_1 + \psi^*_1)|\overline{\psi}|^2
-\beta e^{\beta x_1}  R(\tilde\theta+\Theta+\zeta^*)\overline{\psi}_1 \overline{\vp}
-\beta e^{\beta x_1}  R(\tilde\rho+\vp^*)\overline{\psi}_1 \overline\zeta
\notag \\
&\quad  + \mu e^{\beta x_1} |\nabla \overline{\psi}|^2  
+ (\mu + \lambda)  e^{\beta x_1} |\div \overline{\psi}|^2
\notag \\
& = e^{\beta x_1} \overline{\psi} \cdot \overline{g}
+e^{\beta x_1}  R\overline{\vp}\overline{\psi}\cdot\nabla (\tilde\theta+\Theta+\zeta^*)
+e^{\beta x_1}  R(\tilde\theta+\Theta+\zeta^*) \overline{\vp}\div\overline{\psi}
\notag \\
&\quad +e^{\beta x_1}  R (\tilde{\rho} + \vp^*)\overline\zeta\div\overline{\psi}
+e^{\beta x_1}  R \overline\zeta\overline{\psi}\cdot\nabla(\tilde{\rho} + \vp^*)
 - \mu \beta   e^{\beta x_1}  (\partial_1 \overline{\psi}) \cdot \overline{\psi} 
\notag \\
& \quad - (\mu + \lambda) \beta  e^{\beta x_1} (\div \overline{\psi}) \overline{\psi}_1.
\label{5.1-eq3}
\end{align}
Multiplying \eqref{5.1-eq1-3} by $e^{\beta x_1} \overline{\zeta}\tilde\theta^{-1}$, we obtain
\begin{align}
&\left(\frac{c_v}{2 \tilde\theta}  e^{\beta x_1} (\tilde{\rho} + \vp^*) |\overline{\zeta}|^2\right)_t 
+ \div \left( \frac{c_v}{2 \tilde\theta} e^{\beta x_1} |\overline{\zeta}|^2 (\tilde{\rho} + \vp^*) (\tilde{u} + U  + \psi^*)
-\kappa e^{\beta x_1} \frac{\overline{\zeta}}{ \tilde\theta}\nabla\overline\zeta\right)
\notag \\
&\quad - \frac{\beta c_v}{2 \tilde\theta} e^{\beta x_1}(\tilde{\rho} + \vp^*) (\tilde{u}_1 + U_1 + \psi^*_1)|\overline{\zeta}|^2
+\kappa e^{\beta x_1}\frac{|\nabla \overline{\zeta}|^2}{ \tilde\theta}
\notag
\\
&= e^{\beta x_1}  \frac{\overline{\zeta}}{ \tilde\theta} \overline{h}
-\frac{ c_v}{2} e^{\beta x_1}(\tilde{\rho} + \vp^*) (\tilde{u} + U + \psi^*)\cdot \frac{ \nabla\tilde\theta}{ \tilde\theta^2} |\overline{\zeta}|^2
-\beta\kappa e^{\beta x_1} \frac{\overline{\zeta}}{ \tilde\theta}\partial_1\overline\zeta
+\kappa e^{\beta x_1}\overline{\zeta}\frac{ \nabla\overline\zeta\cdot\nabla\tilde\theta}{ \tilde\theta^2}
\notag \\
&\quad-Re^{\beta x_1} (\tilde{\rho} + \vp^*)\frac{\tilde\theta+\Theta+\zeta^*}{ \tilde\theta} \overline{\zeta} \textrm{div} \overline\psi.
\label{5.1-eq4}
\end{align}
Summing up \eqref{5.1-eq2}--\eqref{5.1-eq4} yields that
\begin{multline}
\left(\frac{1}{2} e^{\beta x_1} \frac{R\tilde\theta}{\rt+\vp^*} \overline{\vp}^2
+\frac{1}{2}  e^{\beta x_1} (\tilde{\rho} + \vp^*) |\overline{\psi}|^2+
\frac{c_v}{2 \tilde\theta}  e^{\beta x_1} (\tilde{\rho} + \vp^*) |\overline{\zeta}|^2\right)_t 
+\div\left(\overline{G}_1+\overline{B}_1\right)
-\beta(\overline{G}_1)_1
\\
+ \mu e^{\beta x_1} |\nabla \overline{\psi}|^2  
+ (\mu + \lambda)  e^{\beta x_1} |\div \overline{\psi}|^2
 +\kappa e^{\beta x_1}\frac{|\nabla \overline{\zeta}|^2}{ \tilde\theta}
=\overline{R}_1+ \beta (\overline{B}_1)_1,
\label{5.1-eq44}
\end{multline}
where
\begin{align*}
\overline{G}_1:=& \frac{1}{2}  e^{\beta x_1} R\tilde\theta \overline{\vp}^2  \frac{\tilde{u} + U + \psi^*}{\rt+\vp^*}
+\frac{1}{2} e^{\beta x_1} \left(|\overline{\psi}|^2+c_v\frac{|\overline{\zeta}|^2}{ \tilde\theta} \right) (\tilde{\rho} + \vp^*) (\tilde{u} + U + \psi^*)
\notag \\
& +e^{\beta x_1}  R(\tilde\theta+\Theta+\zeta^*)\overline{\vp} \overline{\psi} 
+e^{\beta x_1}  R(\tilde\rho+\vp^*)\overline{\zeta} \overline{\psi}  ,
\\
\overline{B}_1:=&-\mu e^{\beta x_1} (\nabla \overline{\psi}) \overline{\psi}
-(\mu+ \lambda) e^{\beta x_1} (\div \overline{\psi}) \overline{\psi}
-\kappa e^{\beta x_1} \frac{\overline{\zeta}}{ \tilde\theta}\nabla\overline\zeta,
\\
\overline{R}_1:=& 
e^{\beta x_1} R\tilde\theta\frac{ \overline{\vp}}{\rt+\vp^*} \overline{f}
+\frac{1}{2}  e^{\beta x_1}  R \overline{\vp}^2 \frac{(\tilde{u} + U + \psi^*)\cdot \nabla\tilde\theta}{\rt+\vp^*}
+    e^{\beta x_1}R\tilde\theta \frac{\div (\tilde{u} + U + \psi^* )}{\rt+\vp^*} \overline{\vp}^2
\notag \\
&
+e^{\beta x_1} \overline{\psi} \cdot \overline{g}
+e^{\beta x_1}  R\overline{\vp}\overline{\psi}\cdot\nabla (\tilde\theta+\Theta+\zeta^*)
+e^{\beta x_1}  R(\Theta+\zeta^*)\overline{\vp}\div\overline{\psi} 
\notag \\
&
+e^{\beta x_1}  R \overline\zeta\overline{\psi}\cdot\nabla(\tilde{\rho} + \vp^*)
+ e^{\beta x_1}  \frac{\overline{\zeta}}{ \tilde\theta} \overline{h}
-\frac{ c_v}{2} e^{\beta x_1}(\tilde{\rho} + \vp^*) (\tilde{u} + U + \psi^*)\cdot\frac{ \nabla\tilde\theta}{ \tilde\theta^2} |\overline{\zeta}|^2
\notag \\
&+\kappa e^{\beta x_1}\overline{\zeta}\frac{ \nabla\overline\zeta\cdot\nabla\tilde\theta}{ \tilde\theta^2}
-Re^{\beta x_1} (\tilde{\rho} + \vp^*)\frac{\Theta+\zeta^*}{ \tilde\theta} \overline{\zeta} \textrm{div} \overline\psi.
\end{align*}

The second term on the left hand side of \eqref{5.1-eq44} is estimated from
below by using the divergence theorem, the fact $(u_b\cdot n) \geq c>0$, $ C\geq (\tilde{\rho} + \vp^*) \geq c>0$ and boundary conditions \eqref{pbc} as
\begin{equation}\label{5.1-eq5}
\int_\Omega \div(\overline{B}_1+\overline{G}_1) \,dx 
= \int_{\partial\Omega} \frac{1}{2} e^{\beta M(x')}R\tilde\theta_b\frac{\overline{\vp}^2}{\rt+\vp^*}\Big|_{\partial\Omega}  (u_b\cdot n) d \sigma
\gtrsim \|\overline{\vp}(t,M(\cdot),\cdot) \|_{L^2(\mathbb R^2)}^2.
\end{equation}
Next we derive the lower estimate of the third term on the
left hand side of \eqref{5.1-eq44}.
We compute the term $\overline{G}_1$ as
\begin{gather*}
\begin{aligned}
(\overline{G}_1)_1
=&
\frac{1}{2} e^{\beta x_1} R\theta_+ \rho_+^{-1} u_+\overline{\vp}^2
+\frac{1}{2} e^{\beta x_1}  \left(|\overline{\psi}|^2+c_v\theta_+^{-1}|\overline{\zeta}|^2 \right)  \rho_+u_+
\\
&+e^{\beta x_1} R\theta_+\overline{\vp}\overline{\psi}_1
+e^{\beta x_1} R\rho_+\overline{\zeta}\overline{\psi}_1
+ e^{\beta x_1} \overline{R}_{2},
\end{aligned}
\\
\begin{aligned}
\overline{R}_{2}
:=& \frac{1}{2} \left\{ R\tilde\theta \frac{\tilde{u} + U + \psi^*}{\rt+\vp^*}
- R\theta_+ \rho_+^{-1} u_+ \right\}\overline{\vp}^2
\\
&+\frac{1}{2}|\overline{\psi}|^2  \left\{(\tilde{\rho} + \vp^*) (\tilde{u}_1 + U_1 + \psi^*_1)-\rho_+u_+ \right\}
\\
&+\frac{c_v}{2}|\overline{\zeta}|^2   \left\{(\tilde{\rho} + \vp^*) (\tilde{u}_1 + U_1 + \psi^*_1)\tilde\theta^{-1}-\rho_+u_+\theta_+^{-1} \right\}
\\
&+R\left\{\tilde\theta+\Theta+\zeta^*-\theta_+\right\}\overline{\vp}\overline{\psi}_1
+R\left\{\tilde\rho+\vp^*-\rho_+\right\}\overline{\zeta}\overline{\psi}_1.
\end{aligned}
\notag
\end{gather*}
Thus, by using this, the third term on the left hand side of \eqref{5.1-eq44} is rewritten as
\begin{gather*}
-\beta (\overline{G}_1)_1
=
\beta e^{\beta x_1}
\Bigl(
F(\overline{\vp},\overline{\psi}_1,\overline\zeta)  - \overline{R}_{2}
\Bigr),
\\
\overline{F}_{1}(\overline{\vp},\overline{\psi}_1,\overline\zeta)
:=
\frac{R\theta_{+}|u_{+}|}{2\rho_{+}} \overline{\vp}^2
+ \frac{\rho_+ |u_+|}{2} \overline{\psi}_{1}^{2}
+\frac{c_{v}\rho_{+}|u_{+}|}{2\theta_{+}} \overline\zeta^{2}
-R\theta_{+}\overline\vp\overline\psi_{1}-R \rho_{+} \overline\zeta\overline\psi_{1},
\nonumber
\end{gather*}
where $\overline{\psi}'$ is the second and third components of $\overline{\psi}$ defined by 
$\overline{\psi}' := (\overline{\psi}_2,\overline{\psi}_3)$.
Owing to the supersonic condition \eqref{super1}, 
the quadratic form $\overline{F}_{1}(\overline{\vp},\overline{\psi}_1)$ becomes
positive definite.
On the other hand,
the remaining terms $\overline{R}_{2}$ satisfy
\begin{equation*}
|\overline{R}_{2}|
\lesssim |(\rt - \rho_+,\ut - u_+,\tilde\theta-\theta_+,U_1,\Theta)| |\overline{\Phi}|^2 + |(\vp^*,\psi^*
,\theta^*)| |\overline{\Phi}|^2
\lesssim \epsilon |\overline{\Phi}|^2.
\end{equation*}
Therefore we get the lower estimate of the integral of
$-\beta (\overline{G}_1)_1$ as 
\begin{equation}
\int_{\Omega} -\beta (\overline{G}_1)_1 \, dx
\geq
\beta
(c - C \epsilon)
\| \overline{\Phi} \|^2_{L^2_{e, \beta} (\Omega)}.
\label{5.1-eq6}
\end{equation}
The right hand side of \eqref{5.1-eq44}
is estimated by using \eqref{uniasp1} and the Schwarz inequality as
\begin{gather*}
|\overline{R}_1| \lesssim \epsilon e^{\beta x_1}|\overline{\Phi}|^2
+\epsilon e^{\beta x_1}|\overline{\Phi}||(\nabla\overline{\psi},\nabla\overline{\zeta})|
+ e^{\beta x_1}|(\psi^\#_t,\zeta^\#_t)|  |\overline{\vp}| |(\overline{\psi},\overline{\zeta})|,
\\
|\beta (\overline{B}_1)_1| \lesssim \beta (\nu e^{\beta x_1}|(\overline{\psi},\overline{\zeta})|^2
+\nu^{-1}e^{\beta x_1}|(\nabla\overline{\psi},\nabla\overline{\zeta})|^2),
\end{gather*}
where $\nu$ is an arbitrary positive constant.
Then the integrals are estimated by \eqref{uniasp1} and the Sobolev inequality as
\begin{gather}
\int_{\Omega}
|\overline{R}_1| \, dx
\lesssim  \epsilon \| (\overline{\Phi},\nabla\overline{\psi},\nabla\overline{\zeta} )\|^2_{L^2_{e,\beta}},
\label{5.1-eq7}
\\
\int_{\Omega}
|\beta (\overline{G}_1)_1|  \, dx
\lesssim
\beta \big(\nu \|(\overline{\psi},\overline{\zeta}) \|^2_{L^2_{e, \beta}} 
+ \nu^{-1}  \| (\nabla\overline{\psi},\nabla\overline{\zeta}) \|^2_{L^2_{e, \beta}} \big).
\label{5.1-eq8}
\end{gather}

Integrate \eqref{5.1-eq4} over $\Omega$,
substitute the estimates \eqref{5.1-eq5}--\eqref{5.1-eq8} into the resulting equality
and then let $\nu$, \footnote{Here it is enough to take the same $\beta$ as in the proof of Lemma \ref{lm1}}{$\beta$}, and $\epsilon$ suitably small to obtain
\begin{equation}\label{5.1-eq14}
\frac{1}{2} \frac{d}{dt} \int e^{\beta x_1} 
\left(  \frac{R\tilde\theta}{\rt+\vp^*} \overline{\vp}^2 + (\tilde{\rho} + \vp^*) |\overline{\psi}|^2+
\frac{c_v}{\tilde\theta} (\tilde{\rho} + \vp^*) |\overline{\zeta}|^2 \right) dx 
+ c \beta \|\overline{\Phi}\|_{L^2_{e, \beta}}^2  
\leq 0. 
\end{equation}
Integrating \eqref{5.1-eq14} over $[0,T^*]$ and using the periodicity of solutions
lead to $\overline{\Phi}=0$.
The proof is complete.
\end{proof}

\subsubsection{Existence}\label{S5.1.2}

For the construction of time-periodic solutions, 
we use the time-global solution $\Phi$ in Theorem \ref{global1}.
Here we see from following \cite[Appendix B]{SZ1}
that there exist initial data satisfying the conditions in Theorem \ref{global1}.
Now we define
\[
 \Phi^k(t,x):=\Phi(t+kT^*,x)
\quad \text{for $k=1,2,3,\ldots$.} 
\]

Let us first show the following lemma.

\begin{lemma}\label{5.2}
Let \eqref{outf} and \eqref{super1}.
For $\beta>0$ being in Theorem \ref{global1} and any $T^*>0$,
there exists $\sigma_0=\sigma_0(\Omega)>0$ and $C_0=C_0(\Omega)>0$ 
depending on $\Omega$ but independent of $k$ and $T^*$ such that
\begin{equation}\label{exiapes0}
\|(\Phi-\Phi^k)(t)\|_{\lteasp{\beta}}\leq C_0e^{-\sigma_0 t}
\quad \text{for $k=1,2,3,\ldots$.}
\end{equation}
\end{lemma}
\begin{proof}
For any $k$, $k'$, let $\overline{\Phi} = (\overline{\vp}, \overline{\psi}, \overline{\zeta} )= \Phi^k - \Phi^{k'}$ satisfy the system   
\begin{subequations} \label{5.2-eq1}
\begin{gather}
\overline{\vp}_{t} + (\tilde{u} + U + \psi^k ) \cdot \nabla \overline{\vp}  + (\tilde{\rho} + \vp^k ) \div \overline{\psi}  
=f^{k,k'}, \label{5.2-eq1-1}
\\
(\tilde{\rho} + \vp^k) \{ \overline{\psi}_{t} + (\tilde{u} + U + \psi^k) \cdot \nabla \overline{\psi}  \} - L \overline{\psi}  + R(\tilde\theta+\Theta+\zeta^k)\nabla \overline{\vp}+R(\tilde{\rho} + \vp^k)\nabla\overline\zeta 
= g^{k,k'}, \label{5.2-eq1-2}
\\
c_v(\tilde{\rho} + \vp^k)\big( \overline\zeta_t+(\tilde{u} + U + \psi^k)\cdot\nabla \overline\zeta\big)+R(\tilde{\rho} + \vp^k)(\tilde\theta+\Theta+\zeta^k)\textrm{div} \overline\psi
-\kappa\Delta \overline\zeta= h^{k,k'}.
\label{5.2-eq1-3}
\end{gather}
The boundary condition for $(\overline{\psi}, \overline{\zeta} )$ 
follows from (\ref{pbc}) as
\begin{gather*}
\overline{\psi} (t,M(x'),x') = 0,\quad\overline{\zeta} (t,M(x'),x') = 0.
\end{gather*}
\end{subequations}
Here $f^{k,k'},g^{k,k'}$ and $h^{k,k'}$ are defined by 
\begin{align*}
f^{k,k'} &:=
- \nabla\tilde{\rho} \cdot \overline{\psi}  - \ut'_1 \overline{\vp}  - \overline{\vp} \div U - \overline{\psi} \cdot \nabla \vp^{k'} - \overline{\vp} \div \psi^{k'},
\\
g^{k,k'} &:= - \left(\{(\rt\!+\!\vp^k)\overline{\psi}+\overline{\vp} \psi^{k'} \}\!\cdot\! \nabla \right) (\tilde{u}\!+\!U) - \overline{\varphi} ((\tilde{u}\!+\!U) \cdot \nabla) (\tilde{u}\!+\!U) -R \overline\zeta\nabla\rt\!-R \overline\vp\nabla (\tilde{\theta}+\Theta)
\\
&\qquad - \overline{\vp}  \psi^{k'}_t - \{(\rt+\vp^k)(\ut+U+\psi^k)-(\rt+\vp^{k'})(\ut+U+\psi^{k'}) \}\cdot \nabla \psi^{k'}
\\ 
&\qquad - R \overline\zeta\nabla\vp^{k'}-R \overline\vp\nabla \zeta^{k'},
\\
h^{k,k'}&:= - c_v(\rt\!+\!\vp^k) (\overline\psi \cdot \nabla) (\tt+\Theta) - c_v\overline\vp\! (\psi^{k'} \cdot \nabla) (\tt+\Theta)- c_v\overline\vp ((\ut+U)\cdot \nabla) (\tt+\Theta)
\\ 
&\qquad - R\rt\overline\zeta\div(\ut+U)
-R\overline\vp(\tilde\theta+\Theta+\zeta^k)\div(\ut+ U)-R\vp^{k'}\overline\zeta\div(\ut+ U)
\\ 
&\qquad +4\mu\mathcal{D}(\overline\psi):\mathcal{D}(\ut+U)+2\lambda\div\overline\psi\div(\ut+U)-c_v\overline{\vp}  \zeta^{k'}_t
\\ 
&\qquad - c_v\{(\rt+\vp^k)(\ut+U+\psi^k)-(\rt+\vp^{k'})(\ut+U+\psi^{k'}) \}\cdot \nabla \zeta^{k'}
\\ 
&\qquad -R\{(\rt+\vp^k)(\tilde\theta+\Theta+\zeta^k)-(\rt+\vp^{k'})
(\tilde\theta+\Theta+\zeta^{k'}) \}\textrm{div} \psi^{k'}
\\ 
&\qquad +2\mu\left(\mathcal{D}( \psi^k)+\mathcal{D}(\psi^{k'})\right):\mathcal{D}(\overline\psi)+\lambda\left(\textrm{div} \psi^k+\div\psi^{k'}\right)\div(\overline\psi).
\end{align*}

Repeat exactly the proof of Proposition \ref{5.1} 
with $k$ in place of $*$, $k'$ in place of $\#$, 
and \eqref{uniasp1} in place of \eqref{bound1} with small initial data,
we obtain
\begin{equation*}
\frac{1}{2} \frac{d}{dt} \int e^{\beta x_1} \left(\frac{R\tilde\theta}{\rt+\vp^k} \overline{\vp}^2 + (\tilde{\rho} + \vp^k) |\overline{\psi}|^2+
\frac{c_v}{\tilde\theta} (\tilde{\rho} + \vp^k) |\overline{\zeta}|^2 \right) dx 
+ c \beta \|\overline{\Phi}\|_{L^2_{e, \beta}}^2  \leq 0 ,  \\
\end{equation*}
which implies \eqref{exiapes0} once we take $k' =0$. The proof is complete.
\end{proof}

Now we can construct the time-periodic solutions:
\begin{proposition}\label{5.3}
Let \eqref{outf} and \eqref{super1} hold, and $m = 3,4,5$.
For $\beta>0$ being in Theorem \ref{global1} and any $T^*>0$,
there exists a constant $\epsilon>0$ independent of $T^*$ such that 
if $\delta \leq \epsilon$, then the problem \eqref{eq-pv1}--\eqref{pbc}
has a time-periodic solution $\Phi^*\in {\mathcal X}^m_\beta(0,T^*)$
with a period $T^*>0$. Furthermore, it satisfies
\begin{equation}\label{apes3}
\sup_{t\in[0,T^*]}(\|\Phi^*(t)\|_{H^m}
+\|\partial_t \vp^*(t)\|_{H^{m-1}}
+\|(\partial_t \psi^*,\partial_t \zeta^* )(t)\|_{H^{m-2}})
\leq C_0\delta,
\end{equation}
where $C_0=C_0(\Omega)>0$ is a constant depending on $\Omega$ but independent of $T^*$.
\end{proposition}

\begin{proof}
First of all, applying Theorem \ref{global1} to the initial-boundary value problem \eqref{eq-pv},
we obtain a time-global solution $\Phi$ to \eqref{eq-pv} with
\eqref{bound1} and \eqref{exiapes0}.

Recall that $\Phi^k(t,x):=\Phi(t+kT^*,x)$ for any $T^* > 0$ and $k=1,2,3,\ldots$. Let us first prove that $\{\Phi^k\}$ is a Cauchy sequence in the Banach space 
$C([0,T^*];H^{m-1}(\Omega)) \cap C^1([0,T^*];H^{m-3}(\Omega))$.
For $k>k'$, one can see from \eqref{exiapes0} that 
for $k>k'$, there holds
\begin{align*}
\sup_{t \in [0,T^*]} \|(\Phi^k-\Phi^{k'})(t) \|_{\lteasp{\beta}}
&= \sup_{t \in [0,T^*]} \|(\Phi(t+kT^*)-\Phi(t+k'T^*)\|_{\lteasp{\beta}}
\\
&= \sup_{t \in [k'T^*,(k'+1)T^*]} \|\Phi(t)-\Phi(t+(k-k')T^*)\|_{\lteasp{\beta}}
\\
& \lesssim_{\Omega} e^{-\sigma k'T^*}.
\end{align*}
This and \eqref{bound1} together with the Gagliardo-Nirenberg inequalities \eqref{GN1} leads to 
\begin{align*}
\sup_{t \in [0,T^*]} \|(\Phi^k-\Phi^{k'})(t) \|_{H^{m-1}}
& \lesssim \sup_{t \in [0,T^*]} \|(\Phi^k-\Phi^{k'})(t)\|_{H^m}^{1-1/m}
\|(\Phi^k-\Phi^{k'})(t)\|_{\lteasp{\beta}}^{1/m}
\notag \\
& \lesssim_{\Omega} e^{-\sigma k'T^*/m}. 
\end{align*}
So, what is left is to show that $\{\Phi^k\}$ is a Cauchy sequence 
in $C^1([0,T^*];H^{m-3}(\Omega))$.

We have already known from the proof of Lemma \ref{5.2} that $\Phi^k-\Phi^{k'}$ satisfies \eqref{5.2-eq1}. From this and \eqref{bound1}, we have
\[
|\partial_t (\Phi^k-\Phi^{k'})| \lesssim
|((\Phi^k-\Phi^{k'}),\nabla(\Phi^k-\Phi^{k'}),\nabla^2(\Phi^k-\Phi^{k'}))|
\]
which gives
\begin{equation*}
\|\partial_t (\Phi^k-\Phi^{k'})\| \lesssim_\Omega e^{-\sigma k'T^*/m}.
\end{equation*}
In the case $m=3$, this estimate is sufficient.
For the case $m=4,5$, this estimate and \eqref{bound1}
together with Gagliardo-Nirenberg inequalities leads to 
\begin{align}
{}&
\sup_{t \in [0,T^*]} \|\partial_t(\Phi^k-\Phi^{k'})(t) \|_{H^{m-3}}
\lesssim_{\Omega} e^{-\sigma k'T^*/\{m(m-3)\}}.
\notag
\end{align}
Hence, we see that  $\{\Phi^k\}$ is a Cauchy sequence 
and thus there exists a limit  $\Phi^*$ such that
\begin{equation}\label{converge1}
\Phi^k \to \Phi^* \quad \text{in} \ \
C([0,T];{\lteasp{\beta}}(\Omega)) \cap
\bigcap_{i=0}^1C^i([0,T^*];H^{m-1-2i}(\Omega)).
\end{equation}
It is straightforward to check that 
the limit $\Phi^*$ satisfies \eqref{eq-pv}.

Then we can check that $\Phi^* \in {\cal X}^m_\beta(0,T)$ as follows.
On the other hand, by a standard argument,
$\Phi^k(t)$ converges to $\Phi^*(t)$ weakly in $H^m(\Omega)$ 
for each $t \in [0,T^*]$ and also
\begin{equation}\label{exies8}
\sup_{t\in [0,T^*]}\|\Phi^*(t)\|_{H^m} 
\lesssim_\Omega \|\Phi_0\|_{\lteasp{\beta}}+\|\Phi_0\|_{H^m}+\delta
\end{equation}
follows from \eqref{bound1}. 
Hence, we conclude $\Phi^* \in L^\infty([0,T^*];H^{m}(\Omega))$.
It is also seen from the system \eqref{eq-pv} that
$\partial_t \vp^* \in L^\infty([0,T^*];H^{m-1}(\Omega))$,
$\partial_t (\psi^*,\zeta^*) \in L^\infty([0,T^*];H^{m-2}(\Omega))$,
\begin{gather}
\sup_{t\in [0,T^*]}\|\partial_t \vp^*(t)\|_{H^{m-1}} 
+\sup_{t\in [0,T^*]}\|(\partial_t \psi^*,\partial_t \zeta^*)(t)\|_{H^{m-2}} 
\lesssim_\Omega \|\Phi_0\|_{\lteasp{\beta}}+\|\Phi_0\|_{H^{m}}+\delta.
\label{exies9}
\end{gather}

Let us show that $\Phi^*$ is a time-periodic function with period $T^*>0$.
The sequences $\Phi^k(T^*,x)$ and $\Phi^{k+1}(0,x)$ 
converge to $\Phi^*(T^*,x)$ and $\Phi^*(0,x)$, respectively, 
as $k$ tends to infinity.
Notice that $\Phi^k(T^*,x)=\Phi^{k+1}(0,x)$ holds and so does $\Phi^*(T^*,x)=\Phi^*(0,x)$.
Hence, we have constructed a time-periodic solution $\Phi^*$ 
to the problem \eqref{eq-pv1}--\eqref{pbc} 
in the function space ${\mathcal X}^{\text{e}}_{m,\beta}(0,T^*)$
in which the uniqueness has been shown.
What is left is to prove the estimate \eqref{apes3}.
We can find  an initial data $\Phi_0^{\#} \in H^{5}$ which satisfies the compatibility condition \eqref{cmpa0} and
$\|\Phi_0^{\#}\|_{\lteasp{\beta}}+\|\Phi_0^{\#}\|_{H^5} \lesssim \delta$ in much the same way as in the proof of \cite[Lemma B.1]{SZ1}.
For the initial data $\Phi_0^{\#} $, we have another time-periodic solution by the above method.
However, Proposition \ref{5.1} together with estimates \eqref{exies8} and \eqref{exies9}
ensure that these periodic solutions are same.
Hence, \eqref{apes3} follows from plugging $\Phi_0=\Phi_0^{\#}$ into \eqref{exies8} and \eqref{exies9}.
The proof is complete.
\end{proof}

\subsection{Stationary solutions}\label{S5.2}

Now we show that the time-periodic solutions constructed in Subsection \ref{S5.1}
are time-independent, which gives us Theorem \ref{th4}.

\begin{proof}[Proof of Theorem \ref{th4}]
Proposition \ref{5.3} ensures the existence of
time-periodic solutions $\Phi^*$
of problem \eqref{eq-pv1}--\eqref{pbc} for any period $T^*$.
We remark that the smallness assumption for $\delta$ is independent of the period $T^*$.
Hence, one can have time-periodic solutions $\Phi^*$ with the period $T^*$ and
$\Phi^*_l$ with the period $T^*/2^l$ for $l \in \mathbb N$ 
under the same assumption for $\delta$.
Furthermore, $\Phi^*=\Phi^*_l$ follows from 
Proposition \ref{5.1}, since $\Phi^*$ and $\Phi^*_l$ are 
the time-periodic solutions with the period $T^*$ and satisfy \eqref{apes3}. 
Hence, we see that
\[
\Phi^*\left(0,x\right)=
\Phi^*\left(\frac{i}{2^l}T^*,x\right)
\quad \text{for $i=1,2,3,\ldots,2^l$ and $l=0,1,2,\ldots$.} 
\]
Because the set 
$\cup_{l \geq 0} \{{i}/{2^l} \ ; \ i=1,2,3,\ldots,2^l\}$
is dense in $[0,T^*]$,
we see from the continuity of $\Phi^*$
that $\Phi^*$ is independent of $t$.
Hence, $\Phi^s=\Phi^*$
is the desired solution to 
the stationary problem corresponding to the problem \eqref{eq-pv}.
\end{proof}

\subsection{Stability with exponential weight functions}\label{S5.3}

We prove the stability of stationary solutions, which gives us Theorem \ref{th5}.

\begin{proof}[Proof of Theorem \ref{th5}]
Theorem \ref{global1} and Lemma \ref{5.2} ensure that the initial--boundary value problem \eqref{eq-pv} has 
a unique time-global solution satisfying \eqref{bound1} and \eqref{exiapes0}
if $\|\Phi_0\|_{\lteasp{\beta}}+\|\Phi_0\|_{H^m}$ and $\delta$ are small enough.
So, it suffices to show that this time-global solution $\Phi$ converges 
to the stationary solution solution $\Phi^s$ exponentially fast 
as $t$ tends to infinity.
Passing the limit $k\to \infty$ in \eqref{exiapes0},
we have $\|(\Phi-\Phi^s)(t)\|_{\lteasp{\beta}} \lesssim_\Omega e^{-\sigma t}$
thanks to \eqref{converge1} and $\Phi^s=\Phi^*$.
This and \eqref{bound1} together with Gagliardo-Nirenberg inequalities \eqref{GN1} 
and Sobolev's inequality \eqref{sobolev2} lead to 
\begin{equation*}
\|(\Phi-\Phi^s)(t)\|_{L^{\infty}} \lesssim_\Omega e^{-\sigma t},
\end{equation*}
where $\sigma$ is a positive constant independent of $t$.
The proof is complete.
\end{proof}

\subsection{Corollary}

We discuss briefly the proof of Corollary \ref{cor2}.

\begin{proof}[Proof of Corollary \ref{cor2}]
From Corollary \ref{cor1}, we have an improved estimate \eqref{bound1} 
with constants $C_0=C_0(\beta)$ and $\sigma=\sigma(\beta)$ independent of $\Omega$. 
In the same way as in Subsections \ref{S5.1}--\ref{S5.3} with the improved estimate, 
we can conclude Corollary \ref{cor2}.
\end{proof}

\section{Stability with no weight function}\label{S6}

In this section we discuss Theorem \ref{th3}, which gives the stability of $(\rho^s, u^s, \theta^s)$ in $H^3$.
Here we do not assume $(\rho_0-\rho^s, u_0-u^s, \theta_0-\theta^s) \in \lteasp{\beta}$.

For $(\rt, \ut, \tt)$ in Proposition \ref{ex-st},
$U, \Theta$ in \eqref{ExBdry0}, and $\Phi^s$ in Theorem \ref{th4}, let us set
\begin{equation*}
(\rho^s, u^s, \theta^s)(x) : = (\rt, \ut, \tt)(\tilde{M}(x)) + (0,U, \Theta)(x)+ \Phi^s(x) .
\end{equation*}
Then it is obvious that $(\rho^s, u^s, \theta^s)$ satisfies \eqref{snse}.

We also introduce the perturbations
\begin{gather*}
 (\vp,\psi, \zeta)(t,x)
:=
(\rho, u, \theta)(t,x)
-
(\rho^s, u^s, \theta^s)( x),
\ \; \text{where} \ \;
\psi = (\psi_1,\psi_2,\psi_3). 
\end{gather*}
%
For notational convenience, we use norms $E_{m,0}$ and $D_{m,0}$,
which are defined in the same way as \eqref{Ekbeta-def} and \eqref{Dkbeta-def}, 
with $\Phi = ( \vp, \psi, \zeta)(t,x)$ being replaced by the functions defined right above. 
Furthermore, we also define 
\[
N (T) := \sup_{0\leq t \leq T} \|\Phi (t)\|_{H^3}.  
\]
We will see that $\Phi = ( \vp, \psi, \zeta)(t,x)$ satisfies the bound $N (T) \ll1$
by assuming the smallness of the initial data $( \vp, \psi, \zeta)(0,x)$.

Owing to (\ref{cnse}),
the perturbation $(\vp,\psi, \zeta)$ satisfies the system of
equations
\begin{subequations}
\label{nowe-eq-pv}
\begin{gather}
\vp_t + u \cdot \nabla \vp + \rho \div \psi
= f,
\label{nowe-eq-pv1}
\\
\rho \{ \psi_t + (u \cdot \nabla) \psi \}
- L \psi + R\theta\nabla \vp+R\rho\nabla\zeta
= g ,
\label{nowe-eq-pv2}
\\
c_v\rho\big(\zeta_t+{u}\cdot\nabla\zeta\big)+R\rho\theta\textrm{div}\psi
-\kappa\Delta\zeta=2\mu|\mathcal{D}(\psi)|^2+\lambda(\textrm{div}\psi)^2+
h.
\label{nowe-eq-pv3}
\end{gather}
The boundary and initial conditions for $(\vp,\psi,\zeta)$ 
follow from (\ref{ini1})-(\ref{bc3}), and (\ref{bb1}) as
\begin{gather}
\psi(t,M(x'),x') = 0,~~\zeta(t,M(x'),x') = 0,
\label{nowe-pbc}
\\
(\vp,\psi,\zeta)(0,x)
=(\vp_0, \psi_0, \zeta_0)(x)
:= (\rho_0, u_0,\theta_0)(x) - (\rho^s, u^s, \theta^s)(x).
\label{nowe-pic}
\end{gather}
\end{subequations}

Here $L \psi$, $f$,  $g$, and $h$ are defined by
\begin{align*}
L \psi &:= \mu \Delta \psi + (\mu + \lambda) \nabla \div \psi,
\\
f &:= - \psi \cdot\nabla \rho^s
-\vp \div u^s    ,
\\
g&:=- \rho \psi \cdot \nabla u^s - \vp u^s \cdot \nabla u^s 
-R \zeta\nabla {\rho}^s-R\vp\nabla\theta^s,
\\
h&:= - c_v\rho (\psi \cdot \nabla) \theta^s - c_v\vp (u^s\cdot \nabla) \theta^s- R\rho^s\zeta\div u^s
\\
&\qquad-R\vp\theta^s\div u^s
+4\mu\mathcal{D}(\psi):\mathcal{D} (u^s) +2\lambda\div\psi\div u^s,
\end{align*}

In order to prove Theorem \ref{th3}, it suffices to show Proposition \ref{nowe-apriori1} below.
Indeed, the global solvability follows from the continuation argument used in \cite{m-n83}.
Furthermore, the decay property also can be obtained 
in much the same way as in Section 5 of \cite{kg06}.

\begin{proposition}\label{nowe-apriori1}
Let \eqref{outf} and \eqref{super1}  hold.
Suppose that $\Phi \in X_3 (0,T)$
be a solution to the initial--boundary value problem \eqref{nowe-eq-pv}
for some positive constant $T$.
Then there exists a positive constant 
$\ep_0=\ep_0(\Omega)$ depending on $\Omega$ such that if $N (T) + \dels \le \ep_0$, 
then the following estimate holds:
\begin{equation}\label{nowe-apes1}
E_{3,0}(t)+\int_0^t D_{3,0}(\tau)\,d\tau
\leq C_0\hs{3}{\Phi_0}^2
\quad \text{for $t \in [0,T]$} .
\end{equation}
Here $C_0(\Omega)$ is a positive constant 
depending on $\Omega$ but independent of $t$.
\end{proposition}

We remark that the essential difference between the problems \eqref{eq-pv} and \eqref{nowe-eq-pv} 
is whether the inhomogeneous terms $F$ and $G$ appear or not.
Therefore the proof of Proposition \ref{nowe-apriori1} is very similar 
to that of Proposition \ref{apriori1}.
In the remainder of this section, we sketch the proof of Proposition \ref{nowe-apriori1},
which is given by making use of $ \rho  =\rt +  \varphi^s  +  \varphi$, $u  = \ut + \psi^s  +  \psi  + U $, 
$\theta  = \tt + \zeta^s  +  \zeta  + \Theta $, 
and Theorem \ref{th4}.

\subsection{$L^2$ estimate} \label{ss-nowe-L2}

This subsection is devoted to derive the 
estimate of  the perturbation $(\vp,\psi,\zeta)$ in $L^2(\Omega)$.
To do this, we introduce an energy form $\cale$ as in Subsection \ref{ss-L2}
\[
\rho\cale := R\rho\theta^s\eta \left(\frac{\rho^s}{\rho}\right) + \frac{\rho}{2} |\psi|^2
+ c_v\rho\theta^s\eta \left(\frac{\theta}{\theta^s}\right),
\quad
\eta(r) := r - \ln r-1.
\]
Under the smallness assumption on $N_\beta (T)$, 
we have $\li{\Phi(t)} \ll 1$ by Sobolev's inequality \eqref{sobolev2}.
Hence, the energy form $\cale$ is equivalent to the 
square of the perturbation $(\vp,\psi, \zeta)$:
\begin{equation}
c (\vp^2 + |\psi|^2+\zeta^2)
\le
\cale
\le
C (\vp^2 + |\psi|^2+\zeta^2).
\label{nowe-sqr}
\end{equation}

Moreover, using $N_\beta(T)+\dels \ll 1$, we can derive the following uniform bounds of solutions:
\begin{equation}
0 < c \le (\rho,\theta)(t,x) \le C,
\quad
|u(t,x)| \le C .
\label{nowe-bdd}
\end{equation}
We obtain the energy inequality in $L^2$ framework, as stated in the following lemma analogous to Lemma \ref{lm1}:

\begin{lemma}
\label{nowe-lm1}
Under the same conditions as in Proposition \ref{nowe-apriori1}, 
it holds that
\begin{equation}
  \|\Phi(t)\|^2
+ \int_0^t  D_{0,0}(\tau) \, d \tau
\lesssim  \|\Phi_0\|^2
+  \dels \int_0^t   \lt{\nabla \vp(\tau)}^2 \, d \tau
\label{nowe-ea0}
\end{equation}
for $t \in [0,T]$.
\end{lemma}

\begin{proof}
By a computation similar to the derivation of \eqref{ea1}, we see that
the energy form $\cale$ satisfies
\begin{equation}
(\rho \cale)_t
-\div (G_1 + B_1)
+ \mu |\nabla \psi|^2
+ (\mu + \lambda) (\div \psi)^2+\frac{\kappa}{\theta}|\nabla \zeta|^2
=
R_{1},
\label{sea1}
\end{equation}
where 
\begin{align*}
G_1
&:=
-\rho u \cale
-R\theta^s \vp\psi-R \rho \zeta\psi,
\nonumber
\\
B_1
&:=
\mu \nabla \psi \cdot \psi
+ (\mu + \lambda) \psi \div \psi
+\kappa\frac{\zeta}{\theta}\nabla\zeta, 
\nonumber
\\
R_{1}
&:=\!\!
\frac{R\theta^{s}\vp}{\rho}f+\psi \cdot g+\frac{\zeta}{\theta}h
+R\vp \psi \cdot \nabla\theta^{s} + R\zeta \psi \cdot \nabla \rho^{s}
\\
& \quad + R\rho\eta\left(\frac{\rho^{s}}{\rho}\right)u\cdot \nabla\theta^{s}
-  R\theta^{s}\frac{\vp^2}{\rho\rho^{s}}u\cdot \nabla\rho^{s}
 +c_v\rho\eta\left(\frac{\theta}{\theta^{s}}\right)u\cdot \nabla\theta^{s}
 \\
& \quad  
+c_v\rho\frac{\theta^{s}}{\theta} \zeta^{2} u\cdot \nabla \frac{1}{\theta^{s}}
+\kappa\frac{\zeta}{\theta^2}\nabla\theta\cdot\nabla\zeta
+\frac{\zeta}{\theta}(2\mu|\mathcal{D}(\psi)|^2+\lambda(\textrm{div}\psi)^2).
\end{align*}
For more details of the derivation, see Appendix \ref{Appenx0}.

We integrate \eqref{sea1} over $\Omega$. The second term on the left hand side is estimated from below by
using the divergence theorem, (\ref{nowe-sqr}), (\ref{nowe-bdd}), the boundary conditions \eqref{outf} and \eqref{nowe-pbc}
as
\begin{equation}
- \int_{\Omega}  \div \bigl\{ G_1 + B_1 \bigr\} \, dx  
=  \int_{\pd \Omega}
( \rho  \cale)(u_b\cdot n)  \, d \sigma \,
  \gtrsim
 \|\vp|_{\partial \Omega} \|^2_{L^2_{x'}}. 
\label{nowe-ea3}
\end{equation}
Furthermore, we claim that  the integral of $R_{1}$ is estimated as
\begin{equation} \label{nowe-ea5}
\left|\int_\Omega R_{1} dx \right| \lesssim ( N(T)+\delta) ( \| \nabla \Phi \|^2 + \|\vp|_{\partial \Omega} \|^2_{L^2_{x'}}) . 
\end{equation}
To verify the claim, we estimate only a typical term $\rho (\psi \cdot \nabla) (\psi^s+ \tilde{u}+U) \cdot \psi$ being in $\psi \cdot g$ in $R_{1}$, since the other terms can be treated similarly.
We decompose  
\begin{gather*}
\int_\Omega \rho (\psi \cdot \nabla) (\psi^s+ \tilde{u}+U)  \cdot \psi  dx =
\int_\Omega \rho (\psi \cdot \nabla) \psi^s \cdot \psi  dx 
+ \int_\Omega \rho (\psi \cdot \nabla) (\tilde{u}+U)  \cdot \psi  dx.
\end{gather*}
The last term on the right hand side can be handled in a very similar way as the one for the integral of $R_{11}$ in Lemma \ref{lm1}.
We estimate another term by integrating by parts with the boundary condition \eqref{nowe-pbc} as
\begin{equation*}
\begin{split}
\left| \int_\Omega  \rho (\psi \cdot \nabla) \psi^s \cdot \psi  dx \right| 
& =  \Big|   - \int_\Omega (\nabla \rho \cdot \psi ) \psi^s \cdot \psi dx - \int_\Omega (\rho \div \psi ) \psi^s \cdot \psi  dx - \int_\Omega  \rho \psi^s \cdot (\nabla \psi \cdot  \psi)  dx   \Big|   \\
& \lesssim   \delta\| e^{-\alpha x_1/2}\psi \|^2 + \| \nabla \varphi^{s} \|_{L^{3}} \|\psi^s \|_{L^3} \| \psi \|_{L^6}^{2} + \| \nabla \Phi \| \|\psi^s \|_{L^3} \| \psi \|_{L^6} \\
& \lesssim  ( N(T)+\delta)  \| \nabla \Phi \|^2,
\end{split}
\end{equation*}
where we have used \eqref{stdc1}, Theorem \ref{th4}, Hardy's inequality \eqref{hardy}, and Sobolev's inequality \eqref{sobolev1} in deriving the above inequalities.

We integrate (\ref{sea1}) over $(0,t) \times \Omega$ and substitute the estimates (\ref{nowe-ea3}) and (\ref{nowe-ea5}) into the resulting equality. Then we let $\ep$ and $N (T)+\dels$ be suitably small.
Furthermore, using \eqref{stdc1}, \eqref{ExBdry0}, \eqref{nowe-eq-pv1}, Theorem \ref{th4}, \eqref{hardy}, and \eqref{sobolev1}, we obtain
\begin{equation*}
\begin{split}
\left\| \frac{d}{dt} \vp  \right\|^2 
 = \| \rho \div \psi + \vp \div u^s + \nabla \rho^s \cdot \psi \|^2  
\lesssim  \|\nabla \psi\|^2 + \delta\|\nabla \vp\|^2 + \delta \|\vp|_{\partial \Omega} \|^2_{L^2_{x'}}  .
\end{split}
\end{equation*}
These computations yield the desired inequality.
\end{proof}

\subsection{Time-derivative estimates} \label{ss-nowe-time-deriv}

In this section we derive time-derivative estimates.
To this end, by applying the differential operator $\partial_t^k$ for $k=0,1$ to 
(\ref{nowe-eq-pv1}) and (\ref{nowe-eq-pv2}), we have two equations:

\begin{gather}
\partial_t^k \vp_t
+ u \cdot \nabla \partial_t^k \vp
+ \rho \div \partial_t^k \psi
= f_{0,k},
\label{nowe-ec1}
\\
\rho \{
\partial_t^k \psi_t
+ (u \cdot \nabla) \partial_t^k \psi
\}
- L(\partial_t^k \psi)
+ R\theta\nabla \partial_t^k \vp+R\rho\nabla \partial_t^k\zeta
= g_{0,k},
\label{nowe-ec2}\\
c_v\rho\big(\partial_t^k\zeta_t+{u}\cdot\nabla\partial_t^k\zeta\big)+R\rho\theta\textrm{div}\partial_t^k\psi
-\kappa\Delta\partial_t^k\zeta=h_{0,k},
\label{nowe-ec3}
\end{gather}
where
\begin{align*}
f_{0,k}
&:=
\partial_t^k f 
- [\partial_t^k, u] \nabla \vp
- [\partial_t^k, \rho] \div \psi,
\nonumber
\\
g_{0,k}
&:= \partial_t^k g 
- [\partial_t^k, \rho] \psi_t
- [\partial_t^k, \rho u] \nabla \psi
- [\partial_t^k, R\theta] \nabla \vp
- [\partial_t^k, R\rho] \nabla \zeta,
\nonumber
\\
h_{0,k}
&:= \partial_t^k \left(2\mu|\mathcal{D}(\psi)|^2+\lambda(\textrm{div}\psi)^2+h\right)
- [\partial_t^k, c_v\rho] \zeta_t
- [\partial_t^k, c_v\rho u] \nabla \zeta
- [\partial_t^k, R\rho\theta] \div \psi.
\end{align*}
Here $[T,u]v := T(uv) - u T v$ is a commutator.


We note that the equations \eqref{nowe-ec1}--\eqref{nowe-ec3} are parallel to \eqref{ec1}--\eqref{ec3}, 
and the essential difference is whether or not there are the inhomogeneous terms $F$, $G$, and $H$.
Firstly, carrying out estimates as in the proof of Lemma \ref{lm2} with $\sigma=0$,
we have the following lemma on $\pd_t \Phi$.

\begin{lemma}
\label{nowe-lm2}
Under the same conditions as in Proposition \ref{nowe-apriori1}, 
it holds that
\begin{equation} \label{nowe-ec0}
 \|\pd_t \Phi(t)\|^2
+ \int_0^t  
\|(\pd_t \nabla \psi, \pd_t \nabla \zeta)(\tau)\|^2 \, d \tau
\lesssim  \|\Phi_0\|_{H^3}^2
+ (  N (T)+\dels) \int_0^t  D_{3,0}(\tau) \, d \tau   
\end{equation}
for $t \in [0,T]$.
\end{lemma}

Secondly, we estimate $\pd_t^k \nabla\psi$ for $k=0,1$ in the same way as in the proof of Lemma \ref{lm3}.


\begin{lemma}
\label{nowe-lm3}
Under the same conditions as in Proposition \ref{nowe-apriori1}, 
it holds that
\begin{multline}
 \lt{(\pd_t^k\nabla\psi,\pd_t^k\nabla\zeta)(t)}^2
+ \int_0^t  \lt{ (\pd_t^{k+1} \psi, \pd_t^{k+1}\zeta)(\tau)}^2 \, d \tau 
\\
\lesssim 
 \hs{3}{\Phi_0}^2
+ \ep \calh_{k} (t)
+ \ep^{-1} \calp_{k} (t)
+ ( N (t)+\dels ) \int_0^t  D_{3,0}(\tau)^2 \, d \tau
\label{nowe-ed0}
\end{multline}
for $t \in [0,T], \ep\in(0,1),$ and $k=0,1$.
Here, $\calh_{k}(t)$ and $\calp_{k}(t)$ are defined by
\begin{align*}
\calh_{k} (t)
& :=
  \lt{\pd_t^k \nabla \vp(t)}^2
+ \int_0^t   \lt{\pd_t^k \nabla \vp(\tau)}^2 \, d \tau,
\\
\calp_{k} (t)
& :=
 \lt{ (\pd_t^k\psi,\pd_t^k\zeta)(t)}^2
+ \int_0^t   \lt{ (\pd_t^k\nabla\psi,\pd_t^k\nabla\zeta)(\tau)}^2 \, d \tau.
\end{align*}
\end{lemma}

\subsection{Spatial-derivative estimates}  \label{ss-nowe-spatial-deriv}

We do the same change of variables as in Subsection \ref{Sptial-deriv}
and also define
\begin{gather*}
\hat{\vp} (t,y) := \vp(t,\Gamma(y)), \quad \hat{\psi} (t,y) : = \psi (t,\Gamma (y)), 
\quad \hat{\zeta} (t,y) : = \zeta (t,\Gamma (y)), 
\\ \hat{\rho} (t,y) := \rho(t,\Gamma(y)), \quad \hat{u} (t,y) := u(t,\Gamma(y)), \quad \hat{\theta} (t,y) := \theta(t,\Gamma(y)).
\end{gather*}
for $y \in \mathbb{R}_3^+ : = \{ (y_1, y_2, y_3) \in \mathbb{R}^3 : y_1 >0 \}$,
where $\Gamma$ is defined in \eqref{CV1} and $\Gamma(\mathbb R^3_+)=\Omega$. 
Furthermore, $\hat{\nabla}$, $\hat{\cal D}$, $\hat{\div}$, $\hat{\Delta}$, 
$\hat{\frac{d}{dt}}$ denote the same differential operators defined in 
\eqref{CV3}--\eqref{CV6}. 

From \eqref{eq-pv}, we obtain the equation for $(\hat{\vp}, \hat{\psi}, \hat\zeta)$
\begin{subequations}
\label{nowe-eq-pv-hat}
\begin{gather}
 \hat{\vp}_t + \hat{u} \cdot \hat{\nabla} \hat{\vp} + \hat{\rho} \hat{\div} \hat{\psi}
= \hat{f} ,
\label{nowe-eq-pv1-hat}
\\
\hat{\rho} \{ \hat{\psi}_t  + (\hat{u} \cdot \hat{\nabla}) \hat{\psi} \}
- \hat{L} \hat{\psi} + R\hat\theta\hat{\nabla} \hat\vp+R\hat\rho\hat{\nabla}\hat\zeta 
= \hat{g},
\label{nowe-eq-pv2-hat}
\\
c_v\hat\rho\big(\hat\zeta_t+\hat{u}\cdot\hat{\nabla}\hat\zeta\big)+R\hat\rho\hat\theta\hat{\textrm{div}}\hat\psi
-\kappa\hat{\Delta}\hat\zeta=2\mu|\hat{\mathcal{D}}(\hat\psi)|^2+\lambda(\hat{\textrm{div}}\hat\psi)^2+\hat h,
\label{nowe-eq-pv3-hat}
\end{gather}
and the initial and boundary conditions
\begin{gather}
(\hat{\vp},\hat{\psi}, \hat\zeta)(0,y)
=(\hat{\vp}_0, \hat{\psi}_0, \hat\zeta_0)(y)
= (\rho_0, u_0, \theta_0)(\Gamma(y)) - (\rho^s, u^s, \theta^s)(\Gamma(y)),
\label{nowe-pic-hat}
\\
\hat{\psi}(t,0,y') = 0, \quad \hat{\zeta}(t,0,y') = 0.
\label{nowe-pbc-hat}
\end{gather}
\end{subequations}
Here $\hat{L} \hat{\psi}$, $\hat{f}$, $\hat{g}$,  and $\hat{h}$ are defined by
\begin{gather*}
\hat{L} \hat{\psi}(t,y) := \mu \hat{\Delta} \hat{\psi} (t,y) + (\mu + \lambda) \hat{\nabla} \hat{\div} \hat{\psi} (t,y),
\\
\hat{f} (t,y) :=f(t,\Gamma(y)),
\quad
\hat{g} (t,y) := g(t,\Gamma(y)),
\quad
\hat{h} (t,y) :=h(t,\Gamma(y)).
\end{gather*}

Since the problem \eqref{nowe-eq-pv-hat} is parallel to the problem \eqref{eq-pv-hat},
we can derive the estimate on the tangential spatial derivatives
by the same method as in the proof of Lemma \ref{lm2-hat} with $\sigma=0$.

\begin{lemma}
\label{nowe-lm2-hat}
Under the same conditions as in Proposition \ref{nowe-apriori1}, 
it holds,
\begin{align*}
{}& \|\nabla^l_{y'} \hat{\Phi}(t)\|_{L^2(\mathbb R^3_+)}^2
+ \int_0^t 
\left(\| (\nabla\nabla^l_{y'}  \hat{\psi}, \nabla\nabla^l_{y'}  \hat{\zeta})(\tau)\|_{L^2(\mathbb R^3_+)}^2
+ \left\| \nabla^l_{y'}  \hat{\frac{d}{dt}} \hat{\vp}  (\tau) \right\|_{L^2(\mathbb R^3_+)}^2 
 \right) \, d \tau 
\notag\\
&\lesssim  \|\Phi_0\|_{H^3}^2
+ \epsilon \int^t_0  (\| \nabla \vp(\tau) \|^2_{H^{l-1}} + \|  (\nabla \psi, \nabla \zeta)(\tau) \|^2_{H^{l}} ) d \tau 
+ \epsilon^{-1} \int^t_0  \|  (\nabla \psi, \nabla \zeta) \|^2_{H^{l-1}} d \tau 
\notag\\
& \quad + (N (T)+\dels ) \int_0^t  D_{3,0}(\tau) \, d \tau
\end{align*}
for $t \in [0,T]$, $\epsilon \in (0,1)$,  and $l = 1, 2, 3$.
\end{lemma}

We also derive the estimates on the normal spatial derivative
by following the proof of Lemma \ref{lm4-hat}. 

\begin{lemma}
\label{nowe-lm4-hat}
Suppose that the same conditions as in Proposition \ref{nowe-apriori1} hold. Define the index $\bm{a} = (a_1, a_2, a_3)$ with $a_1, a_2, a_3 \geq 0$ and $|\bm{a}| := a_1 + a_2 + a_3$. Let $\partial^{\bm{a}} := \partial^{a_1}_{y_1} \partial^{a_2}_{y_2} \partial^{a_3}_{y_3}$. Then it holds that 
\begin{align*}
&  \| \partial^{\bm{a}} \partial_{y_1} \hat{\vp} (t) \|_{L^2(\mathbb R^3_+)}^2 
+ \int^t_0 \left( \|\partial^{\bm{a}} \mathfrak{D} \hat{\vp} (\tau) \|_{L^2(\mathbb R^3_+)}^2 
+ \left\| \partial^{\bm{a}} \partial_{y_1} \hat{\frac{d}{dt}} \hat{\vp}(\tau) \right\|_{L^2(\mathbb R^3_+)}^2  \right) \, d\tau 
\notag\\
&  \lesssim \|\vp_0\|_{H^3}^2 
+ \int_0^t  \left(
\left\|\partial^{\bm{a}} \nabla_{y'} \hat{\frac{d}{dt}} \hat{\vp}(\tau) \right\|_{L^2(\mathbb R^3_+)}^2
+ \|\partial^{\bm{a}} \nabla_y \nabla_{y'} \hat{\psi}(\tau) \|_{L^2(\mathbb R^3_+)}^2 \right)\, d \tau
\notag\\
&\quad + \int_0^t  \left(|{\bm{a}}|\|\nabla {\vp} (\tau)\|^2_{H^{|{\bm{a}}| -1}}
+|{\bm{a}}| \left\|\nabla {\frac{d}{dt}} {\vp}(\tau) \right\|_{H^{|{\bm{a}}|-1}}^2
+\|{\psi}_t (\tau)\|^2_{H^{|{\bm{a}}|}} 
+\| {(\nabla \psi, \nabla\zeta)} (\tau)\|^2_{H^{|{\bm{a}}|}} \right)\, d \tau
\notag\\
& \quad 
+ (N (T)+\dels ) \int_0^t  D_{3, 0}(\tau) \, d \tau  
\end{align*}
for $t \in [0, T]$ and $0\leq |\bm{a}| \leq 2$.
\end{lemma}

\subsection{Cattabriga estimates}  \label{ss-nowe-CattabrigaEst}

As in Subsection \ref{ss-CattabrigaEst}, we apply the Cattabriga estimate in Lemma \ref{CattabrigaEst}, which has crucial dependence on $\Omega$. 
Firstly we have Lemma \ref{nowe-lm1-C}, which is parallel to Lemma \ref{lm1-C}. This follows from applying the Cattabriga estimate \eqref{Cattabriga} to 
a boundary value problem of the Stokes equation:
\begin{align}
 \rho_+\div\psi = V , \quad
 - \mu \Delta \psi + R\theta_+\nabla \vp 
 = W,\quad
 \psi|_{\partial \Omega} =0 , \quad \lim_{|x| \rightarrow \infty} |\psi| =0,\label{nowe-ef7-C} 
 \end{align}
where this problem is just a rewrite of \eqref{nowe-eq-pv}, and $V$ and $W$ are defined as
\begin{align*}
V : =& f  - \frac{d}{dt}\vp - (\rho-\rho_+)\div \psi , \\
W : =& - \rho \{ \psi_t + ( u \cdot \nabla ) \psi \} + (\mu + \lambda) \rho_+^{-1} \nabla V + g  -R (\theta-\theta_+) \nabla \vp- R\rho\nabla \zeta.
\end{align*}

\begin{lemma} \label{nowe-lm1-C}
Under the same assumption as in Proposition \ref{nowe-apriori1},
it holds, for $k=0, 1, 2$, 
\begin{equation}
\| \nabla^{k+2} \psi \|^2 + \| \nabla^{k+1} \vp \|^2  \lesssim_\Omega  \| \psi_t \|^2_{H^{k}} 
+ \|  (\nabla \psi, \nabla \zeta) \|^2_{H^{k}} 
+ \left\| \frac{d}{dt} \vp \right\|^2_{H^{k+1}} 
 + (N (T) + \delta ) D_{3, 0} . 
\label{nowe-ef0-C}
\end{equation}
\end{lemma}

As in Subsection \ref{ss-CattabrigaEst}, 
we show similar estimates for $(\hat{\vp},\hat{\psi})(t,y)$, where $y \in \mathbb R^3_+$.
We use the same differential operator $\check{\partial}_{y_j}$ as defined in \eqref{checkOp}. 
Furthermore, $\check{\nabla}^l_{y'}$ means the totality of all $l$-times tangential derivatives 
$\check{\partial}_{y_j}$ only for $j=2,3$.
Applying $\check{\nabla}^l_{y'}$ to \eqref{nowe-ef7-C}, 
we obtain a boundary value problem of the Stokes equation: 
\begin{equation}\label{nowe-checkEq}
\rho_+\div \check{\nabla}^l_{y'} \psi = \check{V}, \ \
- \mu \Delta \check{\nabla}^l_{y'} \psi +R\theta_+\nabla \check{\nabla}^l_{y'} \vp = \check{W}, \ \  
\check{\nabla}^l_{y'} \psi|_{\partial \Omega} =0 , \ \ \lim_{|x| \rightarrow \infty} | \check{\nabla}^l_{y'} \psi| =0,
\end{equation}
where
\begin{align*}
\check{V} : = 
& \check{\nabla}^l_{y'} f - \check{\nabla}^l_{y'} \frac{d}{dt}\vp - \check{\nabla}^l_{y'} \Big((\rho-\rho_+)\div \psi \Big) - \rho_{+}[\check{\nabla}^l_{y'},\div] \psi, \\
\check{W} : = 
&- \check{\nabla}^l_{y'}  \Big(\rho \{ \psi_t + ( u \cdot \nabla ) \psi \} \Big) + (\mu + \lambda) \check{\nabla}^l_{y'}  \nabla \div \psi + \check{\nabla}^l_{y'} g \\
& - \check{\nabla}^l_{y'} \Big( R (\theta-\theta_+) \nabla \vp + R\rho\nabla \zeta\Big) 
+ \mu [\check{\nabla}^l_{y'},\Delta]\psi 
- R \theta_+ [\check{\nabla}^l_{y'}, \nabla]\vp.
\end{align*}
Then we have the following estimate parallel to Lemma \ref{lm2-C}. This can be shown by an application of the Cattabriga estimate \eqref{Cattabriga} to the problem \eqref{nowe-checkEq}.
    
\begin{lemma} \label{nowe-lm2-C}
Under the same assumption as in Proposition \ref{nowe-apriori1},
it holds, for $k=0, 1$, $k+l =  1,  2$, 
\begin{align}
&\| \nabla^{k+2}  \nabla^l_{y'} \hat{\psi} \|_{L^2(\mathbb R^3_+)}^2 + \|\nabla^{k+1} \nabla^l_{y'} \hat{\vp} \|_{L^2(\mathbb R^3_+)}^2 
\notag\\
& \lesssim_\Omega 
\left\|{\nabla}^l_{y'} \hat{\frac{d}{dt}} \hat{\vp} \right\|^2_{H^{k+1} (\mathbb R^3_+)} 
+ \| \psi_t \|^2_{H^{k+l}} 
+ \|  (\nabla \psi, \nabla \zeta) \|^2_{H^{k+l}} 
+ \| \nabla \vp \|^2_{H^{k+l-1}}
 + (N (T) + \delta ) D_{3,0} . 
\label{nowe-eg0-C}
\end{align}
\end{lemma}
%

\subsection{Elliptic estimates} \label{ss-nowe-EllipticEst}
Similarly as in Subsection \ref{EllipticEst}, we apply the elliptic estimate (Lemmas \ref{ellipticEst} and \ref{ellipticEst2}) to rewrite some terms for the time-derivatives into terms for the spatial-derivatives. The next lemma follows from regarding \eqref{nowe-eq-pv} and its time derivative as elliptic boundary value problems, 
and applying Lemmas \ref{ellipticEst} and \ref{ellipticEst2}.

\begin{lemma} \label{nowe-lm1-E}
Under the same assumption as in Proposition \ref{nowe-apriori1},
it holds that  
\begin{align}
\| \nabla^{k+2} \psi \| & \lesssim  \| \psi_t \|_{H^{k}} + E_{k+1,0}, \quad k=0,1,
\label{5eh0-E} \\
\| \nabla^{k+2} \zeta \| & \lesssim  \| \zeta_t \|_{H^{k}} + E_{k+1,0}, \quad k=0,1,
\label{5eh0-Ez} \\
\| \nabla^{k+2} \zeta \| & \lesssim  \| \zeta_t \|_{H^{k}} + D_{k,0}, \quad k=0,1,2,
\label{5eh0-Ezz} \\
\| \nabla^{2}   \psi_t \| & \lesssim  \|\psi_{tt} \| +D_{2,0},
\label{5ej0-E} \\
\| \nabla^{2}  \zeta_t \| & \lesssim  \|\zeta_{tt} \| +D_{2,0}.
\label{5ej0-Ez}
\end{align}
\end{lemma} 

\subsection{Completion of the a priori estimate} 

In this subsection, we complete the derivation of the a priori estimate.
First, by following the proofs of Lemmas \ref{lm1-comp} and \ref{lm3-comp}
with the aid of Lemmas \ref{nowe-lm2-hat}--\ref{nowe-lm1-E}, we have the next lemma.

\begin{lemma} \label{nowe-lm3-comp}
Under the same assumption as in Proposition \ref{nowe-apriori1},
it holds, for $p=0,  1,  2$, 
\begin{align}
&   \| \nabla^{p+1}  \vp (t) \|^2 + \int^t_0    \|( \nabla^{p+1} \vp, \nabla^{p+2} \psi, \nabla^{p+2} \zeta )(\tau)\|^2   \, d \tau  + \int^t_0  \left\| \nabla^{p+1} \frac{d}{dt} \vp (\tau) \right\|^2 \, d \tau   
\notag\\
& \lesssim_\Omega \| \Phi_0\|^2_{H^3} +   \| \vp (t) \|_{H^p}^2 + \int^t_0  \|(\psi_t,\zeta_t)(\tau) \|^2_{H^p} \, d \tau  
+ \int^t_0   D_{p, 0} (\tau) \, d \tau   
\notag\\
& \qquad + (N (T) + \delta ) \int^t_0  D_{3, 0} (\tau) \, d \tau . 
\label{nowe-em0-comp}
\end{align}
\end{lemma}

Now we can complete the a priori estimate.

\begin{proof}[Proof of Proposition \ref{nowe-apriori1}]
The a priori estimate can be shown in the same way as in the proof of Proposition \ref{apriori1}
in Subsection \ref{ss-comp-apriori}.
Indeed the proof is complete just by 
replacing \eqref{ea0}, \eqref{ec0}, \eqref{ed0}, \eqref{eh0-E}--\eqref{ej0-Ez}, and \eqref{em0-comp}
by \eqref{nowe-ea0}, \eqref{nowe-ec0}, \eqref{nowe-ed0}, \eqref{5eh0-E}--\eqref{5ej0-Ez}, and \eqref{nowe-em0-comp}, respectively.
\end{proof}

We discuss briefly the proof of Corollary \ref{cor}.

\begin{proof}[Proof of Corollary \ref{cor}]
If $\|M\|_{H^{9}(\mathbb R^2)} \leq \eta_0$ holds for $\eta_0$ being in Lemma \ref{CattabrigaEst},
we can replace $\lesssim_{\Omega}$ by $\lesssim$ 
in the inequalities \eqref{nowe-ef0-C}, \eqref{nowe-eg0-C}, and \eqref{nowe-em0-comp}.
Then following the proof of Proposition \ref{nowe-apriori1} with these improved inequalities, 
we conclude Corollary \ref{cor}.
\end{proof}

\begin{appendix}

\section{The equations of the energy forms}\label{Appenx0}

In this section, we derive the equations of the energy forms i.e. \eqref{ea1} and \eqref{sea1}. 
Let us first treat \eqref{ea1}. A direct computation yields
\begin{gather}\label{eform10}
(\rho \cale)_t + \div (\rho u \cale)  = \cale \{\rho_t + \div (\rho u) \}  + I_1 + I_2 +I_3, 
\end{gather}
where
\begin{align*}
I_{1} &= \rho \psi \cdot \psi_t + \rho u \cdot (\nabla \psi) \psi , \quad
\\
I_{2} &=  \frac{R\tt\vp}{\rho} (\vp_t + u \cdot \nabla \vp)
+R\rho\eta(\frac{\rt}{\rho}) u \cdot \nabla \tt
-\frac{R\tt\vp^2}{\rho\rt}  u \cdot \nabla \rt, 
\\
I_{3} &= c_v\rho\frac{\zeta}{\theta} (\zeta_t + u \cdot \nabla \zeta)
+c_v\rho\eta(\frac{\theta}{\tt+\Theta}) u \cdot \nabla ({\tt+\Theta})
+c_v\rho({\tt+\Theta})\frac{\zeta}{\theta}  u \cdot \nabla \frac{1}{\tt+\Theta}.
\end{align*}
Note that the first term on the right hand side is zero owing to $\eqref{ns1}_1$.
Let us rewrite terms $I_{1}$, $I_{2}$, and $I_{3}$ as follows.
Use \eqref{eq-pv2}  to obtain
\begin{equation*}
\begin{split}
I_1 
& = L \psi \cdot \psi - R\theta\psi\cdot \nabla \vp  - R\rho \psi \cdot \nabla \zeta+ \psi\cdot(g+G). \\
\end{split}
\end{equation*}
Notice that
\begin{gather*}
L \psi \cdot \psi = \div \left\{\mu \nabla \psi \cdot \psi+ (\mu + \lambda) \psi \div \psi \right\}
- \mu |\nabla \psi|^{2} - (\mu + \lambda) (\div \psi)^{2}.
\end{gather*}
Using \eqref{eq-pv1}, we arrive at
\begin{equation*}
\begin{split}
I_2 
&=\frac{R\tt\vp}{\rho} (-\rho\div \psi+f+F)
+R\rho\eta(\frac{\rt}{\rho}) u \cdot \nabla \tt
-\frac{R\tt\vp^2}{\rho\rt}  u \cdot \nabla \rt.
\\
\end{split}
\end{equation*}
Furthermore, by \eqref{eq-pv3}, it is seen that
\begin{equation*}
\begin{split}
I_3 & =\frac{\zeta}{\theta} \Big(-R\rho\theta\textrm{div}\psi
+\kappa\Delta\zeta+2\mu|\mathcal{D}(\psi)|^2+\lambda(\textrm{div}\psi)^2+h+H\Big)
\\
&\quad +c_v\rho\eta(\frac{\theta}{\tt+\Theta}) u \cdot \nabla ({\tt+\Theta})
+c_v\rho({\tt+\Theta})\frac{\zeta}{\theta}  u \cdot \nabla \frac{1}{\tt+\Theta}. \\
\end{split}
\end{equation*}
Notice that
\begin{gather*}
\frac{\kappa}{\theta} (\Delta\zeta) \zeta  = \div \left(  \kappa\frac{\zeta}{\theta}\nabla\zeta \right)
- \frac{\kappa}{\theta} |\nabla \zeta |^{2} + \kappa\frac{\zeta}{\theta^2}\nabla\theta\cdot\nabla\zeta.
\end{gather*}
We use the following terms from $I_1, I_2$ and $I_3$ to get the cancellation
\begin{align*}
&- R\theta\psi\cdot \nabla \vp  - R\rho \psi \cdot \nabla \zeta - R\tt\vp\div\psi  - R\rho \zeta\div\psi
\\
&=  - R(\zeta+\Theta)\psi\cdot \nabla \vp - R\tt\div(\psi \vp ) - R\rho \div(\psi\zeta )
\\
&=  - R\Theta\psi\cdot \nabla \vp - R\div(\tt\psi \vp ) +R\vp \psi \cdot \nabla \tt 
- R \div(\rho\psi\zeta )+ R\zeta\psi\cdot \nabla \rt .
\end{align*}
Plugging the results above into \eqref{eform10},  we conclude \eqref{ea1}.
One can also have \eqref{sea1} just by
replacing $(\rt,\ut,\tt,U,\Theta,\nabla \tilde{M})$ by $(\rho^{s},u^{s},\theta^s,0,0,0)$ in the computation above. 

\section{Basic inequalities and estimates}\label{BasciIneq}

This section provides some basic inequalities and estimates that are frequently used throughout the paper. 
The following lemmas cover the case $M\equiv 0$, that is, $\Omega=\mathbb R^3_+$. 
For the proofs of Lemmas \ref{HardyEst}--\ref{ellipticEst}, see \cite[Appendix A]{SZ1}.

\begin{lemma}\textbf{(Hardy's Inequality)}\label{HardyEst}
Let $\alpha>0$. For $f \in H^1(\Omega)$, it holds that
\begin{equation}
\int_{\Omega} e^{-\alpha x_1}|f(x)|^2 \, d x
\lesssim
\lt{\nabla f}^2
+ \|{f(M(\cdot),\cdot)}\|_{L^2(\mathbb R^2)}^2.
\label{hardy}
\end{equation}
\end{lemma}

\begin{lemma}\textbf{(Sobolev's Inequalities)}
For $f \in H^1(\Omega)$ and $g \in H^2(\Omega)$, it holds that
\begin{align}
\|f\|_{L^p(\Omega)} &\lesssim \|f\|_{H^1(\Omega)}, \quad 2\leq p <6,
\label{sobolev0}\\
\|f\|_{L^6(\Omega)} &\lesssim \|\nabla f\|_{L^2(\Omega)},
\label{sobolev1}\\
\|g\|_{L^\infty(\Omega)} &\lesssim \|g\|_{H^2(\Omega)}.
\label{sobolev2}
\end{align}
\end{lemma}

\begin{lemma}
\textbf{(Gagliardo-Nirenberg Inequality)}
Let $k=2,3,4 \cdots$.
For $f \in H^{k} (\Omega)$, there holds that
\begin{equation} \label{GN1}
\|f\|_{H^{k-1}(\Omega)} \lesssim \|f\|_{H^{k}(\Omega)}^{1-1/k}\|f\|_{L^2(\Omega)}^{1/k}.
\end{equation}
\end{lemma}

\begin{lemma} \label{CommEst}
\textbf{(Commutator Estimate)}
Let $k=0,1,2,\cdots$.
For $f, g, \nabla f \in H^{k} (\Omega) \cap L^\infty (\Omega)$, we have
\begin{equation*} 
\|[\nabla^{k+1},f]g\|_{L^2(\Omega)}\lesssim  \|\nabla f\|_{L^\infty(\Omega)}\|g\|_{H^{k}(\Omega)}+\|\nabla f\|_{H^{k}(\Omega)}\|g\|_{L^\infty(\Omega)} , 
\end{equation*}
\begin{equation*} 
\|[\nabla^{k+1},\nabla M]g\|_{L^2(\Omega)} \lesssim \|g\|_{H^{k}(\Omega)}.
\end{equation*} 
\end{lemma}

\begin{lemma} \label{CattabrigaEst}
\textbf{(Cattabriga Estimate)}
Consider the following Stokes system 
\begin{equation*} 
\bar{\rho} \div u = h , \
- \hat{\mu} \Delta u + \hat{p} \nabla p = g \ \text{on $\Omega$}, \quad
u|_{\partial \Omega} =0, \quad
\ \lim_{|x| \rightarrow \infty}u =0
\end{equation*}
with some constants $\bar{\rho}>0$, $\hat{\mu}>0$, and $\hat{p}>0$. For $k= 0, 1, \cdots, 4$ and 
$(h,g)\in H^{k+1}(\Omega)\times H^{k}(\Omega)$, 
if $(u,p) \in H^{k+2}(\Omega)\times H^{k+1}(\Omega)$ is a solution to the Stokes system, 
then it holds that
\begin{equation}\label{Cattabriga}
\| \nabla^{k+2} u \|_{L^2(\Omega)}^2 + \|\nabla^{k+1} p \|_{L^2(\Omega)}^2 
\leq C_0 (\|h \|^2_{H^{k+1}(\Omega)} + \| g \|^2_{H^k(\Omega)} + \|\nabla u \|_{L^2(\Omega)}^2),
\end{equation}
where $C_0=C_0(\Omega)$ is a positive constant depending on $\Omega$.
Furthermore, there exists a positive constant $\eta_0$ such that 
if $\|M\|_{H^{9}(\mathbb R^2)} \leq \eta_0$, 
then \eqref{Cattabriga} holds with $C_0$ independent of $\Omega$.
\end{lemma}

\begin{lemma} \label{ellipticEst}
\textbf{(Elliptic Estimate)}
Consider the following elliptic system 
\begin{equation*} 
\begin{split}
 - \hat{\mu} \Delta u - (\hat{\mu} + \hat{\lambda})   \nabla \div u = f \ \text{on $\Omega$}, \quad  
 u|_{\partial \Omega} =0 , \quad \lim_{|x| \rightarrow \infty} u =0 
\end{split}
\end{equation*}
with some constants $\hat{\mu} > 0$ and $\hat{\mu} + \hat{\lambda} \geq 0$. For $k= 0,1,2$ and $f \in H^{k}(\Omega)$,
if $u \in H^{k+2}(\Omega)$ is a solution to the elliptic system, then it holds that
\begin{equation*}
\| \nabla^{k+2} u \|_{L^2(\Omega)}  \lesssim \|f \|_{H^{k}(\Omega)} + \|  u \|_{L^2(\Omega)}  .  
\end{equation*}
\end{lemma}

\begin{lemma} \label{ellipticEst2}
\textbf{(Elliptic Estimate II)}
Consider the following elliptic equation
\begin{equation*} 
\begin{split}
 - \hat{\kappa} \Delta u  = f \ \text{on $\Omega$}, \quad  
 u|_{\partial \Omega} =0 , \quad \lim_{|x| \rightarrow \infty} u =0 
\end{split}
\end{equation*}
with some constants $\hat{\kappa} > 0$. 
For $k= 0,1,2$ and $f \in H^{k}(\Omega)$,
if $u \in H^{k+2}(\Omega)$ is a solution to the elliptic equation, then it holds that
\begin{equation}\label{Elliptic2}
\| \nabla^{k+2} u \|_{L^2(\Omega)}  \lesssim \| f \|_{H^{k}(\Omega)} + \|\nabla u\|_{L^{2}(\Omega)} .  
\end{equation}
\end{lemma}
\begin{proof}
Without loss of generality, we can suppose that $\hat{\kappa} =1$. 
We recall the change of valuables \eqref{CV1} and the notation \eqref{CV3}--\eqref{CV6}.
The elliptic equation can be written in $y$-coordinate as  
\begin{equation} \label{au2}
-\hat{\Delta} \hat{u} = \hat f \ \text{on $\mathbb R^{3}_{+}$}, \quad
 \hat{u}|_{\partial \mathbb R^{3}_{+}} =0 , \quad \lim_{|y| \rightarrow \infty} \hat{u} =0 .   
\end{equation} 
Let us consider only the case $k=0$, since the other cases $k=1, 2$ can be shown in a similar way with higher derivatives applied. Applying $\partial_{y_{l}}$ for $l=2,3$ to the first equation in \eqref{au2}, multiplying it by $\partial_{y_{l}} \hat{u} (y)$, and integrating the result over $\mathbb R^{3}_{+}$, we obtain
\begin{align*}
& \sum_{k=1}^3  \int_{\mathbb{R}^{3}_{+}}\left| \left( \sum_{j=1}^3 A_{kj} \partial_{y_j} \right) \partial_{y_{l}}  \hat{u} (y)  \right|^{2} dy 
\\
&= - \int_{\mathbb{R}^{3}_{+}} \hat f(y) \partial_{y_{l}}^{2} \hat{u} (y) dy 
+ \int_{\mathbb{R}^{3}_{+}} ( [\hat{\Delta},  \partial_{y_{l}}]\hat{u} ) \partial_{y_{l}} \hat{u} (y) dy
\\
&\quad -  \sum_{k=1}^3  \int_{\mathbb{R}^{3}_{+}} \left\{ \left( \sum_{j=1}^3 A_{kj} \partial_{y_j} \right) \partial_{y_{l}} \hat{u} (y) \right\} \left\{\left( \sum_{j=1}^3 \partial_{y_j}A_{kj}  \right) \partial_{y_{l}}   \hat{u} (y)  \right\}dy
\\
& \leq \nu \| \nabla^{2}_{y} \hat{u} \|^{2}_{L^{2}(\mathbb R^{3}_{+})} +C \nu^{-1} \|(\hat{f},\nabla_{y} \hat{u})\|^{2}_{L^{2}(\mathbb R^{3}_{+})},
\end{align*}
where we have used Schwarz's inequality in deriving the last inequality, and $\nu$ is a positive constant to be determined later.
Summing up the results for $l=1,2$ gives
\begin{gather*}
\sum_{|\bm{b}|= 2, \ b_1\neq2}  \| \partial^{\bm{b}}_{y}  \hat{u} \|^{2}_{L^{2}(\mathbb R^{3}_{+})} \leq \nu \| \nabla^{2}_{y} \hat{u} \|^{2}_{L^{2}(\mathbb R^{3}_{+})} +C \nu^{-1} \|(f,\nabla_{y} \hat{u})\|^{2}_{L^{2}(\mathbb R^{3}_{+})}.
\end{gather*}
On the other hand, the normal derivative $\partial_{y_{1}y_{1}}  \hat{u} $ can be estimated by using the first equation in \eqref{au2} as
\begin{gather*}
\|\partial_{y_{1}y_{1}}  \hat{u} \|^{2}_{L^{2}(\mathbb R^{3}_{+})} \lesssim \sum_{|\bm{b}|= 2, \ b_1\neq2}  \| \partial^{\bm{b}}_{y}  \hat{u} \|^{2}_{L^{2}(\mathbb R^{3}_{+})} +  \|\hat{f} \|^{2}_{L^{2}(\mathbb R^{3}_{+})}.
\end{gather*}
Using these, letting $\nu\ll 1$, and performing the change of variables $y \rightarrow x$, we conclude \eqref{Elliptic2}.
\end{proof}

\end{appendix}





\providecommand{\bysame}{\leavevmode\hbox to3em{\hrulefill}\thinspace}
\providecommand{\MR}{\relax\ifhmode\unskip\space\fi MR }

\end{document}